\documentclass{amsart}
\usepackage{hyperref}

\usepackage{amsrefs,amsmath,amssymb,stmaryrd}
\usepackage{amsthm}
\usepackage{microtype} 

\usepackage{tikz,bbm}
\usetikzlibrary{matrix,arrows,decorations.pathmorphing, cd}
\usepackage{comment}

\usepackage{mathrsfs}

\numberwithin{equation}{section}

\newcommand{\Hom}       {\operatorname{Hom}}
\newcommand{\uHom}      {\underline{\operatorname{Hom}}}

\newcommand{\spec}      {\operatorname{Spec}}
\newcommand{\Spc}       {\operatorname{Spc}}
\newcommand{\supp}      {\mathrm{supp}}

\newcommand{\Fun}       {\operatorname{Fun}}
\newcommand{\Alg}       {\operatorname{Alg}}
\newcommand{\CAlg}      {\operatorname{CAlg}}
\newcommand{\Tot}       {\operatorname{Tot}}

\newcommand{\Ind}       {\operatorname{Ind}}

\newcommand{\op}        {\operatorname{op}}
\newcommand{\Cpl}       {\operatorname{Cpl}}

\newcommand{\QCoh}      {\operatorname{QCoh}}

\newcommand{\GalCat}    {\mathsf{GalCat}}

\newcommand{\perf}      {\mathrm{perf}}
\newcommand{\dual}      {\mathrm{dual}}
\newcommand{\sep}       {\mathrm{sep}}

\newcommand{\fin}       {\mathrm{fin}}
\newcommand{\cov}       {\mathrm{cov}}
\newcommand{\wcov}      {\mathrm{w.cov}}
\newcommand{\br}        {\mathrm{br}}
\newcommand{\lcf}       {\mathrm{lcf}}
\newcommand{\rk}        {\mathrm{rk}}
\newcommand{\cf}         {\mathrm{cf}}
\newcommand{\cts}       {\mathrm{cts}}
\newcommand{\weak}      {\mathrm{weak}}
\newcommand{\Sp}        {{\mathrm{Sp}}}
\renewcommand{\Pr}      {{\mathsf{Pr}^{L}_{\mathrm{st}}}}
\newcommand{\Cat}       {{\mathsf{Cat}}}

\newcommand{\Set}       {{\mathsf{Set}}}
\newcommand{\FinSet}    {{\mathsf{FinSet}}}
\newcommand{\twoRing}   {2\text{-}\mathsf{Ring}}
\newcommand{\Gpd}       {\mathsf{Gpd}}
\newcommand{\Cov}       {\mathrm{Cov}}

\newcommand{\1}{\mathbbm{1}}

\newcommand{\CC}        {{\mathcal{C}}}

\newcommand{\CV}        {{\mathcal{V}}}

\newcommand{\Q}         {{\mathbb{Q}}}
\newcommand{\Z}         {{\mathbb{Z}}}
\newcommand{\N}         {{\mathbb{N}}}
\newcommand{\E}         {\mathbb{E}}

\newcommand{\C}         {{\mathscr{C}}}

\newcommand{\D}         {{\mathscr{D}}}

\newcommand{\G}         {{\mathscr{G}}}

\newcommand{\p}				{\mathfrak{p}}
\renewcommand{\q}			{\mathfrak{q}}

\newcommand{\colim}  {\operatornamewithlimits{\underset{\longrightarrow}{lim}}}

\renewcommand{\mod}[1]{\mathrm{Mod}_{#1}}

\newcommand{\Perf}[1]{\mathrm{Perf}_{#1}}

\newcommand{\StMod}{\mathrm{StMod}}
\newcommand{\stmod}{\mathrm{stmod}}

\newtheorem{theorem}{Theorem}[section]

\newtheorem{lemma}[theorem]{Lemma}
\newtheorem{proposition}[theorem]{Proposition}
\newtheorem{corollary}[theorem]{Corollary}
\theoremstyle{definition}
\newtheorem{remark}[theorem]{Remark}
\newtheorem{definition}[theorem]{Definition}
\newtheorem{example}[theorem]{Example}
\newtheorem{construction}[theorem]{Construction}
\newtheorem{question}[theorem]{Question}

\theoremstyle{plain}
\newtheorem*{question*}{Question}
\newtheorem*{theorem*}{Theorem}

\newtheorem{notation}[theorem]{Notation}

\newtheorem{thmx}{Theorem}

\usepackage{todonotes}

\newcommand{\ltodo}[1]{\todo[color=blue!20]{#1}}

\title[Separable commutative algebras and Galois theory]{Separable commutative algebras and Galois theory in stable homotopy theories}
\author{Niko Naumann and Luca Pol}
\begin{document}

\maketitle
\vspace{-5mm}
\begin{abstract}
We relate two different proposals to extend the \'etale topology into homotopy theory, namely via the notion of finite cover introduced by Mathew and via the notion of separable commutative algebra introduced by Balmer. We show that finite covers are precisely those separable commutative algebras with underlying dualizable module,  which have a locally constant and finite degree function. We then use Galois theory to classify separable commutative algebras in numerous categories of interest. Examples include the category of modules over a connective $\E_\infty$-ring $R$ which is either connective or even periodic with $\pi_0(R)$ regular Noetherian in which $2$ acts invertibly, the stable module category of a finite group of $p$-rank one and the derived category of a qcqs scheme.
\end{abstract}

\setcounter{tocdepth}{1}
\tableofcontents

\section{Introduction}

Tensor-triangular geometry applies algebro-geometric methods to study triangulated categories. This abstraction allows the  transport of ideas and techniques between areas of mathematics such as algebraic geometry, modular representation theory, equivariant stable homotopy theory and noncommutative geometry. The profound impact on algebraic geometry brought by Grothendieck's work on \'etale cohomology advocates for a generalization of the \'etale topology in tensor-triangular geometry. We compare two proposals for this: the first one is due to Balmer~\cite{Balmer2011} via the notion of separable commutative algebra, and the second one is due to Mathew~\cite{Mathew2016} via the notion of finite cover. The goal of this paper is twofold:
\begin{itemize}
    \item[(1)] Provide a comprehensive comparison between Balmer's and Mathew's work. 
    \item[(2)] Use this comparison to derive new classification results for separable algebras.
\end{itemize}

We now recall the two approaches. 

\subsection*{\'Etaleness via separability} 
%To motivate Balmer's approach let us recall the definition of \'etale morphism in commutative algebra. It is often common to define a  morphism of commutative rings $f\colon R \to S$ to be \emph{\'etale} if it is flat, of finite presentation and the module of relative K\"ahler differentials vanishes. This point-of-view, however, does not admit an evident categorical reformulation. Instead we recall that that under the assumption that $f$ is of finite type, the following are equivalent~\cite{Ford}*{Theorem 8.3.5}:
%\begin{itemize}
%    \item[(a)] the module of relative K\"ahler differentials vanishes;
 %   \item[(b)] $f$ is separable: $S$ is projective as a module over $S \otimes_R S$.
%\end{itemize}  
%Condition (b) admits a purely categorical formulation. Let $A$ be a commutative algebra object in a symmetric monoidal category $\C$. Following Balmer~\cite{Balmer2011} we say that $A$ is \emph{separable} if the multiplication map $\mu \colon A \otimes A \to A$ admits a section as a $(A,A)$-bimodule. Note that we recover condition (b) above by taking $\C$ to be the category of $R$-modules.
The \'etale site of a field $k$ consists of all \'etale $k$-algebras $A$. Most concretely, these can be described as finite
products of finite separable field extensions of $k$, but can also be characterised by the existence of a bimodule
section to the multiplication map of $A$. This latter characterisation generalizes immediately to tensor triangular geometry and
was taken by Balmer as the starting point of his theory. 
Many more examples of separable algebras emerge in this context: Balmer--Dell'Ambrogio--Sanders~\cite{BalmerDellAmbrogioSanders15} showed that restriction to a subgroup of finite index in the equivariant stable homotopy category, equivariant $KK$-category and equivariant derived categories all yield commutative separable algebra objects. In algebraic geometry, Balmer~\cite{Balmer2016-etale} showed that \'etale morphisms of schemes yield examples of separable algebras. As a consistency check, a result of Neeman~\cite{Neeman2018} shows that the only separable commutative algebra objects in the perfect derived category of a Noetherian scheme are those arising from finite \'etale morphisms (and smashing localizations).

%This result strongly suggests that separable dualizable commutative algebra objects in a tensor-triangulated category provides a suitable notion of finite \'etaleness.

\subsection*{\'Etaleness via Galois theory} 
Another way to characterise \'etale $k$-algebras $A$ over a field $k$ is by demanding that $A\otimes_k\bar k$ be
a finite product of copies of $\bar k$, where $\bar k$ denotes a separable closure of $k$. The approach of Mathew
then relies on stipulating that a good replacement for the \'etale topology is the descendable topology. Accordingly,
he defines a commutative algebra to be a {\em finite cover} if it is a finite product of copies of the unit,
locally for the descendable topology.
See \cite{SAG}*{Appendix D.3} for an account of the descendable topology.
He then shows that this notion enjoys excellent properties, both formally and computationally. For example, the finite covers organize into a Galois category in the sense of Grothendieck.
Here again, he establishes a consistency check by showing that the finite covers in the (unbounded) derived category
of a scheme are exactly the finite \'etale covers.

\subsection*{Main results} 
Let $\C$ be a stable homotopy theory and let $\C^\dual\subseteq \C$ be the full subcategory of dualizable objects. 
For any separable commutative algebra object $A\in\C^\dual$, Balmer defined a notion of degree measuring the size of the separable algebra~\cite{Balmer2014}. The degree can be organized into a function
\[
\deg(A) \colon\Spc(\C^\dual) \to \mathbb{Z} \cup\{\infty\}
\]
on the Balmer spectrum of $\C^\dual$. We say that the degree function is \emph{finite} if it takes finite values on all prime ideals and \emph{locally constant} if for each prime ideal there exists some neighborhood on which the degree function is constant.
%(this is also equivalent to asking that $A$ has finite degree as discussed in Definition~\ref{def-degree}).
%We are finally ready to state
The following is our comparison result.
%which can also be found in the body of the paper as Corollary~\ref{cor-sep-lcf=cov}.

\begin{thmx}[see Corollary~\ref{cor-sep-lcf=cov}.]\label{thm-intro1}
    Let $\C$ be a stable homotopy theory with unit object $\1\in\C$. Suppose that $\pi_0(\1)$ decomposes as a finite product of indecomposable rings. 
    Then we have an equality
    \[
    \CAlg^{\cov}(\C)=\CAlg^{\sep,\mathrm{lcf}}(\C^\dual)
    \]
    between the finite covers of $\C$ and those separable commutative algebras of $\C$ which have a dualizable underlying module, and whose degree function is finite and locally constant.
\end{thmx}

%To use this result for classifying commutative separable algebras, we recall that  
%for a stable homotopy theory $\C$ with unit $\1$, there are always containments 
%\[
%\Cov_{\spec(\pi_0(\1))}\subseteq \CAlg^{\cov}(\C) \subseteq \CAlg^{\sep}(\C^\dual) \subseteq \CAlg^{\sep}(\C).
%\]
%This gives a way of constructing separable algebras in $\C$ but also a step-by-step strategy for classifying them.  Mathew~\cite{Mathew2016} has provided many tools (such as descent results and invariance properties of the Galois group) for classifying finite covers of $\C$, making the task of describing $\CAlg^\cov(\C)$ quite approachable in practise. Theorem~\ref{thm-intro1} as well as the various descent results established in this paper (Proposition~\ref{prop-descent-sep} and Theorem~\ref{thm-descent} for example) provide new tools for calculating $\CAlg^{\sep}(\C^\dual)$. In the second part of this paper we demonstrate the strength of these techniques by providing numerous classification results for separable algebras. We record some of them in the next result. 

We use this comparison result along with some new structural results about commutative separable algebras in order to obtain the following.

\begin{thmx}\label{thm-intro-2} There is an essentially complete\footnote{modulo the possible existence of separable commutative algebras having infinite degree in the stable homotopy theories considered here.} classification of the separable and commutative algebras with underlying dualizable module in the following homotopy theories:
    \begin{itemize}
        \item[(1)] Modules over a connective $\E_\infty$-ring, see Proposition~\ref{prop-connective-sep}.
        \item[(2)] Modules over an even periodic $\E_\infty$-ring $R$ with $\pi_0(R)$ regular and Noetherian with $2\in \pi_0(R)^\times$, see Theorem~\ref{thm-sep-even-periodic}.
        \item[(3)] Complete modules over an adic $\E_\infty$-ring, see Theorem~\ref{thm-complete-sep}.
        \item[(4)] Modules over Lubin-Tate theories over perfect fields of characteristic $p>0$, see Theorem~\ref{thm-Lubin-Tate}. 
        \item[(5)] Modules over topological complex and real K-theories, see Theorems \ref{complex-k-theory} and \ref{real-k-theory}.
        \item[(6)] Quasi-coherent sheaves over an even periodic refinements of a regular Noetherian Deligne-Mumford stack defined over $\spec(\Z[1/2])$, see Theorem~\ref{thm-even-ref-DM}.
        \item[(7)] Derived categories of  qcqs schemes, see Corollary~\ref{cor-sep-scheme}.
        \item[(8)] The stable module category of a finite group of $p$-rank one with coefficients in a separably closed field of characteristic $p>0$, see Theorem~\ref{thm-stable-module-cat}.
    \end{itemize}
\end{thmx}

We conclude the introduction by a section-wise overview:\\
Section \ref{sec:prelim} collects preliminaries on category theory, including a brief review of the ``small'' and ``large'' homotopy theories we are using and culminating in a description of categories of modules in a limit stable homotopy theory.\\
Section \ref{sec:connected_stable} reviews connected components of stable homotopy theories and their relation to the unit of the category, and the Balmer spectrum of its dualizable objects. We need this to formulate the mild technical hypothesis we need for our comparison result, namely that there be only finitely many connected components.\\
Section \ref{sec:separable} collects various results from the literature for the ease of reference. Most notably, Lemma \ref{lem-module-over-sep-alg}
records a variant of the fact that the formation of modules for a separable algebra commutes over passage to the homotopy category. This result is of key importance, as it allows us to freely pass between the infinity-categorical and the triangulated categorical level.\\
In Section \ref{sec:degree} we show that Balmer's splitting tower for a separable algebra in a tt-categories uniquely lifts to the setting of infinity categories. We relate the resulting local degree function on the Balmer spectrum with the splitting tower, and record its functoriality. We note that separable algebras which are finite products of retracts of the unit object always have locally constant degree functions.\\
In Section \ref{sec:descent-separable} we show that the formation of separable algebras commutes over arbitrary limits of homotopy theories. This result was suggested by the analogous result of Mathew for weak finite covers. The key point in the proof is the uniqueness of splitting idempotents for {\em commutative} separable algebras. This last result seems new even in the case of ordinary commutative rings.\\
In Section \ref{sec:axiomatic_galois} we collect results from Mathew's work on Galois theory which we will have to use later on: The precice working of his Galois correspondence and some observations about $G$-torsors.\\
Section \ref{finite_covers_and_sep} is the core of part 1. Here, we establish our comparison result equating finite covers and suitable separable algebras. Deducing  separability properties of (weak) finite covers is essentially a revisit of results of Rognes in the present more general setting, coupled with descent results for the Balmer spectrum. To see that certain separable algebras are finite covers, the observation is that the final term in their tt-tower is {\em descendable} and splits the algebra.\\
Section \ref{sec-descent} is motivated by the observation that for certain basic homotopy theories, (weak) finite covers agree with separable algebras (with neccessary extra condition). The present section shows that this simple relation is preserved under limits of homotopy theories. 
%This unlocks many examples in part 2 of the paper.
To see this, we need to revisit the notion of limits of Galois categories and equate Mathew's notion of rank in Galois theory with Balmer's notion of degree of separable algebras.\\
Section \ref{sec:modules} is the first of part 2 of the paper, about specific computations. Here we consider categories of modules over $\mathbb E_\infty$-rings which are either connective or even periodic and regular Noetherian in which $2$ acts invertibly. In both cases Mathew showed that all finite covers are classical, and we extend this result to include all separable algebras with underlying dualizable module. The approach is very standard, relying on residue fields, respectively suitable homology theories, to reduce to the case of ordinary (graded) commutative rings.\\
As a first application of our descent results, Section \ref{sec-complete} determines the separable algebras in the category of complete modules over a complete $I$-adic $\mathbb E_\infty$-ring.
This stable homotopy theory is a bit more complicated to access, e.g. its unit is not compact. The key here is a result of independent interest: The only dualizable modules for the completed tensor product are the perfect modules\footnote{we include an example of Mathew showing this requires connectivity}, and these in turn organize into a limit over a suitable $I$-adic tower.\\
In Section \ref{sec:chromotopy} we classify separable algebras over Lubin-Tate theories of perfect fields of characteristic $p>0$, and using descent techniques, we access the $L_n$- and $K(n)$-local categories. We also treat the case of topological K-theories and global sections of suitable non-connective spectral Deligne-Mumford stacks.\\
In Section \ref{sec:spectral-DM} we show that the theories of Mathew and Balmer yield the same result when applied to the categories of quasi-coherent sheaves on a spectral Deligne-Mumford stack. When specialized to Noetherian qcqs schemes, this recovers a special case of previous work of Neeman.\\
In the final Section \ref{sec:stable-mod} we consider stable module categories for finite groups and classify separable algebras there in the case of groups of $p$-rank one. This generalizes previous results of Balmer and Carlson which indeed were the initial motivation for the present work. Our proof is arguably a lot more direct than theirs, which for example relied on unexpected and subtle properties of the Kelly radical. We explain why extending these results to the case of an elementary abelian $p$-group of rank two likely requires new ideas.

%\subsection*{Related work}
%The proof of Theorem~\ref{thm-intro1} requires several ingredients of Galois theory coming from work of Rognes~\cite{Rognes2008} and Mathew~\cite{Mathew2016}, as well as work of Pauwels~\cite{Bregje} on quasi-Galois theory of symmetric monoidal categories. 
%Theorem~\ref{thm-intro-2}(5) can also be obtained from work of Neeman~\cite{Neeman2018} under the additional assumption that scheme is Noetherian.
%for a more detailed discussion see Remark~\ref{rem-Neeman-work}.
%Theorem~\ref{thm-intro-2}(6) extends a result of Balmer--Carlson~\cite{BC2018} for the stable module category of a cyclic $p$-group to any finite group of $p$-rank one. 
%As far as the authors are concerned these are the only classification results available in the literature for separable algebras in the stable module category.

\subsection*{Acknowledgements} We thank Jacob Lurie for useful remarks on section \ref{sec-complete}, and Paul Balmer. We also thank Drew Heard, Marc Hoyois, Lars Kadison and Massimo Pippi for useful conversations. We thank Maxime Ramzi for pointing out a mistake in a previous draft, for helpful discussions about this paper and for sharing a draft of his forthcoming work~\cite{Maxime} which has substantial overlap with the material in our sections~\ref{sec:separable} and \ref{sec-descent}.  We also thanks the anonymous referee for helpful comments on earlier versions of this manuscript.

The authors are supported by the SFB 1085 Higher Invariants in Regensburg. This material is partially based upon work supported by the Swedish Research Council under grant no.~2016-06596 while LP was in residence at Institut Mittag-Leffler in Djursholm, Sweden during the semester ''Higher algebraic structures in algebra, topology and geometry''. LP would also like to thank the Hausdorff Research Institute for Mathematics for the hospitality in the context of the Trimester program ''Spectral Methods in Algebra, Geometry, and Topology'', funded by the Deutsche Forschungsgemeinschaft (DFG, German Research Foundation) under Germany’s Excellence Strategy – EXC-2047/1 – 390685813.

\part{Finite covers and separable commutative algebras}

In this part of the paper we relate the notion of finite cover introduced by Mathew with the notion of separable commutative algebra introduced into homotopy theory by Balmer. After recalling the relevant background we show that finite covers coincide with those separable commutative algebras which have a dualizable underlying module and a locally constant and finite degree function, see Corollary~\ref{cor-sep-lcf=cov}.

\section{Preliminaries}\label{sec:prelim}

In this section we recall the notation and terminology from homotopy theory that we use throughout the paper. We start by introducing the notion of a $2$-ring and of a stable homotopy theory and then discuss various results about limits and module categories.\\

\begin{notation}
    Let $\Cat_\infty^{\mathrm{perf}}$ denote the $\infty$-category of essentially small, idempotent complete, stable $\infty$-categories and exact functors, see~\cite{Mathew2016}*{Definition 2.3}. 
\end{notation}

Recall from~\cite{Mathew2016}*{Definition 2.13} the relevant symmetric monoidal structure on $\Cat_\infty^{\mathrm{perf}}$. 

\begin{definition}
 We let $\twoRing$ denote the $\infty$-category of commutative algebra objects in $\Cat_\infty^{\perf}$. More concretely, this is the $\infty$-category of 
 essentially small, idempotent complete, stable $\infty$-categories equipped with a symmetric monoidal structure whose tensor products are exact in each variable. 
 A functor between such categories is assumed to be symmetric monoidal and exact. We will refer to any $\C \in \twoRing$ simply as a $2$-\emph{ring}. 
\end{definition}

We will also need a big variant of the previous definitions. 

\begin{notation}
    Let $\mathsf{Pr}^L $ denote $\infty$-category of presentable $\infty$-categories and colimit-preserving functors, and let $\Pr$ denote the full subcategory of $\mathsf{Pr}^L$ spanned by the stable $\infty$-categories. 
\end{notation}

Recall from~\cite{HA}*{Proposition 4.8.1.15} that the $\infty$-category $\mathsf{Pr}^L$ inherits a symmetric monoidal structure from the $\infty$-category of (large) $ \infty$-categories $\widehat{\Cat}_\infty$ whose commutative algebra objects are precisely the symmetric monoidal and presentable $\infty$-categories whose tensor product preserves colimits in each variable. For the stable version, we recall from~\cite{HA}*{Proposition 4.8.2.18} that the $\infty$-category of 
spectra $\Sp$ is an idempotent commutative algebra in $\mathsf{Pr}^L$ and that $\mod{\Sp}(\mathsf{Pr}^L)\simeq \Pr$. It then follows from~\cite{HA}*{Proposition 4.8.2.10} that $\Pr$ inherits a symmetric monoidal structure.

\begin{definition}
    We let $\CAlg(\Pr)$ denote the $\infty$-category of commutative algebra objects in $\Pr$. More concretely, this is the $\infty$-category of presentable, symmetric monoidal, stable $\infty$- categories such that the monoidal structure commutes with colimits in each variables. A functor between such categories is assumed to be symmetric monoidal and colimit-preserving. We will refer to any $\C \in \CAlg(\Pr)$ simply as a \emph{stable homotopy theory}. We note that any such $\C$ is closed symmetric monoidal and write $\uHom$ and $D=\uHom(-,\1)$ for the internal hom and duality functor respectively.
\end{definition}

There is a canonical way to produce $2$-rings from a stable homotopy theory:

\begin{lemma}\label{lem-dual-is-2ring}
 For any $\C \in \CAlg(\Pr)$, the full subcategory of dualizable objects $\C^\dual\subseteq \C$ naturally admits the structure of a $2$-ring. %For every regular cardinal $\kappa$ such that $\1\in\C$ is $\kappa$-compact, the full subcategory $\C^\kappa\subseteq\C$ of $\kappa$-compact objects naturally admits the structure of a $2$-ring.
\end{lemma}

\begin{proof}
 See the discussion surrounding~\cite{Mathew2016}*{Definition 2.15}. For the convenience of the reader we flesh out the argument here.
 Since an object $X\in\C$ is dualizable if and only if the natural transformation
 \[ D(X)\otimes (-)\longrightarrow \uHom(X,-)\]
 is an equivalence, and both $D(X)$ and $\uHom(X,-)$ are
 exact functors in $X$, it follows that $\C^{\mathrm{dual}}\subseteq\C$ is a stable and thick subcategory. It is in particular stable under retracts in $\C$, and hence idempotent complete. If both $X$ and $Y$ are dualizable, then we leave it to the reader to check that the given duality data easily provide one for $X\otimes Y$, see~\cite{Mathew2016}*{Definition 2.15} for the definition of a duality datum. It is only left to show that $\C^\dual$ is essentially small, for then $\C^{\mathrm{dual}}$ endowed with the tensor product of $\C$, will be a $2$-ring.\\
 We first claim that the unit $\1\in\C$ is $\lambda$-compact for some regular cardinal $\lambda$. Since $\C$ is in particular $\kappa$-accessible for some regular cardinal $\kappa$, it follows from~\cite{HTT}*{Proposition 5.4.2.2, (2)} that $\1$ is a $\kappa$-filtered colimit of $\kappa$-compact objects. 
  Choose a regular cardinal $\lambda$ such that $\lambda \ge  \kappa$, and such that the above colimit is $\lambda$-small. The existence of $\lambda$ is secured by the existence of arbitrarily large regular cardinals. Since any $\kappa$-compact object is $\lambda$-compact, we have written $\1$ as a $\lambda$-small colimit of $\lambda$-compact objects, hence $\1$ is $\lambda$-compact by~\cite{HTT}*{Corollary 5.3.4.15}. To see that every dualizable object $X$ is $\lambda$-compact consider the following sequences of equivalences
 \[
 \Hom_\C(X,-)\simeq
 \Hom_\C(\1,\uHom(X,-))\simeq\Hom_\C(\1, (-) \otimes DX).
 \]
 As the functors $\Hom_{\C}(\1,-)$ and $(-)\otimes DX$ commute with $\lambda$-filtered colimits, $X$ is $\lambda$-compact, as claimed. We conclude, since the category of $\lambda$-compact objects is essentially small by~\cite{HTT}*{Remark 5.4.2.13}.\\
 %For our second claim, now assume that $\kappa$ is a regular cardinal such that $\1\in\C^\kappa$, and we aim to see that $\C^\kappa$ with the tensor product induced from $\C$ is a $2$-ring. Since finite limits commute over ($\kappa$-)filtered colimits, we see that $\C^\kappa\subseteq\C$ is stable under finite limits, hence a stable and thick subcategory, clearly stable under retracts in $\C$, and thus idempotent complete. If $X,Y\in\C^\kappa$, then since 
 %\[ \Hom_\C(X\otimes Y, -)\simeq\Hom_\C(X,\uHom_\C(Y,-)),\]
 %and both $\Hom_\C(X,-)$ and $\uHom_\C(Y,-)$ commute with $\kappa$-filtered colimits, we see that $X\otimes Y\in\C^\kappa$. We have $\1\in\C^\kappa$ by assumption and noted above that $\C^\kappa$ is essentially small. This shows that $\C^\kappa$ is a $2$-ring.  
\end{proof}

  Conversely, we can promote every $2$-ring $\C$ to a stable homotopy theory:
  The $\omega$-$\mathrm{Ind}$-completion $\mathrm{Ind}(\C)$ is presentable ~\cite{HTT}*{Theorem 5.5.1.1}, stable~\cite{HA}*{Proposition 1.1.3.6} and symmetric monoidal~\cite{HA}*{Corollary 4.8.1.14}, hence a stable homotopy theory. This has the property that the subcategory of compact objects of $\mathrm{Ind}(\C)$ coincides with $\C$, and that the Yoneda embedding $j \colon \C \to \mathrm{Ind}(\C)$ is symmetric monoidal.

The following result describes how to calculate limits of stable homotopy theories.

\begin{lemma}\label{lem-limits-preserved}
 The $\infty$-category $\CAlg(\Pr)$ admits all limits and these are preserved by the 
 forgetful functor $\CAlg(\Pr) \to \widehat{\Cat}_\infty$ into the large $\infty$-category of 
 $\infty$-categories.
\end{lemma}

\begin{proof}
The forgetful functor $\CAlg(\Pr) \to \widehat{\Cat}_\infty$ factors as the composite 
\[
\CAlg(\Pr) \to \Pr \to \mathsf{Pr}^L \to \widehat{\Cat}_\infty.
\]
Limits in $\mathsf{Pr}^L$ exists and the forgetful functor $\mathsf{Pr}^L \to \widehat{\Cat}_\infty$ preserves them by~\cite{HTT}*{Proposition 5.5.3.13}. Using this together with the fact that limits of stable $\infty$-categories are again stable~\cite{HA}*{Theorem 1.1.4.4}, we deduce that $\Pr$ admits limits and that 
these are preserved by the forgetful functor $\Pr \to \widehat{\Cat}_\infty$. For the final step, note that the symmetric monoidal $\infty$-category $\Pr$ defines~\cite{HA}*{Example 2.1.2.18} a cocartesian fibration of $\infty$-operads $(\Pr)^\otimes\to \mathrm{Fin}_*$ whose fibre over $\langle n \rangle \in \mathrm{Fin}_*$ is equivalent to $\prod_{i=1}^n \Pr$ by the Segal condition. Thus by~\cite{HA}*{Corollary 3.2.2.5}, the $\infty$-category $\CAlg(\Pr)$ admits limits, and these are preserved by the forgetful functor $\CAlg(\Pr)\to \Pr$. 
\end{proof}

As an application of the previous lemma, we record the following useful result about commutative algebras and limits.

\begin{lemma}\label{lem-calg-limits}
 Let $K$ be a simplicial set and let $p\colon K \to \CAlg(\Pr)$ be a functor. 
 Then the canonical functor
 \[
 \CAlg(\lim_K p) \xrightarrow{\simeq} \lim_{k \in K}\CAlg(p(k))
 \]
 is an equivalence in $\widehat{\Cat}_\infty$.
\end{lemma}

\begin{proof}
 By Lemma~\ref{lem-limits-preserved}, limits in $\CAlg(\Pr)$ can be calculated in 
 $\widehat{\Cat}_\infty$.  Recall that a commutative algebra object in a symmetric monoidal $\infty$-category $\C$ is a section of the associated cocartesian fibration defined in~\cite{HA}*{Example 2.1.2.18}. Equivalently, a commutative algebra object is a functor $\mathrm{Fin}_*\to \C$ satisfying the Segal conditions (certain maps into a product are equivalences). As limits in functor categories are calculated pointwise we see that the canonical map $\Fun(\mathrm{Fin}_*, \lim_i \C_i)\to \lim_i \Fun(\mathrm{Fin}_*, \C_i)$ is an equivalence. It then suffices to check that a limit of commutative algebras is again a commutative algebra. This can be checked pointwise and it holds as limits preserves equivalences and commutes with products. 
\end{proof}

Let $\C$ be a stable homotopy theory and consider $A\in\CAlg(\C)$. As discussed in~\cite{HA}*{Section 4.5}, there is an $\infty$-category $\mod{A}(\C)$ of $A$-modules internal to $\C$. This is again a stable homotopy theory by~\cite{HA}*{Proposition 4.2.3.4, Corollary 4.2.3.7 and Theorem 4.5.2.1} and comes with a lax monoidal forgetful functor 
\[ U_A \colon \mod{A}(\C) \to \C\]
which is conservative~\cite{HA}*{Corollary 4.2.3.2} and (co)continuous~\cite{HA}*{Corollaries 4.2.3.3 and 4.2.3.5}. It follows that the forgetful functor $U_A$ admits left and right adjoints $F_A$ and $R_A$ respectively. We recall how module categories behaves with respect to lax symmetric monoidal functors.

\begin{remark}\label{rem-transport}
Let $F \colon \C \to \D$ be a lax monoidal functor between stable homotopy theories. For any commutative algebra $A \in \CAlg(\C)$, there is a commutative diagram of lax monoidal functor 
\[
\begin{tikzcd}
\mod{A}(\C) \arrow[r,"\overline{F}"]\arrow[d,"U_A"'] &\mod{F(A)}(\D) \arrow[d,"U_{F(A)}"]\\
\C \arrow[r, "F"] & \D,
\end{tikzcd}
\]
with $U_A$ and $U_{F(A)}$ denoting the forgetful functors:
See the discussion in \cite{ergus2022hopf}*{Remark 1.1.11} or \cite{Robalo}*{Section 3.3.9}. If in addition $F$ is symmetric monoidal and preserves colimits, then $\overline{F}$ is also symmetric monoidal and preserves colimits: see~\cite{Robalo}*{Proposition 3.18} for monoidality and to see the preservation of colimits one can use conservativity and cocontinuity of $U_{F(A)}$ together with the commutative diagram above. In particular, in this case,  $\overline{F}$ is a left adjoint, so that the above diagram extends as follows:

\begin{equation}\label{diag:radj}
\begin{tikzcd}
\mod{A}(\C) \arrow[r, shift left, "\overline{F}"] \arrow[d, shift left, "U_A"] &\mod{F(A)}(\D) \arrow[d,shift left, "U_{F(A)}"] \arrow[l, shift left, "\overline{G}"]\\
\C \arrow[u, shift left, "F_A"]\arrow[r, shift left, "F"] & \D \arrow[u, shift left, "F_{F(A)}"]\arrow[l, shift left, "G"],
\end{tikzcd}
\end{equation}

where $G$ (resp. $\overline{G}$) denotes the 
right adjoint of $F$ (resp. $\overline{F}$) and $F_A$ (resp. $F_{F(A)}$) denotes the free-module functor. Since $F_A(X)=A\otimes X$, we see that the canonical natural transformation 
\[ \overline{F}\circ F_A\simeq F_{F(A)}\circ F \]
is an equivalence, and a formal argument using the present adjunctions shows, that also the exchange transformation
\[ U_A\circ \overline{G}\simeq G\circ U_{F(A)}\]
is an equivalence; i.e. the diagram (\ref{diag:radj}) is right adjointable.

\end{remark}

\begin{lemma}\label{lem-module-limits}
  Let $K$ be a simplicial set and $p\colon K \to \CAlg(\Pr)$ be a functor. Consider $A\in\CAlg(\lim_K p)$ and write $A_k\in\CAlg(p(k))$ for its $k$-component. Then the canonical functor
  \[
  \mod{A}(\lim_K p) \to \lim_k \mod{A_k}(p(k))
  \]
  is an equivalence in $\widehat{\Cat}_\infty$.
\end{lemma}

\begin{proof}
 Set $\C=\lim_K p$ and $\C_k=p(k)$ so that $\C=\lim_k \C_k$. By assumption the projection functor $F_k \colon \C \to \C_k$ is a symmetric monoidal left adjoint (say with right adjoint $G_k$) so by Remark~\ref{rem-transport} there is a diagram of adjunctions
 \[
 \begin{tikzcd}
 \mod{A}(\C)\arrow[r, shift left, "\overline{F}_k"] \arrow[d,"U"'] & \mod{A_k}(\C_k) \arrow[l, shift left, "\overline{G}_k"]\arrow[d,"U_k"]\\
 \C \arrow[r,shift left, "F_k"] & \C_k \arrow[l, shift left, "G_k"]
 \end{tikzcd}
 \]
 satisfying 
\begin{equation}\label{eq-fgt}
 U_k\circ  \overline{F}_k \simeq F_k \circ U \quad \mathrm{and}\quad U \circ \overline{G}_k=G_k\circ U_k.
 \end{equation}
 Using \cite{HA}*{Corollary 4.7.4.18, (2)} this induces another diagram of adjunctions
 \[
 \begin{tikzcd}
 \mod{A}(\C)\arrow[r, shift left, "\overline{F}"] \arrow[d,"U"'] & \lim_k\mod{A_k}(\C_k) \arrow[l, shift left, "\overline{G}"]\arrow[d,"\lim U_k"]\\
 \C \arrow[r,shift left, "F"] & \lim_k \C_k \arrow[l, shift left, "G"],
 \end{tikzcd}
 \]
 given on objects by
 \[
 F(X)=\{F_k(X)\}_k \quad \mathrm{and} \quad\overline{F}(M)=\{\overline{F}_k(M)\}_k
 \]
 and 
 \[
 G(\{X_k\}_k)=\lim_k G_k(X_k) \quad \mathrm{and}\quad \overline{G}(\{M_k\}_k)=\lim_k \overline{G}_k(M_k)
 \]
 where $X\in\C$, $X_k\in\C_k$, $M\in\mod{A}(\C)$ and $M_k\in\mod{A_k}(\C_k)$. Note that by our assumption the unit and counit maps of the adjunction $(F,G)$ are equivalences. We claim that the unit and counit of the adjunction $(\overline{F},\overline{G})$ are also equivalences. This will conclude the proof of the lemma.

 For any $M\in\mod{A}(\C)$, we will show that the unit map $M \to \overline{G}\,\overline{F}M=\lim_k \overline{G}_k\overline{F}_k M$ is an equivalence. By conservativity of the forgetful functor, we can instead check that the map $UM \to U \lim_k \overline{G}_k\overline{F}_k M$ is an equivalence. Using that $U$ preserves limits and~(\ref{eq-fgt}) we can identify this with the map $UM \to \lim_k G_k F_k UM=GFUM$ which is the unit of the adjunction $(F,G)$. Hence an equivalence by assumption.

 Finally let us consider the counit map $\overline{F}\,\overline{G}\{M_k\}\to \{M_k\}$. It will suffices to verify that for all $k\in K$, the map $U_k\overline{F}_k \lim_k \overline{G_k}\{M_k\}\to U_k M_k$ is an equivalence. Using that $U$ preserves limits and equations~(\ref{eq-fgt}) we can identify this with the map $F_k G \{U_k M\}\to U_k M$ which is the $k$-component of the counit map $FG\{U_k M_k\}\to \{U_k M_k\}$, hence an equivalence.
\end{proof}

\section{Connected stable homotopy theories}\label{sec:connected_stable}
In this section we recall the notion of a connected stable homotopy theory following~\cite{Mathew2016}, introduce a slight generalization (fin-connected stable homotopy theories), and give a characterization using the Balmer spectrum. 
%Before diving into the definitions, 
We need to record some results about idempotent elements and indecomposable commutative algebras. We fix a stable homotopy theory
$\C\in\CAlg(\Pr)$. Recall that an {\em idempotent} of $A\in \CAlg(\C)$ is by definition an idempotent of the commutative ring 
\[ \pi_0(A):=\pi_0\left( \Hom_\C(\1,A)\right).\]

\begin{lemma}\label{lem-idempotent-splitting}
Let $e$ be an idempotent of $A\in \CAlg(\C)$.  Then, there is a splitting
\[
A \simeq A[e^{-1}]\times A[(1-e)^{-1}]
\]
in $\CAlg(\C)$ such that $\pi_0(A[e^{-1}])=e\pi_0(A)$, and similarly for $(1-e)$. 
Conversely, given a splitting of $A\simeq B \times C$ in $\CAlg(\C)$, there exists an idempotent 
$e$ of $A$ such that $B \simeq A[e^{-1}]$ and $C\simeq A[(1-e)^{-1}]$.
\end{lemma}

\begin{proof}
 See the discussion after~\cite{Mathew2016}*{Definition 2.38}.
\end{proof}

\begin{definition}
 We say that $A \in \CAlg(\C)$ is \emph{indecomposable} if $A$ is nonzero and if it does not decompose as a 
  product of two nonzero algebras.
\end{definition}

\begin{remark}
    Using Lemma~\ref{lem-idempotent-splitting} we see that $A \in\CAlg(\C)$ is indecomposable if and only if the 
 discrete ring $\pi_0(A)$ is indecomposable. 
\end{remark}

\begin{comment}
We give another characterization of indecomposable algebras which follows easily from .
\begin{lemma}
 Consider $A \in\CAlg(\C)$. Then $A$ is indecomposable if and only if the 
 discrete ring $\pi_0(A)$ is indecomposable. 
\end{lemma}
\begin{proof}
 We prove the contrapositive implications. 
 Firstly, note that $A \simeq 0$ if and only if $\pi_0(A)=0$. The forward implication is clear and 
 the backward implication follows from the fact that if $\pi_0(A)$ is zero, then the unit map $\1\to A$ is zero, which immediately implies that $A\simeq 0$.\\
 Let us now show that $A$ is decomposable if and only if $\pi_0(A)$ is so.
 If $A$ decomposes as $B \times C$ in $\CAlg(\C)$ then $\pi_0(A)$ also decompose as 
 $\pi_0(B) \times \pi_0(C)$ since $\pi_0$ preserves finite products. Thus $\pi_0(A)$ is decomposable. 
 Conversely, if $\pi_0(A)$ decomposes as product of two nonzero rings, then we can find a nontrivial 
 idempotent $0,1\neq e$ of $\pi_0(A)$. Then by Lemma~\ref{lem-idempotent-splitting} we have 
 a decomposition $A \simeq A[e^{-1}]\times A[(1-e)^{-1}]$ with non-zero factors, showing that $A$ decomposes nontrivially. 
\end{proof}
\end{comment}

%We are finally ready to give the key definitions of this section.

\begin{definition}\label{def-connected}
 Let $\C$ be a stable homotopy theory.
 \begin{itemize}
 \item We say that $\C$ is \emph{connected} if $\1$ is indecomposable in $\CAlg(\C)$. 
 This is also equivalent to $\pi_0(\1)$ being indecomposable as a discrete ring.
  Note that if $\C$ is connected then 
  $\1 \not \simeq 0$ so $\C$ is nonzero.
 \item We say that $\C$ is \emph{fin-connected} if $\1$ decomposes as a non-empty finite product of indecomposable rings.
 \end{itemize}
\end{definition}

\begin{example}
 Any finite product of connected stable homotopy theories is fin-connected.
\end{example}

We give another useful characterization of connected and fin-connected stable homotopy theories using the Balmer spectrum. First we need a little bit of preparation.

\begin{lemma}\label{lem-balmer-spectrum-decomposes}
 Let $\C$ be a stable homotopy theory and suppose that the Balmer spectrum $\Spc(\C^\dual)$ is the disjoint union of two Thomason subsets $U$ and $V$.
 Write $\C_U^\dual, \C^\dual_V\subseteq \C^\dual$ for the full subcategories of objects with support contained in $U$ and $V$ respectively. 
 Then there is a decomposition $\C^\dual \simeq \C^\dual_U \times \C^\dual_V$ with 
 $\Spc(\C_U^\dual)=U$ and $\Spc(\C^\dual_V)=V$.
\end{lemma}

\begin{proof}
 By~\cite{Balmer2}*{Theorem 2.11 and Corollary 2.8}, every object $X\in \C^{\dual}$ is equivalent to a direct sum $X_U \oplus X_V$ with $X_U\in \C^\dual_U$ and $X_V \in \C^\dual_V$, and we have
 \[ \Hom_\C( \C_U^{\mathrm{dual}}, \C_V^{\mathrm{dual}}) = \Hom_\C( \C_V^{\mathrm{dual}}, \C_U^{\mathrm{dual}}) =0 .\]
 In particular, the functor $\C_{U}^{\dual} \times \C_{V}^{\dual} \to \C^{\dual}$ sending $(X_U, X_V)$ to $X_U \oplus X_V$ is an equivalence. The claim on the Balmer spectra follows by direct verification.
\end{proof}

\begin{proposition}\label{prop-char-connected-Balmer}
 Let $\C$ be a stable homotopy theory. Then $\C$ is connected if and only if $\Spc(\C^\dual)$ 
 is connected and nonempty. 
 Moreover, $\C$ is fin-connected if and only if $\Spc(\C^\dual)$ is non-empty and has finitely many connected components. 
\end{proposition}

\begin{proof}
 If $\C$ is connected, then $\1\not \simeq 0$ and so $\C^\dual$ is nonzero. 
 It follows from~\cite{Balmer}*{Lemma 2.2} that $\Spc(\C^\dual)$ is nonempty. Now suppose by 
 contradiction that $\Spc(\C^\dual)$ is disconnected so 
 that there exist two nonempty open subsets $U$ and $V$ such that $\Spc(\C^\dual)=U \sqcup V$. 
 Note that $U$ and $V$ 
 are quasi-compact as any cover can be extended to a cover of all $\Spc(\C^\dual)$, 
 which by quasi-compactness admits a finite subcover. Therefore $U$ and $V$ are Thomason and so 
 by Lemma~\ref{lem-balmer-spectrum-decomposes} we get a decomposition 
 $\C^\dual\simeq \C^\dual_U \times \C^\dual_V$ which in turns gives a decomposition of $\1$ as a 
 product of two nonzero rings. This contradicts our starting assumption, so $\Spc(\C^\dual)$ 
 must be connected. 
 Conversely, let $\Spc(\C^\dual)$ be connected and nonempty,
 and suppose that $\C$ is not connected so that $\1$ decomposes nontrivially as 
 $\1_0 \times \1_1$ in $\CAlg(\C)$.  It 
 then follows that $\C \simeq \C_0 \times \C_1$ and so $\C^\dual \simeq \C_0^\dual \times \C_1^\dual$ 
 by~\cite{HA}*{Proposition 4.6.1.11}. By~\cite{BSKS}*{Theorem A.5} or by direct verification, we then obtain a nontrivial 
 decomposition $\Spc(\C^\dual)\simeq \Spc(\C_0^\dual) \sqcup \Spc(\C_1^\dual)$ contradicting our 
 starting assumption. Therefore $\C$ must be connected. 
 If $\C$ is fin-connected then $\C \simeq \prod_{i=1}^n \C_i$ for connected $\C_i$ 
 and by the same argument as above we see that 
 $\Spc(\C^\dual) \simeq \bigsqcup_{i=1}^n \Spc(\C_i^\dual)$ where each $\Spc(\C_i^\dual)$ is 
 connected. Finally, if $\Spc(\C^{\mathrm{dual}})$ is non-empty with finitely many connected components, an easy induction on the number of connected components and using the splitting established above shows that $\C$ is fin-connected.
\end{proof}

\section{Separable commutative algebras}\label{sec:separable}

 In this section we recall the notion of a separable commutative algebra and list some important properties that we will use throughout the paper.
 We record one version of the surprising fact that for a {\em separable} commutative algebra, the formation of modules commutes over the passage to the homotopy category, see Lemma~\ref{lem-module-over-sep-alg}.
 
We set the following definition, directly paralleling ~\cite{Balmer2011}*{Definition 3.1}.

\begin{definition}\label{def-sep-alg}
 Let $\C$ be a symmetric monoidal $\infty$-category and $A\in \CAlg(\C)$.
 We say that $A$ is \emph{separable} if the multiplication map 
 $\mu \colon A \otimes A \to A$ admits a section as a map of $(A,A)$-bimodules. Equivalently, $\sigma$ is a map $\sigma \colon A \to A \otimes A$ in $\C$ such that $\mu \sigma = 1_A$ and $(1_A \otimes \mu) \circ (\sigma \otimes 1_A) = \sigma \mu = (\mu \otimes 1_A) \circ ( 1_A\otimes \sigma) \colon A \otimes A \to A \otimes A$. The commutative separable algebras span a full subcategory $\CAlg^\sep(\C)\subseteq \CAlg(\C)$.
\end{definition}

\begin{example}\label{ex-tensor-is-sep}
 If $A,B \in \CAlg(\C)$ are separable, then $A \otimes B$ is separable, as seen by 
 tensoring the individual bimodule sections. 
\end{example}

\begin{example}\label{ex-separable-product-retract}
 For $A,B \in \CAlg(\C)$, we can form the commutative algebra $A \times B$ with 
 componentwise structure. Then $A$ and $B$ are  separable if and only if $A \times B$ is so. Indeed if we have 
 bimodule sections for $\mu_A$ and $\mu_B$, we can combine then to get a bimodule section for 
 $\mu_{A \times B}$. Conversely if $\mu_{A \times B}$ admits a bimodule section, then by restricting 
 this section along each component we get sections for $\mu_A$ and $\mu_B$. 
\end{example}

We now discuss how separable algebras behaves under symmetric monoidal functors. 

\begin{construction}\label{con-smf-pres-sep}
 A symmetric monoidal functor $F \colon \C \to \D$ between symmetric monoidal $\infty$-categories induces 
 a functor $ \CAlg(\C)\to \CAlg(\D)$ on commutative algebras. 
 We claim that this preserves separable algebras. 
 To see this consider $A\in \CAlg^\sep(\C)$ so that the multiplication map $\mu_A$ admits a 
 $A$-bimodule section $\sigma_A$. Note that since $F$ is symmetric monoidal we have a commutative 
 diagram
 \[
 \begin{tikzcd}
  & F(A) & \\
  F(A) \arrow[ur, "1"] \arrow[r,"F(\sigma_A)"'] & F(A \otimes A) \arrow[u,"F(\mu_A)"'] & 
  F(A)\otimes F(A) \arrow[l,"\sim"] \arrow[ul,"\mu_{F(A)}"'].
 \end{tikzcd}
 \]
 Then the bottom composite defines a $F(A)$-bimodule section for $F(A)$ as required.
 Thus we get a well-defined functor $\CAlg^\sep(\C)\to \CAlg^\sep(\D)$ as claimed. 
\end{construction}

\begin{remark}\label{rem-sep-Balmer}
 Our definition of a separable algebra is identical to that given by Balmer in~\cite{Balmer2011}, with the only caveat that we work in the setting of $\infty$-categories as opposed to that of triangulated categories.
 We can relate the two notions by considering the canonical symmetric monoidal functor $u\colon \C \to \mathrm{h}\C$ into the homotopy category of $\C$. 
 We deduce from Construction~\ref{con-smf-pres-sep} that if $A$ is separable in $\C$, then $uA$ is 
 separable in $\mathrm{h}\C$. In fact, in forthcoming work Ramzi~\cite{Maxime} will show that the converse also holds.
\end{remark}

\begin{remark}\label{rem-sep-counit-section}
 Let $\C$ be a stable homotopy theory, and $A\in\CAlg(\C)$ be separable with multiplication map $\mu$, unit map $\eta$ and bimodule section $\sigma$. Set $\gamma:=\sigma \circ \eta \colon \1 \to A \otimes A$. Separability tells us that 
 \begin{equation}\label{eq-separability}
 \mu \circ \gamma=\eta \quad \mathrm{and} \quad (\mu \otimes 1_A)\circ (1_A \otimes \gamma)=(1_A \otimes \mu)\circ(\gamma \otimes 1_A).
 \end{equation}
 Consider the extension-restriction adjunction 
 \[
 F_A \colon \C \rightleftarrows \mod{A}(\C): U_A.
 \]
 If $A$ is separable, then the counit $\epsilon\colon F_A U_A \Rightarrow \mathrm{Id}$ admits a section $\xi \colon \mathrm{Id}\Rightarrow F_A U_A$ given by the composite 
 \[
 \xi_M \colon M \xrightarrow{\gamma \otimes 1_M} A \otimes A \otimes M \xrightarrow{1_A \otimes \mathrm{act}} A \otimes M= A \otimes U_A M
 \]
 compare with~\cite{Balmer2011}*{Proposition 3.11} (see also~\cite{Monads}*{2.9}(1)).
 One can check that this composite is $A$-linear and a section of $\epsilon_M$ using the identities in~(\ref{eq-separability}). For instance, the fact that $\xi_M$ is a section for $\epsilon_M$ follows from the following commutative diagram which uses the first identity in~(\ref{eq-separability}): 
 \[
 \begin{tikzcd}
  M \arrow[r,"\gamma \otimes 1_M"] \arrow[dr,"\eta \otimes 1_M"'] & A \otimes A \otimes M \arrow[d,"\mu \otimes 1_M"]\arrow[r,"1_A \otimes \mathrm{act}"] & A \otimes M \arrow[r,"\epsilon_M"] \arrow[d,"\mathrm{act}"]& M \\
              & A \otimes M \arrow[r,"\mathrm{act}"']  & M \arrow[ur, "1_M"']     & .
 \end{tikzcd}
 \]
 In particular, the above implies that $U_A\colon \mod{A}(\C)\to \C$ is faithful.
\end{remark}

The following is a remarkable consequence of separability and under stronger assumptions has already appeared in \cite{BCHNP}*{Remark 2.34}, based on \cite{DellAmbrogioBeren2018} and \cite{Sanders21pp}*{Proposition 3.8}. Here we give a more direct proof.

\begin{lemma}\label{lem-module-over-sep-alg}
 Let $\C$ be a stable homotopy theory and let $A \in \CAlg(\C)$ be separable. The symmetric monoidal functor $u\colon \C \to \mathrm{h}\C$ induces a symmetric monoidal equivalence 
 \[
 \mathrm{h}\mod{A}(\C) \simeq \mod{uA}(\mathrm{h}\C)
 \]
 which is compatible with the forgetful functor to $\mathrm{h}\C$.
\end{lemma}

\begin{remark}\label{rem-idempotent}
    In the situation of Lemma~\ref{lem-module-over-sep-alg}, the triangulated categories $\mathrm{h}\mod{A}(\C)$ and $\mod{uA}(\mathrm{h}\C)$ have small direct sums (since $\C$ is cocomplete). It then follows from \cite{BokstedtNeeman}*{Proposition 3.2} that these triangulated categories are idempotent complete. 
\end{remark}

\begin{proof}
Throughout the proof we will need to use that $uA$ is separable in the sense of Balmer (see Remark~\ref{rem-sep-Balmer}), so we can apply all the results of~\cite{Balmer2011}. Also, given $M$ in $\mathrm{h}\mod{A}(\C)$ or $\mod{uA}(\mathrm{h}\C)$, we denote by $U_AM$ the underlying object in $\mathrm{h}\C$ (which can also be seen as object in $\C$).

Let us start by noting that the symmetric monoidal functor $u\colon \C \to \mathrm{h}\C$ induces a lax symmetric monoidal functor at the level of modules $\mod{A}(\C)\to \mod{uA}(\mathrm{h}\C)$ by Remark~\ref{rem-transport}. We claim that this is in fact symmetric monoidal, i.e. that the canonical map $\alpha_{M,N}\colon uM\otimes_{uA}uN \to u(M \otimes_A N)$ is an equivalence for all $A$-modules $M$ and $N$. To this end recall from~\cite{HA}*{Section 4.4.2} that we can calculate $M \otimes_A N$ (resp. $uM \otimes_{uA}uN$) as the geometric realization of the simplicial object $n \mapsto A^{\otimes n} \otimes M \otimes N$ (resp. $n \mapsto uA^{\otimes n} \otimes uM \otimes uN$). We observe that the canonical map $\alpha_{M,N}$ is an equivalence if $M$ is of the form $A \otimes X$ for some $X\in \C$. This is due to the fact that the  simplicial object $n \mapsto A^{\otimes n} \otimes M \otimes N$ has an “extra degeneracy” or splitting \cite{HA}*{Section 4.7.2}, which in particular implies that it is a universal colimit diagram (hence it is preserved by any functor). Since $A$ is separable, for a general $A$-module $M$, we have a retraction diagram 
\[
M \xrightarrow{\xi_M} A \otimes U_A M \xrightarrow{\epsilon_M} M
\]
by Remark~\ref{rem-sep-counit-section}. Therefore for all $L\in \mod{uA}(\mathrm{h}\C)$, we have a commutative diagram 
\[\resizebox{\columnwidth}{!}{$\displaystyle
\begin{tikzcd}[ampersand replacement=\&]
\Hom_{ \mod{uA}(\mathrm{h}\C)}(L,uM \otimes_{uA}uN) \arrow[rr] \arrow[d,"(\alpha_{M,N})_*"]\& \& \Hom_{\mod{uA}(\mathrm{h}\C)}(L, u(A \otimes U_AM)\otimes_{uA} uN) \arrow[d,"(\alpha_{A \otimes U_AM,N)_*}","\sim"']\arrow[rr]\& \& \Hom_{\mod{uA}(\mathrm{h}\C)}(L,uM \otimes_{uA} uN) \arrow[d,"(\alpha_{M,N})_*"]\\
 \Hom_{\mod{uA}(\mathrm{h}\C)}(L, u(M \otimes_A N))\arrow[rr] \& \& \Hom_{\mod{uA}(\mathrm{h}\C)}(L, u(A \otimes U_AM \otimes_A N)) \arrow[rr] \& \& \Hom_{\mod{uA}(\mathrm{h}\C)}(L, u(M \otimes_A N))	
\end{tikzcd}
$}
\]
where the top and bottom horizontal arrows compose to the the respective identities. A simple diagram chase then shows that $(\alpha_{M,N})_*$ is an equivalence for all $L$, which in turn shows that $\alpha_{M,N}$ is an equivalence. This conclude the proof that $\mod{A}(\C)\to \mod{uA}(\mathrm{h}\C)$ is symmetric monoidal. Note that $\mod{A}(\C)\to \mod{uA}(\mathrm{h}\C)$ factors through $\mathrm{h}\mod{A}(\C)$ by~\cite{HTT}*{Proposition 1.2.3.1} as the target is equivalent to the nerve of a $1$-category. We abuse notation and still denote by $u$ the resulting symmetric monoidal functor $\mathrm{h}\mod{A}(\C) \to \mod{uA}(\mathrm{h}\C)$. 

Next we check that this functor induces an equivalence of underlying $\infty$-categories. To this end let $\mathrm{h}\mod{A}(\C)^{\mathrm{free}}$ and $ \mod{uA}(\mathrm{h}\C)^{\mathrm{free}}$ be the full subcategories spanned by the free modules, that is those of the form $A \otimes X$ and $uA \otimes X$ for some $X \in \C$. We have an induced functor $\mathrm{h}\mod{A}(\C)^{\mathrm{free}}\to \mod{uA}(\mathrm{h}\C)^{\mathrm{free}}$ which is evidently essentially surjective, and fully faithful
\[
\Hom_{\mathrm{h}\mod{A}(\C)}(A \otimes X, A \otimes Y)\simeq \Hom_{\mod{uA}(\mathrm{h}\C)}(uA \otimes X, uA \otimes Y)
\]
as one can verified by noting that both sides are equivalent to $\Hom_{\mathrm{h}\C}(X, uA \otimes Y)$ using the corresponding extension-of-scalars-fortgetful adjunctions. Passing to the idempotent completion, this yields an equivalence $$(\mathrm{h}\mod{A}(\C)^{\mathrm{free}})^{\sharp}\simeq \mod{uA}(\mathrm{h}\C)^{\mathrm{free}})^{\sharp}.$$ It is only left to note that
\[
(\mathrm{h}\mod{A}(\C)^{\mathrm{free}})^{\sharp}=\mathrm{h}\mod{A}(\C) \quad \mathrm{and} \quad  \mod{uA}(\mathrm{h}\C)^{\mathrm{free}})^{\sharp} = \mod{uA}(\mathrm{h}\C)
\]
which follows from Remark~\ref{rem-idempotent} together with the fact that, by separability of $A$, any module is a retract of a free one: see~\cite{Balmer2011}*{Remark 3.9 and Proposition 3.11} for the statement in $\mod{uA}(\mathrm{h}\C)$ and Remark~\ref{rem-sep-counit-section} for the statement in $\mod{A}(\C)$.    
\end{proof}

We finally explain how the previous result relates to work of Balmer~\cite{Balmer2011}.

\begin{remark}
 In the situation of Lemma~\ref{lem-module-over-sep-alg}, there is a commutative diagram 
 \[
 \begin{tikzcd}
  \mathrm{h}\mod{A}(\C) \arrow[dr, "\mathrm{fgt}"']\arrow[rr, "\sim"] & & \mod{uA}(\mathrm{h}\C) \arrow[dl, "\mathrm{fgt}"]\\
   & \mathrm{h}\C.& 
 \end{tikzcd}
 \]
 In particular $\mod{uA}(\mathrm{h}\C)$ inherits a triangulated structure from $\mathrm{h}\mod{A}(\C)$. This triangulated structure coincides with that of~\cite{Balmer2011} by~\cite{Balmer2011}*{Main Theorem 5.17(c)}. Furthermore the extension of scalars functor $F_A \colon \C \to \mod{A}(\C)$ induces a functor $hF_A \colon \mathrm{h}\C \to \mathrm{h}\mod{A}(\C)$ on homotopy categories which identifies via Lemma~\ref{lem-module-over-sep-alg} with the extension of scalars  functor $\mathrm{h}\C \to \mod{uA}(\mathrm{h}\C)$ defined in~\cite{Balmer2011}*{Definition 2.4}.
\end{remark}

\section{Degree functions of separable algebras}\label{sec:degree}

In this section we export the notion of separable algebra of finite tt-degree into the world of $\infty$-categories. We use this to introduce the degree function of a separable algebra and list a few of its important properties. 
%Our approach is somewhat indirect, due to technical reasons: In Theorem \ref{thm-splitting} we transfer a key splitting lemma of Balmer to the realm of {\em large} homotopy theories by mimicking Balmer's proof. We then use this in Construction \ref{construct-splitting-tower} to construct Balmer's splitting tower in large homotopy theories\footnote{This is where the technical obstacle appears: Balmer relies on the fact that in the tt-context, modules over a {\em separable} algebra naturally organize into a tt-category. The analogue for $\infty$-categories is straightforward, but requires to work with large homotopy theories.}. Finally, we observe in Lemma \ref{lem-tt-tower-and-compact} a finiteness property of this construction which will allow for smooth passage between our constructions and those of Balmer.

\begin{theorem}\label{thm-splitting}
 Let $\D\in\CAlg(\Pr)$ be given, and let $f \colon A \to B$ and $g \colon B \to A$ be maps in 
 $\CAlg^\sep(\D)$ such that $gf = 1_A$. 
 Then there exists $C \in \CAlg^\sep(\D)$ and an equivalence in $\CAlg(\D)$
 $h \colon B \xrightarrow{\sim} A \times C$ such that $\mathrm{pr}_1 h = g$. 
 In particular, $C$ becomes an $A$-algebra via $\mathrm{pr}_2 h f$. 
 Moreover, if $C'$ is another $A$-algebra and $h'\colon B \xrightarrow{\sim} A \times C'$ is another 
 equivalence such that $\mathrm{pr}_1h'= g$, then there exists an equivalence of $A$-algebras 
 $\ell \colon C \xrightarrow{\sim} C'$ such that $h' = (1 \times \ell)h$. 
\end{theorem}

\begin{proof}
    This is an elaboration of~\cite{Balmer2014}*{Lemma 2.1}. Let $B^e:=B \otimes B^{\op}$ be the enveloping algebra so that $B^e$-modules are the same as $(B,B)$-bimodule. We equip $A$ with a structure of $(B,B)$-bimodule such that $g$ is $B^e$-linear. We then complete $g$ to a fibre sequence
    \[
    C \to B\xrightarrow{g} A\xrightarrow{z} \Sigma C
    \]
    in $\mod{B^e}(\D)$. Applying the exact functor $U_{B^e}$ to the above fibre sequence and using that $g$ is split by $f$ in $\D$, we deduce that $U_{B^e}(z)=0$ in $\D$. 
    As $B^e$ is separable by Example~\ref{ex-tensor-is-sep}, the forgetful functor $U_{B^e}$ is faithful by Remark~\ref{rem-sep-counit-section}. Thus, $z=0$ in $\mod{B^e}(\D)$ too. This gives an equivalence $h\colon B \xrightarrow{\sim} A\oplus C$ in $\mod{B^e}(\D)$ satisfying $\mathrm{pr}_1 h=g$. Applying the canonical functor $u\colon \D \to \mathrm{h}\D$,
    we obtain a decomposition $u(h) \colon uB \xrightarrow{\sim} uA\oplus uC$ in $\mod{uB^e}(\mathrm{h}\D)$. By~\cite{Balmer2014}*{Lemma 2.2} (noting that the proof does not rely on the tt-category being essentially small), we can equip $uA$ and $uC$ with ring structures such that $u(h)$ is a ring isomorphism. Moreover this new ring structure on $uA$ agrees with the original one: this follows from the fact that $u(g)\colon uB \to uA$ is a split epimorphism which is a homomorphism with respect to both ring structures on $uA$ (the old one by assumption and the new one because $u(h)=\big(\begin{smallmatrix}  u(g) \\  \ast\end{smallmatrix}\big)$ is a homomorphism). Therefore we get a splitting of rings $u(h) \colon uB \simeq uA \times uC$ in $\mathrm{h}\D$ satisfying $\mathrm{pr}_1 u(h)=u(g)$. By applying $\pi_0$ to $u(h)$ and noting that $\pi_0(B)=\pi_0(uB)$ (and similarly for $A$ and $C$), we get a splitting of commutative rings
 \[
 \pi_0(h)\colon \pi_0(B)\simeq \pi_0(A) \times \pi_0(C)
 \]
 satisfying $\mathrm{pr}_1 \pi_0(h) =\pi_0(g)$. 
 Therefore there exists an idempotent $e \in \pi_0(B)$ corresponding to $(1, 0)$ under 
 $\pi_0(h)$. Set $C:=B[(1-e)^{-1}] \in \CAlg(\D)$ and note that $A \simeq B[e^{-1}]$.
 Then by Lemma~\ref{lem-idempotent-splitting} we get a decomposition 
 $h\colon B \simeq A \times  C$ in $\CAlg(\D)$ lifting the decomposition in $\mathrm{h}\D$. 
 By construction, we find that $\mathrm{pr}_1 h \simeq g$. Note that $C$ is separable by 
 Example~\ref{ex-separable-product-retract}. Finally, for uniqueness of $C$ as an $A$-algebra, with the notation of the theorem, we obtain an equivalence $k:=h' h^{-1} \colon \1_A \times C \xrightarrow{\sim} \1_A \times C'$ 
 of commutative algebras in $\mod{A}(\D)$ such that $\mathrm{pr}_1 k\simeq \mathrm{pr}_1$. 
 This means that $k$ has the form 
 \[
 \begin{pmatrix} 1 & 0 \\ s & l\end{pmatrix}.
 \]
 Now we can conclude by~\cite{Balmer2014}*{Lemma 2.3} (again noting that the proof does not rely on the tt-category being essentially small).
\end{proof}

\begin{corollary}\label{cor-sep}
 Let $\D \in \CAlg(\Pr)$ be given and consider $A \in \CAlg^\sep(\D)$ with multiplication map $\mu$ 
 and unit map $\eta$. Then there is an equivalence in $\CAlg(\D)$ $h \colon A \otimes A \xrightarrow{\sim} A \times A'$ for some 
 $A' \in \CAlg^\sep(\D)$ in such a way that $\mathrm{pr}_1 h =\mu$. 
 Moreover, the $A$-algebra $A'$ is unique up to equivalence with this property.
\end{corollary}

\begin{proof}
  Apply Theorem~\ref{thm-splitting} to $B=A \otimes A$ with 
  $g=\mu \colon A \otimes A \to A$ and $f=1_A \otimes \eta \colon A \to A \otimes A$.
\end{proof}

%\begin{remark}\label{rem-sep-sthtpythy}
% We have stated Corollary~\ref{cor-sep} for $2$-rings but we note that the result also applies to stable homotopy theories. 
 %Indeed the main ingredient of Theorem~\ref{thm-splitting} is~\cite{Balmer2014}*{2.1} where the assumption that the triangulated category is essentially small is nowhere used in the proof.\ltodo{better argument, you can take cardinal big enough so that $A$ and $B$ are compact and then argue inside the category of compact objects}
%\end{remark}

We can then construct the splitting tower following~\cite{Balmer2014}.

\begin{construction}\label{construct-splitting-tower}
 Let $\D \in \CAlg(\Pr)$ be given, and consider some $A\in \CAlg^\sep(\D)$. We construct {\em the splitting tower of $A$} in $\CAlg(\D)$
\[
A^{[0]} \to A^{[1]} \to A^{[2]} \to \ldots \to A^{[n]} \to \ldots,
\]
where $A^{[0]}:=\1$, $A^{[1]}:=A$ and the first map in the tower is the unit map  $\eta \colon A^{[0]} \to A^{[1]}$. For $n\geq 1$, we define 
$A^{[n+1]}$ to be $(A^{[n]})'$ in the notation of Corollary~\ref{cor-sep} applied to the commutative 
separable algebra $A^{[n]}$ in $\mod{A^{[n-1]}}(\C)$. Equivalently, the separable commutative $A^{[n]}$-algebra $A^{[n+1]}$ is characterized by the existence of an equivalence $h\colon A^{[n]}\otimes_{A^{[n-1]}}A^{[n]}\simeq A^{[n]}\times A^{[n+1]}$ of $A^{[n]}$-algebras such that $\mathrm{pr_1}h=\mu$. 
%\textcolor{blue}{Maybe the solution is to not lift the tower to $\infty$-categories. So start with $A\in\CAlg^\sep(\C^\dual)$ 
%and then work with $\mathrm{h}\C^\dual$ as Balmer does...The problem that we have if we lift the tower to $\infty$-categories is that we define 
%$A^{[n+1]}$ to be $(A^{[n]})'$ in the %notation of Corollary~\ref{cor-sep} applied to the commutative 
%separable algebra $A^{[n]}$ in $\mod{A^{[n-1]}}(\C)^\dual$. It is not clear to me that $A^{[n+1]}\in\C^\dual$, which is a problem; this is related to the question: does the forgetful functor from modules to $\C$ preserves dualizable if A is separable and dualizable? }
\end{construction}

\begin{remark}\label{rem-tower-functorial}
    The splitting tower is functorial as in ~\cite{Balmer2014}*{Theorem 3.7(a)}, i.e. given a functor $F\colon \D \to \D'$ in $\CAlg(\Pr)$, some $A \in\CAlg^{\mathrm{sep}}(\D)$ and $n\geq 0$, then $F(A^{[n]})\simeq F(A)^{[n]}$ as commutative algebra objects. %This is clear from the uniqueness of the splitting tower.
\end{remark}

\begin{definition}
Let $0 \not =\D \in\CAlg(\Pr)$ be given,  and assume that  $A \in \CAlg^\sep(\D)$. We say that $A$ has \emph{finite degree} $n\geq 0$ if $A^{[n]}\not = 0$ and $A^{[n+1]}=0$. We say that $A$ has {\em infinite degree} if $A^{[n]}\not =0$ for all $n\geq 0$. 
\end{definition}

\begin{example}\label{ex-degree}
 Note that $A$ has degree $0$ if and only if $A\simeq 0$ and that $A$ has degree $1$ 
 if and only if $A$ is nonzero and the multiplication map $\mu\colon A \otimes A \to A$ is an equivalence.  
\end{example}

To establish a finiteness property of the splitting tower, we will need the following standard fact.

\begin{lemma}\label{lem-extend-restrict-compacts}
Assume that $\C\in\twoRing$ and $A\in\CAlg(\C)$ are given, and let
\[ 
F_A\colon \Ind(\C)\rightleftarrows\mod{A}(\Ind(\C)): U_A\]
denote the free-forgetful adjunction. Then $F_A$ and $U_A$ preserves compact objects.
\end{lemma}

\begin{proof} 
To ease the notation let us set $\D:=\Ind(\C)$. 
The first claim then follows from $\Hom_{\mod{A}(\D)}(F_A(X),-)\simeq \Hom_\D(X, U_A(-))$ and the fact that $U_A$ preserves all colimits. 
For the second claim recall that $U_A$ admits a right adjoint $R_A$. We claim that $R_A$ preserves filtered colimits. This will conclude the proof since 
\[
\Hom_\D(U_A(M),-)\simeq\Hom_{\mod{A}(\D)}(M,R_A(-)),
\]
and both functors on the right hand side commute with filtered colimits if $M$ is compact.

To prove the claim we need to show that the canonical map $\colim R_A(X_i)\to R_A(\colim X_i)$ is an equivalence for any filtered diagram $(i\mapsto X_i)$ in $\D$. 
This is equivalent to checking that the canonical map
\begin{equation}\label{eq-free}
   \Hom_{\mod{A}(\D)}(F_A(G_i), \colim R_A(X_i))\to \Hom_{\mod{A}(\D)}(F_A(G_i), R_A(\colim X_i)) 
\end{equation}
is an equivalence for $\{G_i\}$ a set of compact generators of $\C$, as $\{F_A(G_i)\}$ provides a set of compact generators for $\mod{A}(\D)$. Using that $F_A(G_i)$ is compact in $\mod{A}(\D)$ 
and the adjunction $(U_A,R_A)$, we can rewrite the source of (\ref{eq-free}) as $\colim \Hom_{\D}(A\otimes G_i, X_i)$. This agrees with the target of (\ref{eq-free}) using the adjunction and the fact that $A\otimes G_i\in \C=\D^\omega$. Therefore the map (\ref{eq-free}) is an equivalence.
\end{proof}

\begin{remark}
    It follows from Lemma~\ref{lem-extend-restrict-compacts} that there is an induced free-forgetful adjunction 
    \[ 
F_A\colon \C=\Ind(\C)^\omega\rightleftarrows\mod{A}(\Ind(\C))^\omega: U_A.
\] 
\end{remark}

\begin{lemma}\label{lem-tt-tower-and-compact}
Let $\C\in\twoRing$ and $A\in\CAlg^{\mathrm{sep}}(\C)$ be given. Then we have $A^{[n]}\in \C$
for all $n\ge 0$.
\end{lemma}

\begin{proof}
Note that a priori we have $A^{[n]}\in\Ind(\C)$ by construction.
We claim more strongly that for all $n\ge 0$ and $0\leq i\leq n-1$ we have $A^{[n]}\in\mod{A^{[i]}}(\Ind(\C))^\omega$, omitting all forgetful functors from the notation. The stated claim will then be the cases with $i=0$.
The cases $n=i=0$ and $n=1$, $i=0$ are clear since $A^{[0]}=\1, A^{[1]}=A \in \C=\Ind(\C)^\omega$.
We now proceed by induction on $n\ge 2$:
By construction, we have a splitting
\[ A^{[n-1]}\otimes A^{[n-1]} \simeq A^{[n-1]} \times A^{[n]}\mbox{ in }\mod{A^{[n-2]}}(\Ind(\C)).\]
This shows that $A^{[n]}$ is a retract of the compact object on the left hand side (using the induction hypothesis and Lemma~\ref{lem-extend-restrict-compacts}). Applying forgetful
functors and Lemma \ref{lem-extend-restrict-compacts}, we obtain the claim for $0\leq i\leq n-2$. The above splitting is also $A^{[n-1]}$-linear, for the $A^ {[n-1]}$-action through the left tensor factor and the map $A^{[n-1]}\to A^{[n]}$ in the splitting tower. Since the left hand side is a compact $A^{[n-1]}$-module by base-change, so is its retract $A^{[n]}$. This settles the induction and the proof.
\end{proof}

This allows us to adapt the construction of the splitting tower to $2$-rings.

\begin{construction}\label{con-tt-tower-2ring}
Let $\C\in\twoRing$ and $A\in\CAlg^{\mathrm{sep}}(\C)$ be given. Then the {\em splitting tower of $A\in\C$} is obtained by embedding $\C\subseteq\D$ into its ind-completion as the compact objects, applying Construction \ref{construct-splitting-tower} to $A\in\CAlg^{\mathrm{sep}}(\D)$, and noting that the resulting tower takes values in $\C\subseteq\D$ by Lemma \ref{lem-tt-tower-and-compact}. 
\end{construction}

We note the following application for future reference.

\begin{lemma}\label{lem-tt-tower}
    Let $\C\in\twoRing$ and $A \in\CAlg^{\mathrm{sep}}(\C)$ be given. Then there is an equivalence of tt-categories
    $\mod{uA}(\mathrm{h}\C) \simeq\mathrm{h}\mod{A}(\Ind(\C))^\omega$.
\end{lemma}

\begin{proof}
    Recall from Lemma~\ref{lem-module-over-sep-alg} that there is an equivalence of tt-categories 
    \[
    \mathrm{h}\mod{A}(\Ind(\C))\simeq \mod{uA}(\mathrm{h}\Ind(\C))
    \]
    and that $\Ind(\C)^\omega=\C$.
    Thus, we only need to check that $M\in\mod{uA}(\mathrm{h}\Ind(\C))$ is compact if and only if $U_AM\in\mathrm{h}\Ind(\C)$ is compact. The forward implication follows from Lemma~\ref{lem-extend-restrict-compacts}. The backward implication follows from the same lemma and separability as $M$ is a retract of $A\otimes U_A M$.
\end{proof}

\begin{remark}
    It follows from Lemma~\ref{lem-tt-tower} that the tower of Construction~\ref{con-tt-tower-2ring} lifts the splitting tower of $uA\in\CAlg^{\mathrm{sep}}(\mathrm{h}\C)$ constructed by Balmer~\cite{Balmer2014}.
\end{remark}

We now list some useful properties that the degree satisfies. 

\begin{proposition}\label{prop-properties-degree}
 Let $\C \in \twoRing$ and consider $A,B \in \CAlg(\C)$ with $A$ separable. Denote by $F_B\colon \C \to \mod{B}(\Ind(\C))^\omega$ 
 the functor $X \mapsto B \otimes X$. 
 \begin{itemize}
 \item[(a)] We have $\deg(F_B(A))\leq \deg(A)$. This is an equality if $B$ is faithful or $B$ is separable and $A$ is a $B$-algebra.
 \item[(b)] We have $\deg(\1^{\times n})=n$ for all $n\geq 1.$
 \end{itemize}
\end{proposition}

\begin{proof}
 For part (a) see ~\cite{Balmer2014}*{Theorem 3.7(a) and (b)} and ~\cite{Bregje}*{Proposition 3.5(c)}. Part (b) is proved in ~\cite{Balmer2014}*{Theorem 3.9(a)}.
\end{proof}

\begin{definition}\label{def-degree}
 Let $\C$ be a $2$-ring and consider $A\in \CAlg^\sep(\C)$. 
 For any prime $\p \in \Spc(\C)$, we can consider the local category $\C_{\p}$ which is defined 
 as the idempotent completion of the Verdier quotient $\C/\p$. 
 There is a canonical symmetric monoidal exact functor $q_{\p}\colon \C \to \C_\p$. We define 
 the $\p$-\emph{local degree} of $A$ to be the degree of $q_\p(A)\in\CAlg^{\mathrm{sep}}(\C_\p)$. We can assemble the local degrees 
 into a function 
 \[
 \deg(A)\colon \Spc(\C) \to \overline{\Z}=\Z \cup \{\infty\}, \qquad \p \mapsto \deg(A)(\p):= \deg(q_{\p}(A)).
 \]
 We say that that the degree function is:
 \begin{itemize}
 \item \emph{finite} if $\deg(A)(\p)$ is finite for all $\p \in \Spc(\C)$. By~\cite{Balmer2014}*{Theorem 3.8} this is equivalent to $A$ having finite degree. Furthermore we have the formula  
 \begin{equation}\label{eq-local-degree}
 \deg(A)=\max_{\p\in \supp(A)} \deg(A)(\p).
 \end{equation}
 \item \emph{locally constant} if the degree function is constant on a suitable open neighbourhood of every point of $\Spc(\C)$.
 \end{itemize}
\end{definition}

\begin{remark}
 The category $\C_{\p}$ is local in the sense that $X \otimes Y =0$ implies $X=0$ or $Y=0$, 
 see~\cite{Balmer2010}*{Proposition 4.2, Example 4.3}. This is also equivalent to $\Spc(\C_\p)$ having exactly one close 
 point. In fact, $\Spc(\C_\p)$ is homeomorphic to the subspace 
 $\{\q \mid \p \subseteq \q \}\subseteq \Spc(\C)$ by~\cite{Balmer}*{Proposition 3.11, Corollary 3.14}.
\end{remark}

The next result shows that separable algebras of finite degree are closed under finite 
products. 

\begin{lemma}\label{lem-product-fin-degree}
 Let $\C$ be a $2$-ring and let $A= \prod_{i=1}^n A_i \in \CAlg(\C)$ be separable. Then \[\deg(A)(\p)= \sum_{i=1}^n \deg(A_i)(\p)\mbox{ for all }\p\in\Spc(\C).\] 
 In particular, $A$ has finite degree if and only if $A_i$ has finite degree for all $i$. 
\end{lemma}

\begin{proof}
 For all $\p \in \Spc(\C)$, the functor $q_\p \colon \C \to \C_\p$ is 
 exact so $q_{\p}(A)= \prod_{i=1}^n q_{\p}(A_i)$. The claim that $\deg(A)(\p)= \sum_{i=1}^n \deg(A_i)(\p)$ for all $\p$ then follows from ~\cite{Balmer2014}*{Corollary 3.12}. The second claim follows from the first one and (\ref{eq-local-degree}).
\end{proof}

\begin{example}\label{ex-locally-constant-degree-function}
 Let $\C$ be a $2$-ring and consider the commutative algebra $\1^{\times n}$ for 
 some integer $n \geq 1$. Then the degree function  $\deg(\1^{\times n})\colon \Spc(\C) \to \Z$ 
 is constant with value $n$. 
 Indeed for all $\p \in \Spc(\C)$, the functor $q_\p \colon \C \to \C_\p$ is 
 exact and symmetric monoidal so 
 \[
 q_\p(\1^{\times n})=q_\p(\1)^{\times n} = \1_\p^{\times n}
 \]
 where $\1_\p$ denotes the unit object of $\C_\p$. Therefore the claim follows from 
 Proposition~\ref{prop-properties-degree}(b). 
\end{example}

We can obtain more examples of separable algebras with locally constant and finite degree function as follows. 

\begin{lemma}\label{lem-locally-constant}
 Let $\C$ be a $2$-ring and consider $\1[e^{-1}]\in \CAlg(\C)$ for an idempotent $e$ of $\1$. Then 
 \[
 \deg(\1[e^{-1}])(\p)=
 \begin{cases}
 1 & \mathrm{if} \;  \p\in \supp(\1[e^{-1}] )\\
 0 & \;\mathrm{otherwise.}
 \end{cases}
 \] 
 In particular, the degree 
 function $\deg(\1[e^{-1}])\colon \Spc(\C) \to \Z$ is locally constant.  More generally if $e_1, \ldots, e_n$ are idempotent elements of $\1$, then $\prod_{i=1}^n \1[e_i^{-1}]$ has finite degree and its degree function is locally constant.
\end{lemma}

\begin{proof}
 Recall that $\p \in \supp(\1[e^{-1}])$ if and only if $\1[e^{-1}]\not\in \p$.
 Clearly, if $\1[e^{-1}]\in \p$ then its image is zero in $\C_\p$ and so it has zero degree. 
 Now suppose that $\1[e^{-1}]\not \in \p$ so that $q_\p(\1[e^{-1}]) \not =0$. Then
 \[
 0 < \deg(q_\p(\1[e^{-1}]))\leq \deg(\1)=1
 \]
 using Proposition~\ref{prop-properties-degree}.
 
 The decomposition $\1\simeq\1[e^{-1}]\times \1[(1-e)^{-1}]$ induces a decomposition of 
 $\C$ as $\C_0 \times \C_1$ which in turns gives a decomposition 
 \[
 \Spc(\C_0) \sqcup \Spc(\C_1) \simeq \Spc(\C),
 \]
 by Lemma~\ref{lem-balmer-spectrum-decomposes}. 
 The previous paragraph shows that the degree function is constant and equal to $1$ on 
 $\Spc(\C_0)$, and equal to $0$ on $\Spc(\C_1)$. 
 In particular, the degree function is locally constant. 
 
 If we have several idempotent elements of $\1$, a similar argument as above gives a decomposition $\Spc(\C)\simeq \bigsqcup_i \Spc(\C_i)$ where each $\deg(\1[e_i^{-1}])$ is constant with value $1$ in $\Spc(\C_i)$ and zero otherwise. By Lemma~\ref{lem-product-fin-degree}, we have $\deg(\prod_i \1[e_i^{-1}])(\p)=\sum_i \deg(\1[e_i^{-1}])(\p)$ which is then finite and locally constant.
\end{proof}

The next result makes explicit the relation between $\supp(A^{[d]})$ and the degree function $\deg(A)$. (The notion of descendability used here is reviewed in Section~\ref{sec:axiomatic_galois}).

\begin{lemma}\label{lem-descend-degree-geq1}
 Let $\C$ be a $2$-ring and consider $A\in \CAlg^\sep(\C)$. Then for all $\p \in \Spc(\C)$ 
 and $d \in \Z_{\geq 0}$ we have $\p \in \supp(A^{[d]})$ if and only if $d \leq \deg(A)(\p)$. 
 Moreover if $\C$ is rigid (which means that $\C=\C^\dual$), then $A$ is descendable in 
 $\C$ if and only if $\deg(A)(\p)\geq 1$ for all $\p \in \Spc(\C)$. 
\end{lemma}

\begin{proof}
 Recall that $\p \in \supp(A^{[d]})$ if and only if $A^{[d]} \not \in \p$. This is equivalent to 
 $0 \not = q_\p(A^{[d]})=q_\p(A)^{[d]}$ by the functoriality of the splitting tower, see Remark~\ref{rem-tower-functorial} . Finally 
 $q_\p(A)^{[d]} \not =0$ if and only if $\deg(A)(\p) \geq d$. 
   
 For the second claim note that $A$ is descendable if and only if $\supp(A)=\Spc(\C)$, see Lemma~\ref{lem-desc}. Since $A=A^{[1]}$, the previous 
 paragraph tells us that this equivalent to $\deg(A)(\p)\geq 1$ for all primes $\p$.
\end{proof}
\begin{comment}
As an immediate consequence we find that:
\begin{proposition}\label{prop-deg-func-up-semi-cont}
    Let $\C$ be a $2$-ring and consider $A\in\CAlg^\sep(\C)$. Then the degree function
    \[
    \deg(A)\colon \Spc(\C) \to \overline{\mathbb{R}}
    \]
    is upper semicontinuous. Moreover for all $\p,\q \in\Spc(\C)$ with $\p \subseteq \q$, we have $\deg(A)(\p)\geq \deg(A)(\q)$. \ltodo{does the second claim follows from the first one?}
\end{proposition}

\begin{proof}
    By definition $\deg(A)$ is upper semicontinuous if for all $r\in \mathbb{R}$ the superlevel sets
    \[
    U_r=\{\p \in \Spc(\C) \mid \deg(A)(\p)\geq r\}
    \]
    are closed. Without loss of generality we can assume that $r\in \Z$ since $\deg(A)$ takes values in $\overline{\Z}$. Now $U_r=\supp(A^{[r]})$ by Lemma \ref{lem-descend-degree-geq1} and so it is closed. 
    For the second claim suppose that $\p \subseteq \q$ so that there is a tt-functor $\C/\p \to \C/\q$ on Verdier quotients. This induced a tt-functor on local categories $\pi\colon \C_{\p} \to \C_{\q}$ which fits into a commutative diagram 
    \[
    \begin{tikzcd}
        & \C \arrow[dl, "q_\p"'] \arrow[dr,"q_{\q}"] & \\
        \C_{\p}\arrow[rr,"\pi"] & & \C_q.
    \end{tikzcd}
    \]
    By the commutativity of the diagram above we find that 
    \[
    \deg(A)(\q)=\deg(q_{\q}(A))=\deg(\pi q_{\p}(A))\leq \deg(q_\p(A))=\deg(A)(\p)
    \]
    where for the inequality we used \cite{Balmer2014}*{Theorem 3.7(a)}.
\end{proof}
\end{comment}
We note the following functoriality of the degree function.

\begin{proposition}\label{prop-diagram commutes}
 Let $F\colon \C \to \D$ be a symmetric monoidal exact functor between $2$-rings and let $\varphi \colon \Spc(\D) \to \Spc(\C)$  be 
 the induced map between Balmer spectra. For every $A \in \CAlg^\sep(\C)$ 
 and $\p \in \Spc(\D)$, we have 
 \[
 \deg(F(A))(\p)=\deg(A)(\varphi(\p)).
 \]
\end{proposition}

\begin{proof}
 For all $\p \in \Spc(\D)$ we have a commutative diagram 
 \[
 \begin{tikzcd}
 \C \arrow[r, "F"] \arrow[dd, bend right=80,"q_{\varphi(\p)}"']\arrow[d] & \D \arrow[d]\arrow[dd, bend left=80, "q_{\p}"]\\
 \C/\varphi(\p) \arrow[d] \arrow[r,"\overline{F}"] & \D/\p \arrow[d]\\
 \C_{\varphi(\p)} \arrow[r,"\overline{F}^\sharp"] & \D_\p.
 \end{tikzcd}
 \]
 Here we used the universal property of idempotent completion and the fact that $F(\varphi(\p))\subseteq \p$. 
 We need to check that $\deg(q_{\varphi(\p)}(A))=\deg(q_\p(F(A)))$. 
 We claim that $\overline{F}^\sharp$ is conservative so that 
 by~\cite{Balmer2014}*{Theorem 3.7(b)} we find 
 \[
 \deg(q_{\varphi(\p)}(A))=\deg(\overline{F}^\sharp(q_{\varphi(\p)}(A)))=\deg(q_\p(F(A)))
 \]
 as required.
 To prove the claim let us first show that $\overline{F}$ is conservative. Pick 
 $\overline{X} \in \C/\varphi(\p)$ such that $\overline{F}(\overline{X})\simeq 0$.   
 Choose $X \in \C$ such that its image under $\C \to \C/\varphi(\p)$ is 
 $\overline{X}$. Note that $\overline{F}(\overline{X})\simeq 0$ if and only if 
 $F(X)\in \p$, which means that $\p \not \in \supp(F(X))$.  
 But $\supp(F(X))=\varphi^{-1}(\supp(X))$ by~\cite{Balmer}*{Proposition 3.6} so we deduce that 
 $\varphi(\p) \not \in \supp(X)$. In other words, $X \in \varphi(\p)$ and so $\overline{X}\simeq 0$ 
 as required. 
 
 Finally, consider an object of $\C_{\varphi(\p)}$, which is a pair $(\overline{X}, e)$ where 
 $\overline{X}\in \C/\varphi(\p)$ and $e$ is an idempotent of $\overline{X}$. Suppose that its 
 image under $\overline{F}^\sharp$ is equivalent to the zero object. This means that 
 $\overline{F}(\overline{X})\simeq 0$ and $\overline{F}(e)$ is an equivalence. 
 By conservativity of $\overline{F}$ we conclude that the pair $(\overline{X}, e)$ is equivalent to 
 zero in $\C_{\varphi(\p)}$ as claimed. 
\end{proof}

\section{Descent for separable commutative algebras}\label{sec:descent-separable}

In this section we prove that the formation of commutative separable algebras respects all limits of stable homotopy theories and deduce some consequences for the degree function. We start by giving a characterization of separability.  

\begin{definition}
    Let $\C\in\CAlg(\Pr)$ be given  and consider $A\in\CAlg(\C)$ with multiplication map $\mu$. 
    We define a space $s(A)$ via the pullback square
    \[
    \begin{tikzcd}
        s(A) \arrow[r] \arrow[d] & \Hom_{\mod{A\otimes A}(\C)}(A, A\otimes A) \arrow[d,"\mu_*"] \\
        * \arrow[r,"\mathrm{id}_A"] & \Hom_{\mod{A\otimes A}(\C)}(A,A).
    \end{tikzcd}
    \]
    Clearly, $s(A)$ is non-empty if and only if $A$ is separable.
\end{definition}

We now show that the bimodule section witnessing separability is essentially unique. We emphasize that this is only true for {\em commutative} algebras, see for example~\cite{Kadison}*{Example 1.5}.

\begin{lemma}\label{lem-unique-section}
The commutative algebra $A\in\CAlg(\C)$ is separable if and only if $s(A)$ is contractible. 
\end{lemma}

\begin{proof}
If $s(A)$ is contractible, then it is non-empty and so a bimodule section for the multiplication map of $A$ exists, i.e. $A$ is separable.\\
Conversely, suppose that $A$ is separable with section $s$ and let $J$ denote the fibre of the multiplication map $\mu \colon A \otimes A \to A$ formed in $(A,A)$-bimodules.
Then consider the following diagram of spaces

\[
\begin{tikzcd}
  \Hom_{\mathrm{Mod}_{A\otimes A}(\C)}(A,A\otimes A) \arrow[rr, "\mu_*"]\ar[d,"+s"] & & \Hom_{\mathrm{Mod}_{A\otimes A}(\C)}(A,A)\ar[d, "+\mathrm{id}_A"]\\
   \Hom_{\mathrm{Mod}_{A\otimes A}(\C)}(A,A\otimes A) \arrow[rr, "\mu_*"] & & \Hom_{\mathrm{Mod}_{A\otimes A}(\C)}(A,A) 
 \end{tikzcd}
 \]
in which the vertical maps are equivalences, using the $H$-space structure of the mapping spaces. This is a commutative square since for all $f\in   \Hom_{\mathrm{Mod}_{A\otimes A}(\C)}(A,A\otimes A)$, we have  
\[
\mu\circ (f+s)=\mu \circ f + \mu \circ s=\mu \circ f + \mathrm{id}_A
\]
using that $s$ is a section for $\mu$.
By taking horizontal fibres we obtain the equivalence
\begin{equation}\label{s(A)}
\Hom_{\mod{A\otimes A}(\C)}(A,J)\xrightarrow{\simeq}s(A).
\end{equation}
Using Corollary~\ref{cor-sep}, we can identify
\[ 
\Hom_{\mathrm{Mod}_{A\otimes A}(\C)}(A,J)\simeq
\Hom_{\mathrm{Mod}_{A\times A'}(\C)}(A,A'),
\]
which is contractible since $A$ and $A'$ are orthogonal. 
\end{proof}

\begin{proposition}\label{prop-descent-sep}
   Let $K$ be a simplicial set and $p \colon K \to \CAlg(\Pr)$ be a functor.
   Then the canonical functors induced by the projections
   \[
   \CAlg^{\sep}(\lim_K p) \xrightarrow{\sim} \lim_{k\in K} \CAlg^{\sep}(p(k))
   \]
and 
 \[
 \CAlg^\sep((\lim_K p)^\dual) \xrightarrow{\sim}  \lim_{k\in K}\CAlg^\sep(p(k)^\dual)
 \]
are equivalences.
   \end{proposition} 

\begin{proof} 
The second isomorphism follows from the first and the fact (\cite{HA}*{Proposition 4.6.1.11}) that the formation of dualizable objects commutes over limits.

Let us set $\C=\lim_K p$ and $\C_k=p(k)$ so that $\C=\lim_k \C_k$. We note that the first functor is fully faithful as it is a subfunctor of the equivalence  of Lemma~\ref{lem-calg-limits}. 
Now consider $A \in \CAlg(\C)$ such that all images $A_k\in\CAlg^\sep(\C_k)$ are separable. To conclude the proof, we need to show that $A$ is separable. We note that there is an equivalence $\mod{A\otimes A}(\C)\simeq \lim_k \mod{A_k \otimes A_k}(\C_k)$ by Lemma~\ref{lem-module-limits}. Thus we can calculate the mapping spaces in the limit category rather than in $\mod{A\otimes A}(\C)$. Since mapping spaces in limits are limits, (\ref{s(A)}) shows that the canonical map
\[
s(A)\xrightarrow{\sim}\lim_k s(A_k)
\]
is an equivalence. Since all $s(A_k)$ are contractible the limit $s(A)$ is non-empty (in fact, contractible by Lemma \ref{lem-unique-section}), and $A$ is separable.
\end{proof}

\begin{example}
Let $G$ be a finite group and consider a stable homotopy theory $\C$. Since $\Fun(BG, \C)=\lim_{BG}\C$ we immediately deduce that 
\[
\CAlg^{\sep}(\Fun(BG,\C))=\Fun(BG, \CAlg^{\sep}(\C)).
\]
\end{example}

We deduce some consequences for the degree of a commutative algebra in a limit.
\begin{corollary}
     Let $K$ be a simplicial set and $p \colon K \to \CAlg(\Pr)$ be a functor. Consider $A\in\CAlg^{\sep}((\lim_K p)^\dual)$ and for each $k\in K$, let $A_k\in\CAlg^{\mathrm{sep}}(p(k)^{\mathrm{dual}})$ be its image under the projection functor $\lim_K p \to p(k)$. Then 
     \[ \deg(A)=\sup_{k\in K}(\deg(A_k))\in\N\cup\{ \infty \}.\]
     
\end{corollary}

\begin{proof}
 The functoriality of the splitting tower implies that, for each $n\ge 0$, we have $A^{[n]}=\{A_k^{[n]}\}_k$. Since an object in the limit is zero if and only if it is so at every vertex, the claim follows.   
\end{proof}

\begin{corollary}\label{cor-bounded-degree}
      Let $K$ be a simplicial set and $p \colon K \to \CAlg(\Pr)$ be a functor.
      Then the equivalence 
       \[
 \CAlg^\sep((\lim_K p)^\dual) \xrightarrow{\sim}  \lim_{k\in K}\CAlg^\sep(p(k)^\dual)
 \]
   from Proposition~\ref{prop-descent-sep} identifies algebras of finite degree with families of algebras of bounded degree.   
    
    %The essential image under the fully faithful functor
     % \[
      % \CAlg^{\sep}((\lim_K p)^\dual) \to \lim_k \CAlg^{\sep}(p(k)^\dual) 
      %\]
      %of the full subcategory of the source spanned by the separable algebras of finite degree coincides with the full subcategory of the target spanned by the commutative algebras with bounded degree.
\end{corollary}

\section{Axiomatic Galois theory}\label{sec:axiomatic_galois}

In this section we recall the necessary background on axiomatic Galois theory following~\cite{Mathew2016}. We start by giving the definition of a Galois category and listing a few simple examples. We then introduce the notion of a (weak) finite cover which will be of key importance throughout this paper and recall that they can be organized into a Galois category. We also explain how to associate a profinite groupoid to any such  category via the  Galois correspondence. We end this section by discussing $G$-torsors in a Galois category and showing that Galois extensions are examples of such.

\begin{definition}
A \emph{Galois category} is a category $\C$ such that 
\begin{itemize}
    \item[(a)] $\C$ admits finite limits and coproducts, and the initial object $\emptyset$ is empty (see~\cite{Mathew2016}*{Definition 5.11});
    \item[(b)] coproducts are disjoint and distributive in $\C$ (see~\cite{Mathew2016}*{Definition 5.12, Definition 5.13});
    \item[(c)] Given an object $x$ in $\C$, there is an effective descent morphism~\cite{Mathew2016}*{Definition 5.14}
    $x' \to *$ into the terminal object of $\C$ and a decomposition $x'= x_1' \coprod \ldots \coprod x_n'$ such that each map $x \times x_i' \to x_i'$ 
    is isomorphic to the fold map $x \times x_i' \simeq \coprod_{S_i} x_i' \to x_i'$ for a finite set $S_i$.
\end{itemize}
The collection of Galois categories and functors between them (which are required to preserve coproducts, effective descent morphisms, and finite limits) can be organized into a $2$-category $\GalCat$. 
\end{definition}

\begin{example}\label{ex-finset_G}
Let $G$ be a finite group. The category $\FinSet_G$ of finite sets with a left $G$-action is a Galois category by~\cite{Mathew2016}*{Example 5.30}. The only axiom that requires verification is (c). Given a finite $G$-set $T$, there is an effective descent morphism $G \to \ast$ such that  $T \times G$ decomposes in $\FinSet_G$ as a disjoint union of finite copies of $G$ (since it is free).
\end{example}

\begin{example}\label{ex-fin-gpd}
 Let $\G$ be a finite groupoid. The category $\Fun(\G, \FinSet)$ is a finite product of Galois categories so it is again a Galois category.
\end{example}

Before giving the main examples of Galois categories, we will need to introduce some terminology. Recall from~\cite{Mathew2016}*{Definition 3.18} that $A \in\CAlg(\C)$ is called {\em descendable} if the thick tensor ideal generated by $A$ is all of $\C$. We give the following useful characterization of descendability which we will use repeatedly.  

\begin{lemma}\label{lem-desc}
 Let $\C$ be a $2$-ring which is rigid, i.e we have $\C=\C^\dual$. Then, the following are equivalent for $A\in \CAlg(\C)$:
 \begin{itemize}
     \item[(a)] $A$ is descendable in $\C$;
     \item[(b)] $\supp(A)=\Spc(\C)$.
 \end{itemize}
\end{lemma}
\begin{proof}
$A$ is descendable if and only if $A$ and $\1$ generate the same 
 thick ideal in $\C$. By the classification of thick ideals~\cite{Balmer}*{Theorem 4.10} (together with~\cite{Balmer2}*{Proposition 2.4}) this is equivalent to $\supp(A)=\Spc(\C)$. 
\end{proof}

 We are finally ready to introduce (weak) finite covers following~\cite{Mathew2016}*{Definitions 6.1 and 6.2}.
 
 \begin{definition}
 Let $\C\in\CAlg(\Pr)$ and some $A\in \CAlg(\C)$ be given.
 \begin{itemize}
 \item We say that $A$ is a \emph{weak finite cover} if there exists $A'\in \CAlg(\C)$ 
 such that 
 \begin{itemize}
 \item[(i)] $A'$ is faithful (i.e. $A'\otimes-:\C\to\C$ is faithful).
 \item[(ii)] $A'\otimes -\colon \C \to \C$ commutes with all limits.
 \item[(iii)] there is an equivalence 
 \[
 A \otimes A' \simeq \prod_{i=1}^n A'[e_{i}^{-1}] \qquad \mathrm{in} \;\;\CAlg(\mod{A'}(\C))
 \]
 where each $e_i$ is an idempotent of $A'$.\\
 The weak finite covers span a full subcategory 
 $\CAlg^\wcov(\C)\subseteq \CAlg(\C)$.
 \end{itemize}
 \item We say that $A$ is a \emph{finite cover} if there exists $A'\in\CAlg(\C)$  
 such that 
 \begin{itemize}
 \item[(i)] $A'$ is descendable in $\C$;
 \item[(ii)] there is an equivalence 
 \[
 A \otimes A' \simeq \prod_{i=1}^n A'[e_{i}^{-1}] \qquad \mathrm{in} \;\;\CAlg(\mod{A'}(\C))
 \]
 where each $e_i$ is an idempotent of $A'$.\\
 The finite covers span a full subcategory $\CAlg^\cov(\C)\subseteq \CAlg(\C)$.
 \end{itemize}
 \end{itemize}
 \end{definition}

As already mentioned at the beginning of this section, (weak) finite covers form a Galois category. 
\begin{proposition}\label{prop-wcov=cov}
 Let $\C\in\CAlg(\Pr)$ be given. Then, the $\infty$-categories
 \[
 \CAlg^\cov(\C)^{\op} \quad \mathrm{and}\quad \CAlg^\wcov(\C)^{\op}
 \]
 are Galois categories. We always have the containment $\CAlg^\cov(\C)\subseteq \CAlg^\wcov(\C)$, and this is an equality if $\1\in \C$ is compact.  
 Furthermore, any functor $F \colon \C \to \D$ in $\CAlg(\Pr)$ induces functors 
 \[
 \CAlg^\wcov(\C) \to \CAlg^\wcov(\D) \quad \mathrm{and} \quad \CAlg^\cov(\C) \to \CAlg^\cov(\D).
 \]
\end{proposition}

\begin{proof}
    The first two claims are proved in~\cite{Mathew2016}*{Theorem 6.5}. Implicitly, we are also claiming that the $\infty$-categories of (weak) finite covers are in fact 1-categories, which is consequence of~\cite{Mathew2016}*{Proposition 5.28}. The final claim is~\cite{Mathew2016}*{Proposition 6.6}.
\end{proof}

We now turn to discuss the Galois correspondence. To this end, we note that the collection of finite groupoids, functors and natural transformations can be organized into a $2$-category which we denote by $\Gpd_{\fin}$. 
Using Example~\ref{ex-fin-gpd}, we see that there is a functor 
\begin{equation}\label{eq:Galcor}
\Gpd_{\fin}^{\op}\to \GalCat, \quad \G \mapsto \Fun(\G, \FinSet).
\end{equation}
Now by~\cite{Mathew2016}*{Proposition 5.34}, $\GalCat$ admits filtered colimits and these are calculated at the level of underlying categories. 
It follows that (\ref{eq:Galcor}) uniquely extends to a colimit preserving functor 
\begin{equation}\label{functor}
\mathrm{Pro}(\Gpd_{\fin})^{\op}\simeq \Ind(\Gpd_{\fin}^{\op}) \to \GalCat. 
\end{equation}
\begin{theorem}[Galois correspondence]\label{thm-Galois-correspondence}
 The functor~(\ref{functor}) is an equivalence of $2$-categories.
\end{theorem}
\begin{proof}
 See~\cite{Mathew2016}*{Theorem 5.36}.
\end{proof}
Therefore we can make the following definition following~\cite{Mathew2016}*{Definition 6.8}.

\begin{definition}
 Let $\C$ be a stable homotopy theory. The \emph{Galois groupoid} $\pi_{\leq 1}(\C)$ of $\C$ is the profinite groupoid associated to the 
 Galois category $\CAlg^\cov(\C)^{\op}$ via the Galois correspondence. 
 The \emph{weak Galois groupoid} $\pi_{\leq 1}^\weak(\C)$ of $\C$ is the profinite groupoid associated 
 to the Galois category $\CAlg^\wcov(\C)^{\op}$. 
\end{definition}

\begin{remark}\label{rem-galois-connected}
 If $\C$ is connected, these Galois groupoids can be 
 represented by profinite groups, denoted by $\pi_1(\C)$ and $\pi_1^\weak(\C)$.
\end{remark}

For a profinite group $G$, we let $\FinSet_G^{\cts}$ denote the category of finite discrete sets with a continuous left $G$-action, i.e., a $G$-action that factors through $G/U$ for $U$ an open normal subgroup (which is automatically of finite index). In other words, \[\FinSet_G^{\cts} =\colim_{U \subseteq G}\FinSet_{G/U}.\]
By specializing the Galois correspondence to the case where $\C$ is connected and using Remark~\ref{rem-galois-connected}, we obtain the following result.
 
\begin{corollary}\label{cor-con-galois-cor}
Let $\C$ be a connected stable homotopy theory. Then there 
are equivalences of Galois categories 
\[
\CAlg^\cov(\C)^{\op} \simeq \FinSet^\cts_{\pi_1(\C)} \quad \mathrm{and} \quad \CAlg^\wcov(\C)^{\op} \simeq \FinSet^\cts_{\pi_1^\weak(\C)}.
\]
\end{corollary}

We finish this section by recalling the notion of $G$-torsor in a Galois category and list some examples. 

\begin{definition}[\cite{Mathew2016}*{Definition 5.31}]
Let $\C$ be a Galois category and let $G$ be a finite group. A $G$-\emph{torsor} in $\C$ is an object $x\in \C$ with a $G$-action such that there exists an effective descent morphism $y \to \ast$ such that $y\times x \to y \in \C_{/y}$, as an object with $G$-action, is given by 
\begin{equation}\label{G-torsor}
y \times x \simeq \coprod_G y
\end{equation}
where $G$-acts on the latter by permuting the summands. 
\end{definition}

\begin{example}\label{ex-G-torsor-G}
The Galois category $\FinSet_G$ admits a $G$-torsor which is given by $x=G$. Indeed the canonical right $G$-action makes $x$ into an object with $G$-action in $\FinSet_G$. Consider $y=G\in \FinSet_G$ with its canonical left $G$-action. 
We have already noted that there is an effective descent morphism $y \to \ast$ in $\FinSet_G$, and there is a canonical map
\[
\coprod_G y \to y \times x, \qquad y_g \mapsto ( y_g, g)
\]
which is a $G$-equivariant bijection with respect to both the left and right $G$-actions on the source and on the target. Therefore it provides an equivalence of $G$-objects in $\FinSet_G$.  
In fact we can say a little more: up to isomorphism $G$ is the unique $G$-torsor which is indecomposable in $\FinSet_G$. To see this consider a $G$-torsor $x$ and an effective descent morphism $y \to \ast$ such that~(\ref{G-torsor}) holds. If $x$ is indecomposable in $\FinSet_G$, then $x\simeq G/H$ for some subgroup $H\subseteq G$. However, the equivalence~(\ref{G-torsor}) implies that $G$ acts freely on $x$, forcing $H=1$. 
\end{example}

In order to give other examples of $G$-torsors, we recall the following definition.

\begin{definition}\label{def-galois-extension}
 Let $\C\in\CAlg(\Pr)$ be given. An object $A\in\CAlg(\C)$ with an action of a finite group $G$ in $\CAlg(\C)$ is a 
 \emph{$G$-Galois extension} if:
 \begin{itemize}
 \item The canonical map $ \1 \to A^{hG}$ is an equivalence;
 \item The map $h\colon A\otimes A \to \prod_{G} A$ in $\mod{A}(\C)$, given informally by 
 $a_1 \otimes a_2\mapsto (a_1 g(a_2))_{g\in G}$, is an equivalence.
 \end{itemize}
 We say that a $G$-Galois extension $A$ is \emph{faithful} if furthermore $A$ is faithful.
\end{definition}

The following result provides more examples of $G$-torsors. 

\begin{proposition}\label{prop-G-galois-are-G-torsor}
 Let $G$ be a finite group, and $\C\in\CAlg(\Pr)$. 
  \begin{itemize}
   \item[(a)] The $G$-torsors in the Galois category 
   $\CAlg^\wcov(\C)$ are precisely the faithful $G$-Galois extensions. Moreover, given $A \in \CAlg^\wcov(\C)$, there exists a faithful $G$-Galois 
   extension $B$ such that 
   \[
   B \otimes A\simeq\prod_{i=1}^k B[e_i^{-1}] \qquad \mathrm{in}\;\;\CAlg(\mod{B}(\C))
   \]
   where each $e_i$ is an idempotent of $B$.
   \item[(b)] The $G$-torsors in the Galois category $\CAlg^\cov(\C)$ are precisely the (faithful) $G$-Galois extensions which are descendable. Moreover, given $A \in \CAlg^\cov(\C)$, there exists a 
   (faithful) $G$-Galois extension $B$ which is descendable and such that 
   \[
   B \otimes A\simeq\prod_{i=1}^k B[e_i^{-1}] \qquad \mathrm{in}\;\;\CAlg(\mod{B}(\C))
   \]
   where each $e_i$ is an idempotent of $B$.
  \end{itemize}
\end{proposition}
 
 Recall that descendability implies faithfullness by~\cite{Mathew2016}*{Proposition 3.19}. This explains the use 
 of the parenthesis in part (b) above.
 
\begin{proof}
 The first part of (a) and (b) is proved in~\cite{Mathew2016}*{Proposition 6.13 and Corollary 6.15}. The rest follows from~\cite{Mathew2016}*{Corollary 5.41}.
\end{proof}

Finally we record the following result.

\begin{proposition}\label{prop-galois-is-descendable}
 Let $\C\in\CAlg(\Pr)$ be given, and $A\in\CAlg(\C)$ be a faithful $G$-Galois extension. Then $A\in\C^{\mathrm{dual}}$. 
 Moreover if $\1$ is compact in $\C$ then $A$ is descendable in $\C$.
\end{proposition}

\begin{proof}
 See~\cite{Mathew2016}*{Proposition 6.14 and Theorem 3.38}.
\end{proof}

\section{Finite covers and separability}\label{finite_covers_and_sep}

The goal of this section is to show that finite covers are precisely the separable 
commutative algebras with underlying perfect module and locally constant and finite degree function. Let us first introduce some 
notation.
 
\begin{definition}
 Consider $\C \in \CAlg(\Pr)$.
 \begin{itemize} 
  \item We denote by $\CAlg^{\sep,\mathrm{f}}(\C)$ the full subcategory of $\CAlg(\C)$ spanned by the separable commutative algebras of finite degree. 
  \item  We denote by $\CAlg^{\sep,\cf}(\C^\dual)$ the full subcategory of 
 $\CAlg(\C^\dual)$ spanned by the separable commutative algebras with constant and 
 finite degree function.
 \item We denote by $\CAlg^{\sep,\lcf}(\C^\dual)$ the full subcategory of 
 $\CAlg(\C^\dual)$ spanned by the separable commutative algebras with locally constant and 
 finite degree function.
 \end{itemize}
\end{definition}
Our proof can be divided in four steps:
\begin{itemize}
    \item[(1)] We show that any weak finite cover is separable of finite degree and dualizable: 
    $\CAlg^{\wcov}(\C)\subseteq \CAlg^{\sep,\mathrm{f}}(\C^\dual)$, see Corollary~\ref{cor-wcov-sep-dual}. This step relies on results of Rognes which, for completeness, we recall below. 
    \item[(2)] We show that finite covers are separable, dualizable and have locally constant and finite degree function: $\CAlg^{\cov}(\C)\subseteq \CAlg^{\sep, \lcf}(\C^\dual)$ see Theorem~\ref{thm-(a)}.
    \item[(3)] If $\C$ is fin-connected, we show that any separable and dualizable commutative algebra with locally constant and finite degree function is a finite cover: $\CAlg^{\sep,\lcf}(\C^\dual)\subseteq \CAlg^{\cov}(\C)$ see Theorem~\ref{thm-(b)}.
    \item[(4)] Combining the previous steps, we deduce that $\CAlg^{\sep,\lcf}(\C^\dual)=\CAlg^{\cov}(\C)$, see Corollary~\ref{cor-sep-lcf=cov}.
\end{itemize}

To prove step one, we will first show that any faithful $G$-Galois extension is separable and has a dualizable underlying module, and then deduce the general result from this.

\begin{proposition}\label{prop-galois-is-separable}
 Let $A\in\CAlg(\C)$ be a $G$-Galois extension. Then $A$ is separable of degree $|G|$.
\end{proposition}

\begin{proof}
 The first claim is essentially~\cite{Rognes2008}*{Lemma 9.1.2}; we reproduce the argument here for 
 completeness. Since $\C$ is additive and $G$ is finite, we have a natural equivalence $\coprod_G A \simeq \prod_G A$. Therefore the canonical inclusion $\{e\}\to G$ of the neutral 
 element induces a map  
 \[
 i_e= A \to  \prod_G A.
 \] 
 Endow $\prod_G A$ with the $A$-bimodule structure making $h$ into an $A$-bimodule map. Informally, this is given by 
 \[
 a_0. (a_g)_{g}. a_1:=(a_0a_g g(a_1))_{g} \qquad  a_0,a_1 \in A \quad \mathrm{and} \quad (a_g)_g \in \prod_G A.
 \]
 Note that the map $i_e$ is also an $A$-bimodule map. The required $A$-bimodule section $\sigma$ to $\mu$ is described 
 by the following diagram
 \begin{equation}\label{eq-diagram-galois-sep}
 \begin{tikzcd}
 A \arrow[r, dotted, "\sigma"] \arrow[rd, "i_e"']& A\otimes A \arrow[d, "\simeq"', "h"] \arrow[r,"\mu"] 
 & A \\
  & \prod_G A \arrow[ru, "\mathrm{pr}_e"']& .
 \end{tikzcd}
 \end{equation}
 To calculate the degree note that $F_A(A)= \1_{A}^{\times |G|}$ in $\mod{A}(\C)$ and apply 
 Proposition~\ref{prop-properties-degree} (b) and (a).
\end{proof}

\begin{lemma}[\cite{Rognes2008}*{Lemma 6.2.6}]\label{lem-dual-nu-map}
 Let $A\in\Fun(BG, \CAlg(\C))$ for a finite group $G$, and let 
 $X\in\C$ be dualizable. Then the canonical map 
 \[
 \nu_X \colon  X \otimes A^{hG}\to (X \otimes A)^{hG}
 \]
 is an equivalence.
\end{lemma}

\begin{proof}
    The result follows from the fact that $A^{hG}=\lim_{BG} A$ and that $X \otimes -$ preserves limits since $X$ is dualizable. 
\end{proof}

The next result gives the first part of step one. 

\begin{lemma}[\cite{Rognes2008}*{Lemma 6.2.4}]\label{lem-detect-dualizable-obj}
 Let $A\in\CAlg(\C)$ be dualizable and faithful. 
 If $X\in \C$ is such that $A\otimes X$ is dualizable in $\mod{A}(\C)$, 
 then $X$ is dualizable. In particular any weak finite cover is dualizable.
\end{lemma}

\begin{proof}
Consider the free-forgetful adjuction $F_A \colon \C \to \mod{A}(\C): U_A$. 
We note that $F_A=A\otimes-$ is conservative and preserves all limits since $A$ is faithful and dualizable.
Therefore by the $\infty$-categorical version of the Barr-Beck theorem~\cite{Mathew2016}*{Theorem 3.3}, the adjunction $(F_A, U_A)$ is comonadic. 
An application of~\cite{HA}*{Theorem 4.7.5.2} shows that $\C \simeq \Tot(\mod{A^{\otimes \bullet +1}}(\C))$. 
By~\cite{HA}*{Proposition 4.6.1.11}, the object $X$ is dualizable in $\C$ if and only if $A^{\otimes n}\otimes X$ is dualizable in $\mod{A^{\otimes n}}(\C)$ for all $n\geq 1$. 
This now follows from our assumption and the fact that the free functor preserves dualizable objects.

  For the second claim consider a weak finite cover $A$. Then by 
 Proposition~\ref{prop-G-galois-are-G-torsor} there exists a faithful $G$-Galois extension $B$ such 
 that $F_B(A)=\prod_{i=1}^k B[e_i^{-1}]$. 
 Note that $F_B(A)$ is dualizable in $\mod{B}(\C)$ as it is a finite product of retracts of the unit 
 objects. Moreover $B$ is dualizable by Proposition~\ref{prop-galois-is-descendable} and faithful 
 by definition. So by the previous paragraph $A$ is dualizable.
\end{proof}

\begin{lemma}[\cite{Rognes2008}*{Lemma 7.2.5}]\label{lem-kappa-equivalence}
 Let $A$ be a faithful $G$-Galois extension. For every subgroup $K\subseteq G$, the commutative 
 algebra $A^{hK}$ is dualizable in $\C$ and the canonical map 
 $\kappa \colon A^{hK} \otimes A^{hK}\to (A\otimes A)^{h(K\times K)}$ is an 
 equivalence.
\end{lemma}

\begin{proof}
 Note that $A$ is dualizable in $\C$ by Proposition~\ref{prop-galois-is-descendable}. Thus by 
 Lemma~\ref{lem-detect-dualizable-obj}, we can prove that $A^{hK}$ is dualizable in $\C$ 
 by checking that $A\otimes A^{hK}$ is dualizable in $\mod{A}(\C)$.  
 The equivalence $h\colon A \otimes A \to\prod_G A$ is $K$-equivariant with respect to the left 
 $K$-action on $A \otimes A$ via the right copy of $A$, and a right action on $\prod_G A$ via 
 $(a_g)_g. k=(a_{gk})_g$ for $k\in K$ and $(a_g)_g\in\prod_G A$. Then using Lemma~\ref{lem-dual-nu-map} we see that 
 \[
 A\otimes A^{hK}\simeq (A \otimes A)^{hK} \simeq (\prod_{G} A)^{hK} \simeq \prod_{G/K} A,
 \]
 where in the last equivalence we used that $K$ acts on $\prod_G A$ only via the indexing set $G$. 
 The right hand side is dualizable in $\mod{A}(\C)$ since it is a finite product of the unit 
 object. Therefore $A^{hK}$ is dualizable in $\C$.
 
 For the second claim we note that the map $\kappa$ factors as the composite 
 \[
 A^{hK}\otimes A^{hK} \xrightarrow{\nu} (A\otimes A^{hK})^{hK} \xrightarrow{\nu}(A\otimes A)^{h(K\times K)}.
 \]
 We have already noted that $A$ is dualizable in $\C$, and by the previous paragraph $A^{hK}$ is 
 dualizable too. Therefore the two maps above are equivalences by Lemma~\ref{lem-dual-nu-map} 
\end{proof}

\begin{proposition}\label{prop-fixed-points-galois-is-separable}
 Let $A$ be a faithful $G$-Galois extension. For every subgroup $K\subseteq G$, the commutative 
 algebra $B=A^{hK}$ is separable in $\C$. 
\end{proposition}

\begin{proof}
 This is essentially~\cite{Rognes2008}*{Proposition 9.1.4}; we record the argument here for 
 completeness. The canonical map 
 $h\colon A \otimes A \to \prod_G A$ is $(K \times K)$-equivariant with respect to the 
 actions given by $(k,k'). (a \otimes a')=k.a \otimes k'.a'$ on the source, and 
 $(k,k'). (g \mapsto a_g)=(g \mapsto k .a_{k^{-1}gk'})$ on the target.  
 We note that there are maps
 \[
 \prod_K A \xrightarrow{i_K} \prod_G A \xrightarrow{\mathrm{pr_K}} \prod_K A
 \]
 whose composite is the identity.
 The first map is induced by the inclusion $K \to G$ and uses the equivalence 
 $\coprod_K A \simeq \prod_K A$ coming from the fact that $\C$ is additive and $K$ is 
 finite. The map $\mathrm{pr}_K$ is the projection onto the $K$-factors. 
 The maps $i_K$ and $\mathrm{pr_K}$ are $(K \times K)$-equivariant where the action on 
 $\prod_K A$ is defined in a way similar to that of $\prod_G A$.  
 The equivalence $A \to (\prod_K A)^{hK}$ induces an equivalence 
 $A^{hK}\to (\prod_K A)^{h(K\times K)}$ which makes the following diagram commute
 \[
 \begin{tikzcd}[column sep=large, row sep=large]
  A^{hK}\arrow[r, dotted] \arrow[d, dotted, "="] & A^{hK}\otimes A^{hK} \arrow[d,"\kappa", "\simeq"'] \arrow[r,"\mu"] & A^{hK}\arrow[d,"="] \\
  A^{hK} \arrow[d, "\simeq"] \arrow[r,dotted] & (A\otimes A)^{h(K\times K)} \arrow[d, "h^{h(K\times K)}", "\simeq"'] \arrow[r, "\mu^{h(K\times K)}"] & A^{hK} \arrow[d, "\simeq"]\\
  (\prod_K A)^{h(K\times K)} \arrow[r, "i_K^{h(K\times K)}"] & (\prod_G A)^{h(K\times K)} 
  \arrow[r,"\mathrm{pr}_K^{h(K\times K)}"] & (\prod_K A)^{h(K\times K)}  .
 \end{tikzcd}
 \]
 The map $\kappa$ is an equivalence by Lemma~\ref{lem-kappa-equivalence}, and the maps 
 $h^{h(K\times K)}$ and $\mathrm{pr}_K^{h(K \times K)}\circ i_K^{h(K\times K)}$ are 
 obtained from equivalences by passing to homotopy $(K\times K)$-fixed points. 
 A diagram chase then shows that $\mu$ admits a section as claimed. 
\end{proof}

\begin{theorem}\label{thm-fin-cover-sep}
  Let $\C\in \CAlg(\Pr)$ and consider a weak finite cover $A\in \CAlg(\C)$. Then there exist finite collections of finite groups $\{G_\alpha\}$ and $\{K_{j_\alpha}\}$ 
  with $K_{j_\alpha} \subseteq G_\alpha$, and faithful $G_\alpha$-Galois extension $B_\alpha \in \CAlg(\C)$ such that $A=\prod_{\alpha}\prod_{j_\alpha} B_\alpha^{hK_{j_\alpha}}$. 
  In particular, $A$ is separable of finite degree.
\end{theorem}

\begin{proof}
 By the Galois correspondence (Theorem~\ref{thm-Galois-correspondence}) we have an equivalence of Galois categories
 \[
 \CAlg^{\wcov}(\C)^{\op} \simeq \colim_{\lambda\in \Lambda} \Fun(B\G_{\lambda}, \Set)
 \]
 where $\Lambda$ is filtered and the $\G_{\lambda}$'s are finite groupoid. In particular, we can find a finite groupoid $\G=\G_{\lambda}$ for some $\lambda\in \Lambda$ and a functor of Galois categories
 \[
 F \colon \Fun(B\G, \FinSet) \to \CAlg^{\wcov}(\C)^{\op}
 \]
 whose essential image contains $A$. Say that $X \in \Fun(B\G, \FinSet)$ is such that $F(X)\simeq A$. We decompose the finite groupoid into its finite connected components $B\G\simeq \coprod_{\alpha} BG_{\alpha}$ and so rewrite $X=\coprod_{\alpha} X_\alpha$ where $X_\alpha$ is a finite $G_\alpha$-set. Thus, we have functors of Galois categories 
 \[
 F_\alpha \colon \FinSet_{G_{\alpha}} \to \CAlg^\wcov(\C)^{\op}, \quad X_\alpha \mapsto A_\alpha
 \]
 such that $\prod_\alpha A_\alpha\simeq A$. 
 The Galois category $\FinSet_{G_\alpha}$ has a natural $G_\alpha$-torsor which is given by $G_{\alpha}$, see Example~\ref{ex-G-torsor-G}. Passing this along $F_\alpha$ gives us a faithful $G_{\alpha}$-Galois extension $B_{\alpha}\in \CAlg(\C)$ by 
 Proposition~\ref{prop-G-galois-are-G-torsor}.
 We decompose $X_\alpha$ even further into its orbits $X_\alpha\simeq \coprod_{j_{\alpha}} G_{\alpha}/K_{j_{\alpha}}$ and note that $X_\alpha=\coprod_{j_\alpha}(G_{\alpha})_{hK_{j_\alpha}}$. Thus after applying $F_\alpha$ to this identity we find that $A_\alpha\simeq \prod_{j_\alpha} B_\alpha^{K_{j_\alpha}}$. Putting all together, $A \simeq \prod_{\alpha} \prod_{j_\alpha}B_\alpha^{K_{j_\alpha}}$. The fact that $A$ is separable follows from 
 Proposition~\ref{prop-fixed-points-galois-is-separable} and Example~\ref{ex-separable-product-retract}. To calculate the degree of $A$, recall that by definition of weak finite cover, there exists a faithful commutative algebra $B$ and an equivalence 
 \[
 B \otimes A \simeq \prod_{i=1}^s B[e_i^{-1}] \in \CAlg(\mod{B}(\C))
 \]
 for some idempotent elements $e_i$'s of $B$. By Proposition~\ref{prop-properties-degree}(a), the degree of $A$ agrees with the degree of $B \otimes A$ calculated in $\mod{B}(\C)$. Now the claim follows from Lemmas~\ref{lem-product-fin-degree} and~\ref{lem-locally-constant}.
\end{proof}

We are finally ready to prove step one.

\begin{corollary}\label{cor-wcov-sep-dual}
   For any $\C \in \CAlg(\Pr)$, there is an inclusion $$\CAlg^{\wcov}(\C)\subseteq \CAlg^{\sep, \mathrm{f}}(\C^\dual).$$
\end{corollary}

\begin{proof}
Combine Lemma~\ref{lem-detect-dualizable-obj} and Theorem~\ref{thm-fin-cover-sep}.
\end{proof}

Before proving step two we will need the following result.

\begin{lemma}\label{lem-quotientbyG}
 Let $\C\in\CAlg(\Pr)$ and $A \in \CAlg(\C^\dual)$ be a faithful $G$-Galois extension which is descendable. Then the map 
 induced by the extension of scalars functor 
 \[
 \Spc(\mod{A}(\Ind(\C^\dual))^\omega) \to \Spc(\C^\dual)
 \]
 is isomorphic to the quotient map 
 \[
 \Spc(\mod{A}(\Ind(\C^\dual))^\omega) \to \Spc(\mod{A}(\C)^\dual)/G.
 \]
\end{lemma}

\begin{proof}
From ~\cite{Bregje}*{Theorem 9.1} with ${\mathcal K}=\mathrm{h}\C^{\mathrm{dual}}$ there we obtain
\[ \Spc(\C^{\mathrm{dual}}) = \Spc(\mathrm{h}\C^{\mathrm{dual}}) \simeq \Spc(\mod{uA}(\mathrm{h}\C^\dual))/G,\]
using that $\mathrm{supp}(A)=\Spc(\C^{\mathrm{dual}})$ because $A$ is descendable. The claim then follows from $\mod{uA}(\mathrm{h}\C^\dual)\simeq \mathrm{h}\mod{A}(\Ind(\C^\dual))^\omega$, see Lemma~\ref{lem-tt-tower}.

%Note that $A$ is separable by Proposition~\ref{prop-galois-is-separable} and that $\supp(A)=\Spc(\C^\dual)$ as $A$ is descendable. The claim then follows from~\cite{Bregje}*{Theorem 9.1}. Here we are implicitly using the tt-equivalence $\mathrm{h}\mod{A}(\Ind(\C^\dual))^\omega\simeq \mod{uA}(\mathrm{h}\C^\dual)$ from Lemma~\ref{lem-tt-tower}.
\end{proof}

We now prove step two.

\begin{theorem}\label{thm-(a)}
 For any $\C \in \CAlg(\Pr)$, there is an inclusion 
 \[
 \CAlg^\cov(\C)\subseteq \CAlg^{\sep,\lcf}(\C^\dual).
 \]
\end{theorem}

\begin{proof}
 Consider $A\in \CAlg^\cov(\C)$. Since a finite cover is a weak finite cover, we deduce that the underlying module of $A$ is dualizable by Lemma~\ref{lem-detect-dualizable-obj}, and that $A$ is separable of finite degree by Theorem~\ref{thm-fin-cover-sep}.\\ 
 It remains to show that the degree function is locally 
 constant. By Proposition~\ref{prop-G-galois-are-G-torsor}(b), there exists a 
 faithful $G$-Galois extension $B$ which admits descent such that 
 $F_{B}(A)=\prod_{i=1}^k B[e_i^{-1}]$. By Lemma~\ref{lem-quotientbyG} we 
 can identify the map 
 \[
 \varphi\colon \Spc(\mod{B}(\Ind(\C^\dual))^\omega)\to \Spc(\C^\dual)
 \]
 induced by the extension of scalars functor with the quotient map by the action of $G$. 
 Here an element $g \in G$ acts on a prime ideal $\p$ via $\p.g :=g^{-1}(\p)$.  
 By Proposition~\ref{prop-diagram commutes} we know that 
 \begin{equation}\label{eq-degree}
 \deg(F_B(A))(\p)=\deg(A)(\varphi(\p))
 \end{equation}
 for all $\p\in \Spc(\mod{B}(\Ind(\C^\dual))^\omega)$. Given primes $\p, \q$ with 
 $\varphi(\p)=\varphi(\q)$, there exists $g \in G$ such that $\p.g=\q$. Then multiplication 
 by $g$ defines an equivalence 
 \[
 \mod{B}(\Ind(\C^\dual))^\omega_{\q} \simeq \mod{B}(\Ind(\C^\dual))^\omega_\p
 \]
 which shows that $\deg(F_B(A))(\p)=\deg(F_B(A))(\p.g)$ for all $g \in G$, using that $F_B(A)g\simeq F_B(A)$. In other words 
 $\deg(F_B(A))$ is constant on the orbits of $G$. This fact together with~(\ref{eq-degree}) 
 implies that $\deg(A)$ is locally constant whenever $\deg(F_B(A))$ is so. 
 The latter degree function was shown to be locally constant 
 in Lemma~\ref{lem-locally-constant}.
\end{proof}

  The degree function of a {\em weak} finite cover need not be locally constant, see Example~\ref{ex-wc-not-lc-degree}.

We now turn to giving conditions under which a separable algebra is a finite cover.

\begin{proposition}\label{prop-separable-constant-degree-equal-split-ring}
 Consider $A \in \CAlg^{\mathrm{sep}}(\C^\dual)$. Then the following are equivalent:
 \begin{itemize}
 \item[(a)] $A$ has constant degree $ d\in \Z_{\geq 0}$.
 \item[(b)] There exists $B \in \CAlg(\C^\dual)$ descendable such that $F_B(A) = B^{\times d}$ 
 in $\mod{B}(\Ind(\C^\dual))^\omega$.
 \end{itemize} 
 In particular we have an inclusion $\CAlg^{\sep,\cf}(\C^\dual)\subseteq \CAlg^\cov(\C)$. 
\end{proposition}

 This is a slight modification of~\cite{Bregje}*{Proposition 8.4}.
 
\begin{proof}
 Assume that (a) holds and put $B=A^{[d]}\in\CAlg(\C^{\mathrm{dual}})$. We claim that $B$ is descendable. 
 By part (a) we have $\deg(A)(\p)=d$ for all $\p \in \Spc(\C^\dual)$ and so 
 $\supp(B)=\Spc(\C^\dual)$ by the first part of Lemma~\ref{lem-descend-degree-geq1}.  By Lemma~\ref{lem-desc}, this is equivalent to $B$ being descendable. 
 It is only left to show that $F_B(A) = B^{\times d}$. 
 This is~\cite{Balmer2014}*{Theorem 3.9(c)}.\\ 
  Conversely if $B$ is descendable then $\supp(B)=\Spc(\C^\dual)$. It follows from ~\cite{Balmersurj}*{Theorem 1.7} that the map 
 $\Spc(F_B) \colon \Spc(\mod{B}(\Ind(\C^\dual))^\omega) \to \Spc(\C^\dual)$ is surjective. Then $A$ has constant degree if and only if $F_B(A)$ has constant degree 
 by~\cite{Bregje}*{Lemma 8.3}. By direct verification $B^{\times d}$ has constant degree equal to $d$, see 
 Example~\ref{ex-locally-constant-degree-function}.
\end{proof}

We now prove step three.

\begin{theorem}\label{thm-(b)}
 Let $\C \in \CAlg(\Pr)$ be fin-connected. 
 Then there is an inclusion 
 $\CAlg^{\sep,\lcf}(\C^\dual)\subseteq \CAlg^\cov(\C)$. 
\end{theorem}

\begin{proof} 
 By assumption $\C\simeq \prod_{i=1}^n \C_i$ where each $\C_i$ is connected. 
 We then get decomposition
  \[
  \CAlg^\sep (\C^\dual)\simeq \prod_{i=1}^n \CAlg^\sep(\C_i^\dual) \quad \mathrm{and} \quad
  \CAlg^\cov(\C)= \prod_{i=1}^n \CAlg^\cov(\C_i)
  \]
  by Proposition~\ref{prop-descent-sep}, cf.~\cite{Mathew2016}*{Proposition 7.2}.  
  The canonical projection maps $\pi_i \colon \C^\dual \to \C_i^\dual$ induce maps on Balmer spectra 
  $\varphi_i \colon \Spc(\C_i^\dual) \to \Spc(\C^\dual)$. For $A \in\CAlg^{\sep} (\C^\dual)$, Proposition~\ref{prop-diagram commutes} tells us that
  \[
  \deg(A)(\varphi_i(\p))=\deg(\pi_i(A))(\p)
  \]
  for all $\p \in \Spc(\C_i^\dual)$ and $1 \leq i \leq n$. This formula together with the fact that  
  the map $\varphi_i$ is the inclusion of a connected component (see proof of 
  Proposition~\ref{prop-char-connected-Balmer}), yield the equivalence
  \[
    \CAlg^{\sep,\lcf} (\C^\dual)\simeq \prod_{i=1}^n \CAlg^{\sep,\cf}(\C_i^\dual),
  \]  
  using that every locally constant $\Z$-valued function on the connected space $\Spc(\C_i^{\mathrm{dual}})$ is constant.
  Thus it suffices to show that $\CAlg^{\sep,\cf}(\C_i^\dual)\subseteq \CAlg^\cov(\C_i)$ 
  for all $1 \leq i \leq n$. 
  This is the content of Proposition~\ref{prop-separable-constant-degree-equal-split-ring}.
\end{proof}

We are finally ready to state and prove the main result of this section.

\begin{corollary}\label{cor-sep-lcf=cov}
 Let $\C \in \CAlg(\Pr)$ be fin-connected. 
 Then there is an equality \[
 \CAlg^{\sep,\lcf}(\C^\dual)= \CAlg^\cov(\C).
 \]
\end{corollary}

\begin{proof}
 Combining Theorems~\ref{thm-(a)} and~\ref{thm-(b)} 
 we get inclusions
 \[
 \CAlg^\cov(\C) \subseteq \CAlg^{\sep,\lcf}(\C^\dual) \subseteq \CAlg^\cov(\C).
 \]
 Therefore the inclusions in the previous display are in fact equalities.
\end{proof}

\section{Weak finite covers and separable algebras in limit categories}\label{sec-descent}

As we will see in the second part of this paper, there are many situations in which the containment of Lemma \ref{cor-wcov-sep-dual}
\[
\CAlg^{\wcov}(\C)\subseteq \CAlg^{\sep,\mathrm{f}}(\C^\dual)
\]
is in fact an equality. In this section we show that this property is preserved under arbitrary limits of stable homotopy theories, see Theorem~\ref{thm-descent}. To this end, we will need to discuss limits of Galois categories in details, following~\cite{Mathew2016}. We start by recalling the rank of an object in a Galois category from~\cite{Mathew2016}*{Proof of 5.28}.

\begin{definition}
 Let $\C$ be a Galois category and consider an object $x \in \C$.
 By the axiom of a Galois category, there exists an effective descent morphism $x' \to \ast$ and a decomposition $x'= x_1' \sqcup \ldots \sqcup x_n'$ such that each map $x \times x_i' \to x_i'$ identifies with the fold map $x \times x_i' \simeq \sqcup_{S_i} x_i' \to x_i'$ for a finite set $S_i$. We define the \emph{rank} of $x$ to be \[\rk(x):=\sup_i |S_i|.\]
 One can verify that the rank is well-defined, i.e. it does not depend on the effective descent morphism and decomposition chosen, by reducing, via the Galois correspondence, to the case that $\C=\Fun(\G, \FinSet)$ for a finite groupoid $\G$, and then using Example~\ref{ex-rank-groupoid} below.
\end{definition}

\begin{example}\label{ex-rankfinsetG}
 Consider $S\in \FinSet_G$ for some finite group $G$. In this case we have $\rk(S)=|S|$. To see this pick any 
 effective descent morphism $T \to \ast$. By virtue of~\cite{Mathew2016}*{Proposition 5.22(iii)}, we deduce that $T$ is nonempty. Decompose $T$ into its orbits $T \simeq \coprod_i G/H_i$. Each projection map $S \times G/H_i \to G/H_i$ can be written as the fold map $\coprod_{|S|}G/H_i \to G/H_i$, so the rank of $S$ coincides with its cardinality. 
\end{example}

\begin{example}\label{ex-rank-groupoid}
Consider a finite groupoid $\G\simeq\coprod_{i}G_i$. In this case the rank of $X=(X_i) \in \Fun(\G, \FinSet)\simeq \prod_i \FinSet_{G_i}$ is given by $\rk(X)=\sup_i |X_i|$. 
\end{example}

We next verify that the notions of degree for separable algebras and of rank for Galois extensions coincide, in the cases in which they are both defined.

\begin{proposition}\label{prop-degree-rank}
Assume that $\C\in\CAlg(\Pr)$ and that \[ A\in \CAlg^{\wcov}(\C)\subseteq \CAlg^{\mathrm{sep},\mathrm{f}}(\C^{\mathrm{dual}})\] is given. Then we have $\deg(A)=\rk(A)$.
\end{proposition}

\begin{proof}
    Before diving into the proof of the proposition it will be helpful to recall some facts about abstract Galois theory.  Let $\CC=\CAlg(\C)^{\op}$ and $\mathcal{E}$ be the collection of maps $A \to B$ in $\CC$ such that the base change functor $B \otimes_A -$ commutes with limits and is conservative. Then by \cite{Mathew2016}*{Lemma 6.4} the pair $(\CC,\mathcal{E})$ forms a Galois context in the sense of \cite{Mathew2016}*{Definition 5.26} and the collection of Galoisable objects in the sense of~\cite{Mathew2016}*{Definition 5.27}) agrees with the weak finite covers.
    After this small digression we are ready to prove the proposition. Let $A$ be a finite cover so that there exists $B \in \CAlg(\C)$ faithful such that $B \otimes- \colon \C \to \C$ preserve limits, and an equivalence of $B$-algebras $A\otimes B \simeq \prod_{i=1}^n B[e_{i}^{-1}]$ for some idempotents $e_i$.  By the the remark following the proof of \cite{Mathew2016}*{Proposition 5.28} (apply to the Galois context defined above), we may assume that $B$ is itself a weak finite cover. It then follows from \cite{Mathew2016}*{Corollary 5.29} that the map $\1\to B$ is an effective descent morphism, so we can compute $\rk(A)$ using base-change along $\1\to B$. 
% Using the language of Galois theory, $\1 \to B$ is an effect descent 
 %morphism in $\CAlg^{\wcov}(\C)$ ( see~\cite{Mathew2016}*{6.4, 5.26(2)}) so we can calculate the rank of $A$ after base changing along $\1\to B$.
 We regroup the above decomposition
 \[
 F_B(A)=A\otimes B \simeq \prod_{i=1}^n B[e_{i}^{-1}] =\prod_{j=1}^s (B[e_j^{-1}])^{\times n_j}
 \]
 by insisting that all the $e_j$'s are pairwise orthogonal, and that $B[e_j^{-1}]\neq 0$. Then by definition, the rank of $A$ is given by $\max_j n_j$. Note that $A$ is separable of finite degree by Theorem~\ref{thm-fin-cover-sep}. We now claim that
 \[
 \deg(A)= \deg(F_B(A)) = \rk(A).
 \]
 The first equality follows from Proposition~\ref{prop-properties-degree}(a), so let us discuss the second one. 
 For all Balmer primes $\p$, we know that 
 \[
 \deg(F_B A)(\p)=\sum_{j=1}^s n_j  \deg(B[e_j^{-1}])(\p)
 \]
 by Lemma~\ref{lem-product-fin-degree}. Note that if $j\not =k$, then $\supp(B[e_j^{-1}])\cap \supp(B[e_k^{-1}])= \emptyset$ as the idempotents are orthogonal. Therefore by~\cite{Balmer2014}*{Corollary 3.12} and Lemma~\ref{lem-locally-constant}, we calculate
 \[
 \deg(F_B(A))=\max_{\p} \left\{ \sum_{j=1}^s n_j  \deg(B[e_j^{-1}])(\p) \right\}= \max_j \{n_j\}=\rk(A).
 \]
\end{proof}

We exploit the rank to describe limits of Galois categories as follows.

\begin{lemma}\label{lem-lim-galcat}
 Let $K$ be a simplicial set, and let $F\colon K \to \GalCat$ be a diagram. Then the limit of $F$ exists and is given by \[(\lim F)^{\br}\subseteq \lim F,\]
 namely the full subcategory of the categorical limit spanned by the objects of bounded rank. In particular,  if $K$ has finitely many vertices, then $\lim F = (\lim F)^{\br}$.
\end{lemma}

\begin{proof}
 The proof is the same as that of~\cite{Mathew2016}*{Lemma 5.37}, see also~\cite{Mathew2016}*{Remark 5.38}.
\end{proof}

\begin{lemma}\label{lem-wcov-com-limits}
 Let $K$ be a simplicial set and let $p\colon K \to \CAlg(\Pr)$ be a functor. 
 Then the canonical functor 
 \[
 \CAlg^{\wcov}(\lim_K p) \xrightarrow{\sim} (\lim_{k\in K} \CAlg^{\wcov}(p(k)))^{\br}
 \]
 is an equivalence in $\Cat_\infty$ (but also an equivalence of Galois categories).
\end{lemma}

\begin{proof}
Combine~\cite{Mathew2016}*{Proposition 7.1} which equates $\CAlg^{\wcov}(\lim_K p)$ with a limit of Galois categories with Lemma~\ref{lem-lim-galcat} which describes limits of Galois categories.
\end{proof}

We now establish the permanence result formulated at the beginning of this section.

\begin{theorem}\label{thm-descent}
  Let $K$ be a simplicial set and let $p\colon K \to \CAlg(\Pr)$ be a functor. Suppose that for each $k\in K$, we have $\CAlg^{\wcov}(p(k))=\CAlg^{\sep,\mathrm{f}}(p(k)^\dual)$. Then we also have
  \[
  \CAlg^{\wcov}(\lim_K p) = \CAlg^{\sep,\mathrm{f}}((\lim_K p)^\dual).
  \]
 \end{theorem}

\begin{proof}
 Consider the following commutative square in $\Cat_\infty:$
 \[
 \begin{tikzcd}
 \CAlg^{\sep,\mathrm{f}}((\lim_K p)^\dual) \arrow[r, hook] & \lim_k \CAlg^{\sep,\mathrm{f}}(p(k)^\dual) \\
   \CAlg^{\wcov}(\lim_K p) \arrow[u, hook] \arrow[r, hook] & \lim_k \CAlg^{\wcov}(p(k))  \arrow[u, "\sim"'].
 \end{tikzcd}
 \]
 The left vertical functor is fully faithful by Corollary~\ref{cor-wcov-sep-dual} and the right vertical is an equivalence by our assumption. 
 The horizontal arrows are fully faithful as they are subfuctors of the equivalence in Lemma~\ref{lem-calg-limits}. We only need to prove that the left vertical arrow is essentially surjective. To this end pick a separable commutative algebra $A$ with underlying dualizable module in $\lim_k p$ which has finite degree. Then its image $(A_k)$ under the top horizontal arrow will have bounded degree by Corollary~\ref{cor-bounded-degree}. Under the right vertical equivalence this corresponds to a compatible sequence of weak finite covers $(A_k')$ which by Proposition~\ref{prop-degree-rank} must have bounded rank.  Then by Lemma~\ref{lem-wcov-com-limits} there exists a weak finite cover $A$ in $\lim_K p$ with components $(A_k')$. By construction the image of $A'$ under the left vertical arrow is $A$.
\end{proof}

For the next result recall that an object $k$ of a simplicial set $K$ is said to be {\em weakly initial} if it maps to any other object of $K$. 

\begin{proposition}\label{prop-descent}
 Let $K$ be a simplicial set and let $p\colon K \to \CAlg(\Pr)$ be a functor. Suppose that $K$ admits a weakly initial object $k_0$ such that $ \CAlg^\wcov(p(k_0))=\CAlg^{\sep}(p(k_0)^\dual) $. Then we also have  
 \[
 \CAlg^\wcov(\lim_K p)= \CAlg^{\sep}((\lim_K p)^\dual). 
 \]
\end{proposition}

\begin{proof}
  Note that we have a natural transformation of functors from $K$ to $\Cat_\infty$
  \[
  \iota_\bullet \colon \CAlg^{\wcov}(p(\bullet)) \to \CAlg^{\sep}(p(\bullet)^\dual)
  \] 
  which is fully faithful at all vertices of $K$, by Corollary~\ref{cor-wcov-sep-dual}. Moreover $\iota_{k_0}$ is 
  an equivalence by our assumption. The same argument as in~\cite{Chromatic}*{Lemma 5.16} applies to show that 
  
  \[
  \lim_k \iota_{k}\colon \lim_k\CAlg^{\wcov}(p(k)) \xrightarrow{\sim} \lim_k\CAlg^{\sep}(p(k)^\dual)
  \]
  is an equivalence. Now there is a commutative diagram 
  \[
  \begin{tikzcd}
  \CAlg^{\sep}((\lim_K p)^\dual) \arrow[r,"\sim"] & \lim_k \CAlg^\sep(p(k)^\dual) \\
  \CAlg^\wcov(\lim_K p) \arrow[u, hook] \arrow[r] & \lim_k \CAlg^{\wcov}(p(k)) \arrow[u, "\sim","\lim_k \iota_{k}"']
  \end{tikzcd}
  \]
    where the top horizontal arrow is an equivalence by Proposition~\ref{prop-descent-sep} and the left vertical arrow is fully faithful by Corollary~\ref{cor-wcov-sep-dual}. We ought to show that the left vertical arrow  is essentially surjective. By the commutativity of the diagram it suffices to show that the bottom horizontal arrow is an equivalence. This would follow from Lemma~\ref{lem-wcov-com-limits} if we knew that any object in $A\in \lim_k \CAlg^{\wcov}(p(k))$ had bounded rank (or bounded degree by Proposition~\ref{prop-degree-rank}). To see this, we write $A=(A_k)_{k\in K}$ and note that for all $k\in K$, we have $\deg(A_{k_0}) \geq \deg(A_k)$ since $k_0$ is weakly initial. It is only left to note that $A_{k_0}$ has finite degree as, by our assumption, it is a weak finite cover. 
\end{proof}

\begin{corollary}\label{cor-descent-sep}
  Assume that $\C\in\CAlg(\Pr)$, and let $A \in \CAlg(\C)$ be descendable. If 
  \[ \CAlg^{\wcov}(\mod{A}(\C)) = \CAlg^{\sep}(\mod{A}(\C)^\dual),\]
  then we also have 
  \[ \CAlg^{\wcov}(\C) = \CAlg^{\sep}(\C^\dual).\]
\end{corollary}

\begin{proof}
 Recall from~\cite{Mathew2016}*{Proposition 3.22} that in this case there is a symmetric monoidal equivalence $\C\simeq \Tot(\mod{A^{\bullet +1}}(\C))$. The claim then follows from Proposition~\ref{prop-descent}.
\end{proof}

\part{Applications: classification of separable commutative algebras}

In this second part of the paper we use Galois theory to classify separable dualizable commutative algebras in various stable homotopy theories of interest. In more detail, we will study categories of modules over an $\E_\infty$-ring $A$, which is either connective or even periodic with $\pi_0(A)$ regular and Noetherian, categories of complete modules over an adic $\E_\infty$-ring, categories of quasi-coherent sheaves on an even periodic derived stack or a spectral Deligne--Mumford stack, and finally the stable module category of a finite group of $p$-rank one.

\section{Modules over an \texorpdfstring{$\E_\infty$}{E}-ring}\label{sec:modules}

In this section we classify all separable commutative algebras with underlying dualizable module in the category of modules over an $\E_\infty$-ring $A$, which is either connective or even periodic with $\pi_0(A)$ regular and Noetherian. As an application, we classify all separable commutative algebras in finite spectra.  

\begin{notation}
  Let $A$ be an $\E_\infty$-ring. The full subcategory of dualizable $A$-modules coincide with the full 
  subcategory of perfect $A$-modules. Therefore we write $\Perf{A}$ instead of $\mod{A}^\dual$.
\end{notation}

We start by recalling the following definition from~\cite{HA}*{Definition 7.5.0.4}.

\begin{definition}
Let $R$ be a discrete commutative ring and let $\phi \colon A \to B$ be a map of $\E_\infty$-rings. 
\begin{itemize}
\item  We let $\Cov_{R}$ denote the category of finite \'etale $R$-algebras. Here by finite we mean that the \'etale $R$-algebra is finitely generated as an $R$-module.
\item We say that $\phi$ is \'etale if $\pi_0(\phi)\colon\pi_0(A) \to \pi_0(B)$ is \'etale (in the sense of classical commutative algebra), and $A\to B$ is flat in the sense of~\cite{HA}*{Definition 7.2.2.10}, i.e. the natural map $\pi_0(B) \otimes_{\pi_0(A)}\pi_*(A) \to \pi_*(B)$ is an isomorphism and $\pi_0(A)\to\pi_0(B)$ is flat. 
\end{itemize}
\end{definition}

\begin{lemma}[\cite{HA}*{Theorem 7.5.4.2}]\label{lem-etale}
  Let $A$ be an $\E_\infty$-ring. The $\infty$-category of \'etale $A$-algebras is equivalent under $\pi_0$ to the category of \'etale $\pi_0(A)$-algebras. 
\end{lemma}

Finite \'etale algebras provide examples of separable algebras by the following result. 

\begin{lemma}\label{lem-inclusions}
 For any $\E_\infty$-ring $A$, we  have fully faithful inclusions
\[
\Cov_{ \pi_0 (A)} \subseteq \CAlg^{\cov}(\mod{A}) =\CAlg^{\wcov}(\mod{A})\subseteq \CAlg^{\sep,\mathrm{f}}(\Perf{A}).
\]
\end{lemma}
\begin{proof}
  Combine~\cite{Mathew2016}*{Proposition 6.10} with Corollary~\ref{cor-wcov-sep-dual}. The above equality uses~\cite{Mathew2016}*{Theorem 6.5}.
\end{proof}

In his work on Galois theory, Mathew exhibits two important classes of $\mathbb E_\infty$-rings, 
for which the inclusion 
\[ \Cov_{\pi_0(A)}\subseteq \CAlg^{\mathrm{cov}}(\mod{A})\]
is an equality. The next two results generalize this to include separable algebras with perfect underlying module.

\begin{proposition}\label{prop-connective-sep}
  Let $A$ be a connective $\E_\infty$-ring. There is an equality
  \[
  \Cov_{\pi_0(A)} =\CAlg^{\sep}(\Perf{A}).
  \]
  In particular, any separable commutative algebra in $\Perf{A}$ has finite degree.
\end{proposition}

\begin{proof}
 The proof is a variation of the one of~\cite{MathewLennart}*{Example 5.5}. 
 
  Let $B$ be a separable $A$-algebra and a perfect $A$-module. By Lemma~\ref{lem-inclusions}, it will suffice to show that $A \to B$ is \'etale and that $\pi_0(A) \to \pi_0(B)$ is finite. For any morphism $\pi_0 A \to k$ into a field, we have maps of commutative rings
 \[
 A \to \pi_0(A) \to k.
 \]
 Base-changing $B$ along this composite, we get a separable commutative algebra $k \otimes_A B$ in $\Perf{k}$ by Construction~\ref{con-smf-pres-sep}. This is equivalent to a finite product of finite separable field extension of $k$ by~\cite{Neeman2018}*{Proposition 1.6}. \\
  We now argue that $B$ is connective, for which we can assume that $B\neq 0$. Using \cite{HA}*{Corollary 7.2.4.5} we see that there is a smallest $n\in\mathbb Z$ with $\pi_n(B)\neq 0$, and that the $\pi_0(A)$-module $\pi_n(B)$ is finitely presented. There thus exists a residue field $\pi_0(A)\to k$ such that $\pi_n(B)\otimes_{\pi_0(A)}k\neq 0$. 
  By inspection of the Tor-spectral sequence~\cite{EKMM}*{Chapter IV, Theorem 4.1}
  \[
  E^2_{p,q}= \mathrm{Tor}_{p,q}^{\pi_*A}(k, \pi_*(B)) \Rightarrow \pi_{p+q}(k \otimes_A B)
  \]
  and using the minimality of $n$, one sees that $0 \not =k\otimes_{\pi_0(A)}\pi_n(B) \simeq \pi_n(k \otimes_A B)$. By the previous paragraph $k\otimes_A B$ is discrete, so we conclude that $n=0$. It follows that $B$ is connective and $\pi_0(B)$ is finitely presented as $\pi_0A$-module. In particular, $\pi_0(A) \to \pi_0 (B)$ is finite.\\
  We next argue that $A \to B$ is flat. By~\cite{HA}*{Theorem 7.2.2.15}, it suffices to verify that for any discrete $A$-module (i.e. $\pi_0(A)$-module) $N$, the spectrum $B \otimes_A N$ is discrete. Since
  \[
  B \otimes_A N \simeq (B \otimes_{A} \pi_0(A))\otimes_{\pi_0(A)} N,
  \]
  it suffices to show that $B \otimes_A \pi_0(A)$ is a discrete, flat $\pi_0(A)$-module. 
  Now we appeal to the following fact of commutative algebra~\cite{stacks-project}*{Tag 068V}: given a commutative ring $R$ and a perfect complex of $R$-modules $P$, then $P$ is quasi-isomorphic to a projective $R$-module concentrated in degree $0$ if and only if the same 
  holds (over $k$) for $P \otimes_R k$ for every residue field $k$ of $R$. 
  We apply this result to $R=\pi_0(A)$ and $P= B \otimes_{A} \pi_0(A)$, using that $B \otimes_A k$ is a finite product of separable extensions of $k$, and conclude that $B \otimes_{A} \pi_0(A)$ is a discrete, flat $\pi_0(A)$-module. 
  Thus $A \to B$ is flat. 
  
  It is only left to show that $\pi_0(A) \to \pi_0(B)$ is \'etale. Note that $\pi_0(A) \to \pi_0(B)$ is flat since $A\to B$ is so. 
  It is also formally unramified since the K\"ahler-differentials satisfy $\Omega_{\pi_0(B)/\pi_0(A)}=0$ if and only if $k \otimes_{\pi_0 (A)}\Omega_{\pi_0(B)/\pi_0(A)}=0$ 
  for all residue field $k$ of $\pi_0(A)$, and we know that $k \otimes_{\pi_0 (A)}\Omega_{\pi_0(B)/\pi_0(A)}\simeq \Omega_{\pi_0(B) \otimes_{\pi_0 (A)}k/k}=0$ since $k\subseteq \pi_0(B)\otimes_{\pi_0(A)}k$ is separable. We showed that $\pi_0(A) \to \pi_0(B)$ is finite, flat and formally unramified, so it is finite \'etale. 
\end{proof}

\begin{corollary}
 The only separable commutative algebras in finite spectra are those of the form $(S^0)^{\times n}$ for $n\in \mathbb{N}$.
\end{corollary}

\begin{proof}
 The $\infty$-category of spectra $\Sp$ can be identified with the $\infty$-category of modules over the sphere spectrum $S^0$, which is connective. The \'etale fundamental group of $\pi_0(S^0)=\Z$ is trivial by Minkowski’s theorem, so the claim follows from Proposition \ref{prop-connective-sep}. 
\end{proof}

 Recall that an $\mathbb E_\infty$-ring $A$ is \emph{even periodic} if $\pi_i(A)=0$ if $i$ is odd, and the multiplication map $\pi_2(A) \otimes_{\pi_0(A)} \pi_* (A) \to \pi_{*+2}(A)$ is an isomorphism. The proof of the next result is inspired by~\cite{Mathew2016}*{Theorem 6.29}.

\begin{theorem}\label{thm-sep-even-periodic}
Let $A$ be an even periodic $\E_\infty$-ring with $\pi_0 (A)$ regular and Noetherian with $2\in \pi_0(A)^\times$. Then we have 
\[ \Cov_{\pi_0(A)} = \CAlg^{\sep}(\Perf{A}).\] 
In particular, every separable commutative algebra in $\Perf{A}$ has finite degree.
\end{theorem}

\begin{proof}
 Let $B$ be a commutative separable $A$-algebra with perfect underlying $A$-module. One sees that $\pi_0(B)$ is a finite $\pi_0(A)$-module by induction on the number of cells and using that $A$ is Noetherian and even periodic. Therefore, it will suffice to show that $A \to B$ is \'etale. To do this, we can assume that $\pi_0(A)$ is regular local, say with maximal ideal $\mathfrak{m}$ generated by a regular sequence $(x_1,\ldots,x_n)$ and residue field $k$. Our assumptions on $A$ ensure that we can choose a structure of $\mathbb E_1$-algebra on
 \[
 A/(x_1,\ldots, x_n):=A/(x_1)\otimes_A \ldots \otimes_A A/(x_n)\in \Alg_{\E_2}(\mod{A})
 \]
 and that $\pi_*(A/(x_1,\ldots, x_n))=k[t^{\pm 1}]$ with $|t|=2$, see~\cite{Algeltveit}*{Section 3}. Moreover $ A/(x_1,\ldots, x_n)$ is automatically homotopy commutative by~\cite{Strickland99}*{Corollary 3.12}. 
 We can then define a homology theory $P_*$ from the (homotopy) category of $A$-modules to the category of graded $k[t^{\pm 1}]$-modules by 
 \[
 P_*(M)= \pi_*(M \otimes_A A/(x_1,\ldots, x_n)).
 \] 
 (here by homology theory we mean a functor which sends triangles to long exact sequences and preserves products). 
 An easy argument using the Tor-spectral sequence and the fact that $k[t^{\pm 1}]$ is a graded field, shows that the homology theory $P_*$ satisfies the K\"unneth isomorphism
 \[
 P_*(M)\otimes_{k[t^{\pm 1}]} P_*(N) \simeq P_*(M \otimes_A N).
 \]
 If we endow the category of graded $k[t^{\pm 1}]$-modules with the graded tensor product and commutativity constraint $x \otimes y = (-1)^{|x||y|}y \otimes x$, then the functor $P_*$ lifts to a symmetric monoidal functor (this can be checked on homotopy categories).
It will be convenient to replace this symmetric monoidal category of graded $k[t^{\pm 1}]$-modules with the equivalent symmetric monoidal category of super $k$-vector spaces $\mathrm{sVect}_k$. This is the category of $\Z/2$-graded $k$-vector spaces endowed with the graded tensor product and commutativity constraint $v\otimes w= (-1)^{|v||w|} w \otimes v$, so that by definition, any commutative algebra in $\mathrm{sVect}_k$ is graded commutative. 
 We denote the modified symmetric monoidal functor by $Q_*\colon \mod{A} \to \mathrm{sVect}_k$. Then, $Q_*(B)$ is a separable commutative algebra in $\mathrm{sVect}_k$, so by Lemma~\ref{lem-sep-supervectorspace} below, we must have $Q_1(B)=0$. 
 It then follows from~\cite{Mathew2016}*{Lemma 6.32} that $\pi_1(B)=0$ and that 
 $\pi_0(B)$ is a free $\pi_0(A)$-module, necessarily finite. This in particular implies that 
 $A \to B$ is flat and that $\pi_0(B \otimes_A B)\simeq \pi_0(B)\otimes_{\pi_0(A)}\pi_0(B)$ so that $\pi_0(B)$ is a separable $\pi_0(A)$-algebra. It then follows from~\cite{Ford}*{Theorem 8.3.6} that $\pi_0 (A) \to \pi_0 (B)$ is formally unramified and hence \'etale.
\end{proof}

\begin{lemma}\label{lem-sep-supervectorspace}
Let $k$ be a field of characteristic $p>2$ and let $A= (A_0, A_1) \in \CAlg(\mathrm{sVect}_k)$ be separable. Then $A_1=0$ and $A_0$ is a finite product of finite separable field extensions of $k$.
\end{lemma}

\begin{proof}
  We can argue in the same way as in the proof of~\cite{Neeman2018}*{Lemma 1.4} and see that $A_1=0$. From this it follows that $A_0$ is a separable $k$-algebra and so a finite product of finite separable field extensions of $k$ by~\cite{Neeman2018}*{Proposition 1.6}.
\end{proof}

\section{Complete modules over an adic \texorpdfstring{$\E_\infty$}{E}-ring}\label{sec-complete}

In this section we recall the $\infty$-category of complete modules over a connective $\E_\infty$-ring and we classify its dualizable and separable commutative algebras. We start by recalling some background and terminology following~\cite{SAG}*{Chapter 8.1}.

\begin{definition}
 An \emph{adic} $\E_\infty$-ring is a connective $\E_\infty$-ring $A$ with an adic topology on the commutative ring $\pi_0(A)$ admitting a finitely generated ideal $I \subseteq  \pi_0(A)$ of definition. 
\end{definition}

Given a ideal $I$ in a discrete commutative ring $\pi_0(A)$, one can form the $I$-adic tower of quotient rings
\[
 \ldots \to \pi_0(A)/I^n \to \ldots \to \pi_0(A)/I^2 \to \pi_0(A)/I.
\]
This tower governs the theory of $I$-adic completion for $\pi_0(A)$-modules. 
The next results provides a similar construction applicable to all adic $\E_\infty$-rings. 

\begin{lemma}\label{lem:adic-tower}
 Let $A$ be an adic $\E_\infty$-ring with finitely generated ideal of definition $I\subseteq \pi_0(A)$. Then, there exists a tower $\ldots\to  A_4 \to A_3 \to A_2 \to A_1 $ in 
 the $\infty$-category $\CAlg_{A}$ of $A$-algebras satisfying the following properties:
 \begin{itemize}
     \item[(a)] Each $A$-algebra $A_i$ is connective and each map $A_{i+1}\to A_i$ induces a surjection on $\pi_0$ with nilpotent kernel.
     \item[(b)] For every connective $A$-algebra $B$, the canonical map $$\colim\Hom_{\CAlg}(A_n, B)\to \Hom_{\CAlg}(A,B)$$ induces a homotopy equivalence between the source and the summand of the target consisting of those maps $A \to B$ which annihilate some power of $I$.
     \item[(c)] Each $\mathbb E_\infty$-ring $A_n$ is almost perfect when regarded as an $A$-module. 
 \end{itemize}
 Moreover we can also arrange that the given map induces an isomorphism 
 \[ \pi_0(A)/I\stackrel{\simeq}{\longrightarrow}\pi_0(A_1).\]
\end{lemma}

\begin{proof}
 The existence of a tower satisfying (a), (b) and (c) is proved in~\cite{SAG}*{Lemma 8.1.2.2}, and we note that the construction given there has the additional property that $\pi_0(A_1)\simeq \pi_0(A)/I$. 
\end{proof} 
As in classical commutative algebra, there is a notion of completeness for modules over an adic $\E_\infty$-ring $A$. If $I$ is an ideal of definition of $\pi_0(A)$, there is an associated full subcategory $\mod{A}^{\Cpl(I)}\subseteq \mod{A}$ of $I$-complete $A$-modules~\cite{SAG}*{Definition 7.3.1.1}. The inclusion $\mod{A}^{\Cpl(I)}\subseteq \mod{A}$ admits a left adjoint~\cite{SAG}*{Notation 7.3.1.5} which we refer to as the $I$-completion functor.  The $\infty$-category $\mod{A}^{\Cpl(I)}$ is a stable homotopy theory where the symmetric monoidal structure is given by the completed tensor product~\cite{SAG}*{Corollary 7.3.5.2} (presentability and stability of $\mod{A}^{\Cpl(I)}$ follows by combining~\cite{SAG}*{Proposition 7.1.1.12 and Proposition 7.3.1.7}.  The next result determines the complete modules which are dualizable with respect to this symmetric monoidal structure.

 \begin{theorem}\label{thm-dual=perfect}
  Let $A$ be an adic $\E_\infty$-ring which is complete with respect to a finitely generated ideal of definition $I \subseteq\pi_0(A)$. Then we have the following equality of full subcategories
 \[ \Perf{A} = \mod{A}^{\Cpl(I),\mathrm{dual}}\subseteq\mod{A}\]
 and the symmetric monoidal structures on $\Perf{A}$ and $\mod{A}^{\Cpl(I),\mathrm{dual}}$ agree.
 \end{theorem}

 \begin{proof}
 The inclusion $\Perf{A}\subseteq\mod{A}^{\Cpl(I),\mathrm{dual}}$ holds because $A$ is complete and $\mod{A}^{\Cpl(I),\mathrm{dual}}\subseteq\mod{A}$ is a thick subcategory.\\
 To see the reverse inclusion we consider some $M\in\mod{A}^{\Cpl(I),\mathrm{dual}}$. To show that $M\in\Perf{A}$, by~\cite{SAG}*{Corollary 8.3.5.9} it suffices to show that $M$ is almost connective and that $M\otimes_A (\pi_0(A)/I)$ is perfect.
Since $M$ is dualizable for the completed tensor product, so is the 
$\pi_0(A)/I$-module $M\otimes_A (\pi_0(A)/I)$. Since $I$ is zero here, this module is dualizable for the ordinary tensor product, i.e. it is perfect.\\
To see that $M$ is almost connective, we choose generators $I=(x_1,\ldots,x_n)$, the corresponding Koszul-objects
$C_k:=(A/x_1)\otimes_A\cdots\otimes_A (A/x_k)$ ($0\leq k\leq n$) and
show by descending induction on $k\leq n$ that $M\otimes_A C_k$ is almost connective. The final case $k=0$ will give the result since $M\otimes_A C_0=M$. To start the induction, we claim that $M\otimes_A C_n$ is a perfect $A$-module, and in particular almost connective. It is clear that $M\otimes_A C_n$ is dualizable for the completed tensor-product, 
say with dual $N$. Since completion is exact, it commutes with $-\otimes_A C_n$, and hence the fiber of the completion map
\[ \mathrm{fib}(M\otimes_A C_n\otimes_A N\longrightarrow 
(M\otimes_A C_n)\widehat{\otimes}_A N)\simeq C_n\otimes \mathrm{fib}(M\otimes_A N\longrightarrow M\widehat{\otimes}_A N)\]
is both $I$-nilpotent (\cite{SAG}*{Definition 7.1.1.6}) and $I$-local (using~\cite{SAG}*{Proposition 7.3.1.4}), and thus vanishes. Hence  
$M\otimes_A C_n$ is dualizable for the ordinary tensor product, as claimed.\\
For the induction step, we assume given $1\leq k\leq n$ and that 
$M\otimes_A C_k$ is almost connective, say it is $r$-connective.
Using the fiber sequences
\[ M\otimes_A C_{k-1}\otimes_A A/x_k^n\longrightarrow
M\otimes_A C_{k-1}\otimes_A  A/x_k^{n+1}\longrightarrow
M\otimes_A C_{k-1}\otimes_A A/x_k=M\otimes_A C_k\]
and induction on $n$ shows, that for every $n$, $M\otimes_A C_{k-1}\otimes_A A/x_k^n$
is $r$-connective. Since $M\otimes_A C_{k-1}$ is $I$-complete, we have
\[ M\otimes_A C_{k-1}\simeq\lim_n M\otimes_A C_{k-1}\otimes_A A/x_k^n,\]
and this is $(r-1)$-connective by the Milnor exact sequence. This concludes the induction and the proof of the claimed equality. 
Regarding the symmetric monoidal structures, note that the unit objects of $\Perf{A}$ and $\mod{A}^{\Cpl(I),\dual}$ agree as $A$ is complete and that the tensor products of two perfect $A$-modules is already complete as it is perfect.

\end{proof}
 
This result implies the following description 
of complete dualizable modules as a limit, which will be the key to determining the separable algebras.

\begin{corollary}\label{lem-compl-mod}
 Let $A$ be an adic $\E_\infty$-ring which is complete with respect to a finitely generated ideal of definition $I \subseteq\pi_0(A)$ and let $\{A_n\}$ be its associated adic tower as given in Lemma~\ref{lem:adic-tower}. Then we have a symmetric monoidal equivalence 
\[
\Perf{A}=\mod{A}^{\Cpl(I),\dual}\stackrel{\simeq}{\longrightarrow} \lim_n \Perf{A_n}.
\]
\end{corollary}

\begin{proof}

    Recall from the proof of~\cite{SAG}*{Theorem 8.3.4.4} that there is an adjunction 
    \[
    F\colon \mod{A}^{\Cpl(I)} \leftrightarrows \lim_n \mod{A_n}:G
    \]
with $F$ symmetric monoidal and given on objects by $F(M)=\{A_n\otimes_A M\}_n$ and $G(\{M_n\})=\lim_n M_n$. The cited result also proves that $(F,G)$ restricts to an equivalence on suitably connective objects, namely 
    \begin{equation}\label{connective-eq}
         F^{\mathrm{cn}}\colon \mod{A}^{\Cpl(I)}\cap \mod{A}^{\mathrm{cn}} \stackrel{\simeq}{\longrightarrow} \lim_n\mod{A_n}^{\mathrm{cn}}:G^{\mathrm{cn}}.
    \end{equation}
    We will prove that the adjunction $(F,G)$ restricts to an equivalence between dualizable objects. Let us first argue why the adjunction restricts to dualizable objects. A direct argument using~\cite{SAG}*{Proposition 7.3.5.1} and the fact that $A_n \otimes_A M$ is $I$-complete for all modules $M$, shows that $F$ is symmetric monoidal and hence it restricts to dualizable objects. 
    The argument for the functor $G$ is more involved. Consider $\{M_n\}\in\lim_n \Perf{A_n}$, we want to show that $\lim_n M_n \in \Perf{A}$. We first observe that the connectedness of the compatible collection $\{M_n\}$ is uniformly bounded below: this is because we have equivalences $A_n \otimes_{A_{n+1}}M_{n+1} \simeq M_n$ with $A_{n}$ and $A_{n+1}$ connective. It follows that there exists $s \geq 0$ such that $\Sigma^s M_n$ is connective for all $n\geq 0$. Then by (\ref{connective-eq}), the $A$-module $G(\Sigma^s M_n)=\lim\Sigma^s M_n$ is $I$-complete and connective. By~\cite{SAG}*{Corollary 8.3.5.9} the $A$-module $\lim\Sigma^s M_n$ is perfect if $(\lim_n \Sigma^s M_n)\otimes_A \pi_0(A_1)$ is a perfect $\pi_0(A_1)$-module (we used that $\pi_0(A_1)=\pi_0(A)/I$). But using the adjoint equivalence~(\ref{connective-eq})
    \[
    (\lim_n \Sigma^s M_n)\otimes_A \pi_0(A_1) = (\lim_n \Sigma^s M_n \otimes_{A} A_1 )\otimes_{A_1} \pi_0(A_1)= \Sigma^s M_1 \otimes_{A_1} \pi_0(A_1)
    \]
    which is perfect since $M_1 \in \Perf{A_1}$. It follows that $G(\{\Sigma^sM_n\})$ and $G(\{M_n\})$ are perfect $A$-modules. This shows that $(F,G)$ restricts to an adjunction 
    \[
    F^\dual \colon \Perf{A}=\mod{A}^{\Cpl(I),\dual} \leftrightarrows \lim_n \Perf{A_n}: G^\dual.
    \]
    
    The fact that $(F^\dual, G^\dual)$ is an adjoint equivalence now follows from (\ref{connective-eq}) since any object in $\Perf{A}$ and $\lim_n \Perf{A_n}$ is bounded below so up to suspending can be assumed to be connective.
\end{proof}

\begin{remark}
There are also variants of Theorem \ref{thm-dual=perfect}  for certain non-connective rings, for example~\cite{Mathew2016}*{Proposition 10.11} which establishes the result for Morava $E$-theories. Mathew points out that this uses the regularity, as one can see by considering the cochain algebra
$A:=C^*(BC_p,E_1)$ with $E_1$ denoting $p$-complete $K$-theory. Then $A$ is a finite flat $E_1$-algebra, 
$p$-complete (but not regular), and we claim that the inclusion
\[ \Perf{A}\subsetneq\mod{A}^{\Cpl((p)),\mathrm{dual}}\]
is proper in this case. To see this, we recall from~\cite{Ambidexterity}*{Corollary 5.4.4}
that the global sections functor is a symmetric monoidal equivalence
\[ \mathrm{Fun}(BC_p,\mod{E_1}^{\Cpl((p))})\stackrel{\simeq}{\longrightarrow} \mod{A}^{\Cpl((p))}.\]
Restricting to dualizable objects, we obtain a symmetric monoidal equivalence
\[ \mathrm{Fun}(BC_p,\Perf{E_1})\simeq\mod{A}^{\Cpl((p)),\mathrm{dual}}.\]
The induced representation $E_1\otimes C_{p,+}$ is an object on the left hand side, which is not contained in the thick subcategory generated by the unit (as one sees after rationalization). Its global sections are $E_1$, considered as an $A$-module via the augmentation $A\longrightarrow E_1$, which hence is an example of a $p$-complete dualizable $A$-module which is not perfect.
\end{remark}

\begin{theorem}\label{thm-complete-sep}
 Let $A$ be an adic $\E_{\infty}$-ring which is complete with respect to a finitely generated  ideal of definition $I\subseteq \pi_0(A)$. Then 
 \[
 \CAlg^{\sep,\mathrm{f}}(\mod{A}^{\Cpl(I),\dual}) \simeq \Cov_{\pi_0(A)/I}.
 \]
\end{theorem}

\begin{proof}
 Let $\{A_n\}$ be the associated adic tower of $A$, as in Lemma \ref{lem:adic-tower}. Recall that 
 \[
 \mod{A}^{\mathrm{Cpl}(I),\dual}\simeq \lim_n \Perf{A_n}
 \]
 and that $\CAlg^{\wcov}(\mod{A_n})=\CAlg^{\sep,\mathrm{f}}(\Perf{A_n})=\Cov_{\pi_0(A_n)}$ by Proposition~\ref{prop-connective-sep}.  
 Therefore by our descent result (Theorem~\ref{thm-descent}) together with Lemma~\ref{lem-wcov-com-limits}, we find that 
 \[
 \CAlg^{\sep,\mathrm{f}}(\mod{A}^{\Cpl(I),\dual})=\left(\lim_n \CAlg^{\sep,\mathrm{f}}(\Perf{A_n})\right)^{\br}=\left(\lim_n \Cov_{\pi_0(A_n)}\right)^{\br}.
 \]
 As the given map $\pi_0(A_n)\to \pi_0 (A_1)$ is a surjection with nilpotent kernel, we have $\Cov_{\pi_0(A_n)}\simeq \Cov_{\pi_0(A_1)}$
 (see~\cite{EGAIV}*{Theorem 18.1.2}), so the transition maps in the limit are all equivalences. It follows that 
  \[
  \CAlg^{\sep,\mathrm{f}}(\mod{A}^{\Cpl(I),\dual})\simeq \Cov_{\pi_0(A_1)}.
  \]
  We conclude by recalling that $\pi_0 (A_1)\simeq \pi_0(A)/I$.
\end{proof}

\section{Chromatic homotopy theory}\label{sec:chromotopy}

In this section we use our descent result to classify separable and dualizable algebras in the setting of chromatic homotopy theory. We have three main examples: the category of modules over Lubin-Tate theories over perfect fields, the category of modules over topological complex and real K-theories and the category of quasi-coherent sheaves on even periodic derived stacks defined over $\spec(\Z[1/2])$. 

\subsection{Chromatic localizations}
In this section we fix a perfect field $k$ of characteristic $p>0$ and a formal group $\mathbb{G}_0$ of height $n$ over $k$. This data determines a Lubin-Tate theory $E(k;\mathbb{G}_0)\in \CAlg$, which we denote by $E(k)$. The homotopy groups of $E(k)$ are non-canonically isomorphic to 
\[
W(k)\llbracket u_1, \ldots,u_{n-1} \rrbracket [u^{\pm 1}]
\]
where $|u_i|=0$, $|u|=2$ and $W(k)$ denotes the ring of Witt vectors of $k$. The Lubin-Tate theory $E(k)$ is even periodic with regular and noetherian $\pi_0(E(k))$ so by Theorem~\ref{thm-sep-even-periodic} we can classify all separable commutative algebras in $\Perf{E(k)}$ provided that $p\not =2$. In this section we will complete this classification result by showing that the same result also holds if $p=2$. 

To this end recall that for each $n\geq 0$, the category of spectra admits localizations functors
\[
\Sp \to \Sp_{T(n)} \to \Sp_{K(n)}
\]
where: 
\begin{itemize}
    \item $\Sp_{T(n)}$ denotes the full subcategory of $\Sp$ spanned by the $T(n)$-local objects. Here $T(n)$ is the telescope of a $v_n$-self map of a type $n$-spectrum. 
    \item $\Sp_{K(n)}$ denotes the full subcategory of $\Sp$ spanned by the $K(n)$-local objects. Here $K(n)$ denotes the \emph{Morava K-theory} spectrum which is defined as
    \[
    K(n)= E(k)/(u_1)\otimes_{E(k)} \ldots \otimes_{E(k)}E(k)/(u_{n-1})\in \mod{E(k)}.
    \]
    Since $E(k)$ is even periodic and the sequence $(p,u_1,\ldots, u_{n-2})$ is regular in $\pi_0(E(k))$, we can choose a structure of $\mathbb E_1$-algebra on $K(n)$ such that $\pi_*(K(n))=k[u^{\pm 1}]$ with $|u|=2$, see~\cite{Algeltveit}*{Section 3}.
\end{itemize}
The above full subcategories admit the structure of stable homotopy theories making the above localizations into symmetric monoidal functors. Since the Lubin-Tate spectrum $E(k)$ is $K(n)$-local, we obtain induced symmetric monoidal localizations
\[
\mod{E(k)} \to \mod{E(k)}(\Sp_{T(n)}) \to \mod{E(k)}(\Sp_{K(n)})
\]
on the corresponding categories of modules. We have the following result which relates the dualizable objects in all these localizations.

\begin{lemma}\label{lem-dual-in-local}
The above functors induces a symmetric monoidal equivalence of $\infty$-categories 
\[
\Perf{E(k)}\simeq \mod{E(k)}(\Sp_{T(n)})^\dual \simeq \mod{E(k)}(\Sp_{K(n)})^\dual.
\]
\end{lemma}

\begin{proof}
    We note that any perfect $E(k)$-module is automatically $K(n)$-local (and so  $T(n)$-local) as $E(k)$ is $K(n)$-local.
    It follows that we have fully faithful functors $\Perf{E(k)}\to \mod{E(k)}(\Sp_{T(n)})^\dual \to\mod{E(k)}(\Sp_{K(n)})^\dual$. The result follows from the fact that
     \[ \left(\mod{E(k)}(\Sp_{K(n)})\right)^{\mathrm{dual}}\simeq \Perf{E(k)}\]
     which is proved in~\cite{Mathew2016}*{Proposition 10.11}, see also Remark~\ref{rem-k-local}.
\end{proof}
Implicitly in the previous proof and often throughout this section we will use the following fact.
\begin{remark}\label{rem-k-local}
Given a Morava $K$-theory spectrum $K(n)$ associated to a Lubin-Tate theory $E(k)$, there are two candidates for the category of $K(n)$-local $E(k)$-modules:
\[
\mod{E(k)}(\Sp_{K(n)}) \quad \mathrm{and} \quad L_{K(n)} \mod{E(k)}
\]
where the right hand side is the Bousfield localization of the category of $E(k)$-modules at $K(n)$. These two categories agree by \cite{Hovey2008}*{Proposition 2.2}. 
\end{remark}

Before diving into our classification result we will need some preliminary results.

\begin{lemma}\label{lem-k-mon-cons}
    Let $K(n)$ be the Morava $K$-theory spectrum attached to $E(k)$. The $K(n)$-homology functor 
    \[
    K(n)_*(-):=\pi_*(K(n)\otimes_{E(k)}-)\colon \mod{E(k)} \to \mod{K(n)_*}
    \]
    is monoidal, and conservative on perfect $E(k)$-modules.
\end{lemma}

\begin{proof}
    We note that for any $E(k)$-module $M$, the homology groups $K(n)_*(M)$ are naturally a graded module over $K(n)_*$ as $K(n)$ admits an $\E_1$-algebra structure, so the functor is well defined. Monoidality follows using the Tor-spectral sequence and the fact that $K(n)_*$ is a graded field. By definition of Bousfield localization, the functor $K(n)_*(-)$ is conservative on $K(n)$-local $E(k)$-modules and so in particular it is conservative on all perfect $E(k)$-module.
\end{proof}

\begin{construction}
Let $\CAlg_k^{\perf}$ denote the category of perfect $k$-algebras, that is those $k$-algebras on which the Frobenius map is an isomorphism. As show in \cite{Nullstellensatz}*{Definition 2.32}, the Lubin-Tate theory extends to a functor 
\[
E(-)\colon \CAlg_k^\perf \to \CAlg(\mod{E(k)}(\Sp_{T(n)}))
\]
in such a way that for any perfect $k$-algebra $P$, we have a non-canonical isomorphism
\[
\pi_*(E(P)) \simeq W(P)\llbracket u_1,\ldots, u_{n-1}\rrbracket [u^{\pm 1}].
\]
\end{construction}

\begin{theorem}\label{thm-Lubin-Tate}
    Let $k$ be a perfect field of characteristic $p>0$. Then any separable commutative algebra in $\Perf{E(k)}$ has finite degree and there are equivalences
    \[
    \Cov_{k}\simeq \Cov_{\pi_0(E(k))}\simeq\CAlg^{\cov}(\mod{E(k)})\simeq\CAlg^\sep(\Perf{E(k)}). 
    \]
\end{theorem}

\begin{proof}
   The equivalence $\Cov_{k}\simeq\Cov_{\pi_0(E(k))}$ follows from the fact that $\pi_0(E(k))$ is a complete local Noetherian ring with residue field $k$ and that the \'etale site is invariant under thickenings, see~\cite{EGAIV}*{Theorem 18.1.2}.
   The second equivalence follows from Matthew's work on Galois theory~\cite{Mathew2016}*{Theorem 6.29}. Moreover by Theorem~\ref{thm-(a)} we have the containment $\CAlg^{\cov}(\mod{E(k)})\subseteq\CAlg^\sep(\Perf{E(k)})$. Therefore it is only left to show that any $A\in \CAlg^\sep(\Perf{E(k)})$ is a finite cover. It will then follows that $A$ has finite degree. 
   
   Let us first deal with the case that $k$ is an uncountable algebraically closed field.  Recall that the $K(n)$-homology functor is monoidal and conservative by Lemma~\ref{lem-k-mon-cons}. We prove by induction on $d=\dim_{K(n)_*} K(n)_*(A)$ that $A \simeq E(k)^d$ as a $E(k)$-algebra. The case $d=0$ follows by conservativity of $K(n)_*(-)$. Now suppose that $d\geq1$. By~\cite{Nullstellensatz}*{Theorem 6.21}, the unit map $E(k) \to A$ admits a retraction so that $A \simeq E(k) \times A'$ as algebras by Theorem~\ref{thm-splitting}. By the induction hypothesis $A'=E(k)^{\times d-1}$ since $\dim K(n)_*(A')=d-1$. It follows that $A\simeq E(k)^d$ and so a finite cover. Thus we proved the theorem for $k$ an uncountable algebraically closed field. 

   If $k$ is a general perfect field, choose an uncountable algebraically closed extension $k \to l$.
   We now claim that the induce map on Lubin-Tate theories $E(k) \to E(l)$ is faithfully flat. Flatness follows from \cite{LurieEll}*{Theorem 6.1.2} provided that we show that the module of K\"ahler differentials $\Omega_{l/k}$ vanishes. Note that $\Omega_{k/\mathbb{F}_p}=0$ since $k$ is a perfect field of characteristic $p$: given an $\mathbb{F}_p$-derivation $d$ on $k$ we see that $d(x^p)=px^{p-1}d(x)=0$ for all $x\in k$, so $d=0$ by surjectivity of the Frobenius map. Since $l$ is also perfect, this gives $\Omega_{l/\mathbb{F}_p}=0$. 
   The transitivity sequence then implies $\Omega_{l/k}=0$ as desired. It is only left to check that $\pi_0(E(k)) \to \pi_0(E(l))$ is faithfully flat. But this follows as it is induced by the base changed map $W(k)\to W(l)$ which is faihtfully flat (as it is a flat map between local rings).   
   It then follows from \cite{Mathew2018}*{Theorem 2.40} that there is an equivalence 
   \[
   \mod{E(k)}\simeq \Tot(\mod{E(l)^{\otimes_{E(k)} \bullet +1}}).
   \]
   We can then apply Proposition~\ref{prop-descent} and the previous paragraph to conclude that 
   \[
   \CAlg^{\cov}(\mod{E(k)})=\CAlg^{\wcov}(\mod{E(k)})=\CAlg^{\sep}(\Perf{E(k)})
   \]
   where the first equality uses the fact that the unit is compact so there is no difference between weak finite covers and finite covers.
\end{proof}

For the rest of this section it will be convenient to write $E_n$ for a Lubin-Tate theory associated to a perfect field $k$ and a formal group of height $n$ over $k$. Let $L_n$ denote the functor of (Bousfield) localization at $E_n$. It is known that $L_n$ is a smashing localization so $L_n \Sp \simeq \mod{L_n S^0}$.

\begin{proposition}
We have 
  \[\CAlg^{\sep}(\Perf{L_n S^0})= \CAlg^{\wcov}(\mod{L_n S^0})\simeq \Cov_{\Z_{(p)}}.\]
\end{proposition}

\begin{proof}
  By work of Hopkins-Ravenel~\cite{Ravanel}*{Chapter 8}, the $\E_\infty$-ring $E_n$ is descendable in $\mod{L_n S^0}$. By Theorem~\ref{thm-Lubin-Tate} we know that 
  \[
  \CAlg^{\sep}(\Perf{E_n})=\CAlg^{\cov}(\mod{E_n})=\CAlg^{\wcov}(\mod{E_n})
  \]
  using that there is no difference between weak finite covers and finite covers as the unit is compact. 
  Thus Corollary~\ref{cor-descent-sep} applies to give 
  \[
  \CAlg^{\sep}(\Perf{L_n S^0})=\CAlg^{\wcov}(\mod{L_n S^0})
  \]
  proving the first claim. The second claim is proved in~\cite{Mathew2016}*{Theorem 10.15}.
\end{proof}

We have already discussed that associated to $E_n$ there is a Morava $K$-theory spectrum $K(n)$. We refer the reader to \cite{Mathew2016}*{section 10.2} for more background, and in particular for the definition of the extended Morava stabilizer group $\mathbb G^{\mathrm{ext}}_n$. The $\infty$-category of $K(n)$-local spectra $\Sp_{K(n)}$ is an example of a stable homotopy theory where the unit object is not compact. We now classify all its separable commutative algebras with underlying dualizable module. To do this, we use the following result from the proof of \cite{Mathew2016}*{Theorem 10.9}.

\begin{lemma}\label{lem:akhils-claim}
The canonical inclusion as the classically \'etale algebras
\[ \Cov_{\pi_0(E_n)} \subseteq \CAlg^\wcov(\mod{E}(\Sp_{K(n)}))\]
is an equality.
\end{lemma}

\begin{proof}
We have a chain of inclusions and equalities
\[ \CAlg^\wcov(\mod{E_n}(\Sp_{K(n)}))\stackrel{}{\subseteq}
\CAlg^{\mathrm{sep}}\left(\mod{E_n}(\Sp_{K(n)}))^{\mathrm{dual}}\right)\stackrel{}{=}\]
\[\stackrel{}{=} \CAlg^{\mathrm{sep}}(\Perf{E_n})\stackrel{}{=} \Cov_{\pi_0(E_n)},\]
given in turn by Corollary \ref{cor-wcov-sep-dual}, \cite{Mathew2016}*{Proposition 10.11}, and Theorem \ref{thm-Lubin-Tate}. 
\end{proof}

\begin{proposition} We have
    \[\CAlg^{\sep}(\Sp_{K(n)}^\dual)=\CAlg^\wcov(\Sp_{K(n)}),\]
    and these are classified by the the extended Morava stabilizer group $\mathbb{G}_n^{\mathrm{ext}}$.
\end{proposition}

\begin{proof}
    We will show that $\CAlg^{\sep}(\Sp_{K(n)}^\dual)=\CAlg^\wcov(\Sp_{K(n)})$. The classification of the weak finite covers in the $K(n)$-local category is then due to Mathew~\cite{Mathew2016}*{Theorem 10.9}.
    To prove the claim we use that $E_n$ is descendable in $\Sp_{K(n)}$ (see ~\cite{Mathew2016}*{Proposition 10.10}) and apply Corollary~\ref{cor-descent-sep}, which reduces us to to verifying that
    \[ \CAlg^{\sep}\left(\left(\mod{E_n}(\Sp_{K(n)})\right)^{\mathrm{dual}}\right) =\CAlg^\wcov(\mod{E_n}(\Sp_{K(n)})).\]
     By Lemma~\ref{lem-dual-in-local} there is a symmetric monoidal equivalence 
     \[ \left(\mod{E_n}(\Sp_{K(n)})\right)^{\mathrm{dual}}\simeq \Perf{E_n}.\] 
     Theorem~~\ref{thm-Lubin-Tate} implies $\CAlg^{\mathrm{sep}}(\Perf{E_n})=\Cov_{\pi_0(E_n)}$. This agrees with the right hand side by Lemma \ref{lem:akhils-claim} above. The implicit identifications are easily seen to be compatible, since in both cases they are given by the classically \'etale algebras.
     \end{proof}

\subsection{Topological K-theories}
     Let $KU$ denote the $\E_\infty$-ring of topological complex $K$-theory whose homotopy groups are given by $\Z[\beta^{\pm 1}]$ with $|\beta|=2$. As $2$ is not a unit in $\pi_0$, we do not get a classification of separable commutative algebras in $\Perf{KU}$ just by the results of the previous subsection. The following result bridges this gap.

\begin{theorem}\label{complex-k-theory}
    The $\E_\infty$-ring $KU$ is separably closed, i.e., the only separable commutative algebras in $\Perf{KU}$ are those of the form $KU^{\times n}$ for $n \geq 0$.
\end{theorem}

\begin{proof}
An easy inspection on homotopy groups shows that    there is a pullback of $\E_\infty$-rings 
    \[
    \begin{tikzcd}
        KU \arrow[r] \arrow[d] & KU[1/2]\arrow[d] \\
        KU_2^\wedge  \arrow[r] &(KU_2^{\wedge})[1/2]
    \end{tikzcd}
    \]
    which induced a squares of $2$-rings
    \[
    \begin{tikzcd}
        \Perf{KU} \arrow[r] \arrow[d] & \Perf{KU[1/2]}\arrow[d] \\
        \Perf{KU_2^\wedge}  \arrow[r] &\Perf{(KU_2^{\wedge})[1/2]}
    \end{tikzcd}
    \]
    which we claim is cartesian. This follows from~\cite{LandTamme}*{Proposition 1.17} and the discussion after the cited result by noting that the canonical map $KU[1/2]\otimes_{KU}KU_2^{\wedge}\to (KU_2^{\wedge})[1/2]$ is an equivalence by the Tor-spectral sequence.

    We now note that $KU_2^\wedge$ is a form of Lubin-Tate theory at the prime $2$ and height $1$, and that $KU[1/2]$ and $(KU_2^{\wedge})[1/2]$ are even periodic $\E_\infty$-rings with regular and Noetherian $\pi_0$ in which $2$ acts invertibly. Therefore in these cases, Theorems~\ref{thm-sep-even-periodic} and \ref{thm-Lubin-Tate} tells us that there is no difference between finite covers and separable commutative algebras with underlying dualizable objects. Therefore by Theorem~\ref{thm-descent} 
    we find that 
    \[
    \CAlg^{\sep,\mathrm{f}}(\Perf{KU})=\CAlg^\wcov(\mod{KU})=\CAlg^\cov(\mod{KU})=\Cov_{\Z}
    \]
    where in the second equality we used Proposition~\ref{prop-wcov=cov} and in the last equality we used~\cite{Mathew2016}*{Theorem 6.29}. We are only left to argue why every separable algebra in $\Perf{KU}$ has finite degree, but this follows from the fact that there is a conservative symmetric monoidal exact functor 
    \[
    \Perf{KU} \to \Perf{KU_2^\wedge}\times \Perf{KU[1/2]}
    \]
    in the target of which every separable commutative algebra has finite degree.
\end{proof}

Next, we consider the real $K$-theory spectrum $KO$ and recall from~\cite{Rognes2008}*{Proposition 5.3.1} that there is a faithful $C_2$-Galois extension $KO \to KU$. 

\begin{theorem}\label{real-k-theory}
There is an equality $\CAlg^\sep(\Perf{KO}) =\CAlg^\cov(\mod{KO})$ and these are precisely those of the form $KO^{\times n} \times KU^{\times m}$ for $n,m\geq 0$.   
\end{theorem}

\begin{proof}
    The equality follows from Corollary~\ref{cor-descent-sep} and the fact that faithful Galois extension are descendable~\cite{Mathew2018}*{Proposition 3.21}. The classification of finite covers is proved in~\cite{Mathew2016}*{Corollary 10.5(2)}. 
\end{proof}

\subsection{Affiness and topological modular forms}

Let $M_{\mathrm{FG}}$ be the moduli stack of formal groups. Let $X$ be a Noetherian and separated Deligne--Mumford stack together with a flat map $X \to M_{\mathrm{FG}}$. 
For any \'etale map $\spec R \to X$, the composition $\spec R \to X \to M_{\mathrm{FG}}$ is again flat, so by the Landweber exact functor theorem there exists a canonically associated even periodic, multiplicative homology theory whose formal group is classified by the map $\spec R\to M_{\mathrm{FG}}$. This defines a presheaf of multiplicative homology theories on $\mathrm{Aff}^{et}_{/X}$, the affine \'etale site of $X$. Following~\cite{MathewLennart}, we call an \emph{even periodic refinement} of this data a pair $\mathfrak{X}=(X,\mathcal{O}^{top})$ where $\mathcal{O}^{top}$ is a sheaf of even periodic $\E_\infty$-rings on $\mathrm{Aff}^{et}_{/X}$ lifting the above diagram of multiplicative homology theories. If we assume in addition that $X$ is regular, then the $\E_\infty$-ring $\mathcal{O}^{top}(\spec R)$ is weakly even periodic and regular. One then defines~\cite{MathewLennart} a stable homotopy theory of quasi-coherent sheaves on $\mathfrak{X}$ by
\[
\QCoh(\mathfrak{X}):=\lim_{(\spec R \to X)\in \mathrm{Aff}^{et}_{/X}} \mod{\mathcal{O}^{top}(\spec(R))},
\]
and an $\E_\infty$-ring of global sections
\[
\Gamma(X,\mathcal{O}^{top}):=\lim_{(\spec R \to X)\in \mathrm{Aff}^{et}_{/X}} \mathcal{O}^{top}(\spec(R)).
\]

\begin{theorem}\label{thm-even-ref-DM}
 Let $X$ be a regular Deligne--Mumford stack defined over $\spec(\Z[1/2])$, let $X \to M_{\mathrm{FG}}$ be a flat, quasi-affine map and assume that there exists an even periodic refinement $\mathfrak{X}=(X,\mathcal{O}^{top})$ as above. Then 
 \[
 \CAlg^{\sep}(\Perf{\Gamma(X, \mathcal{O}^{top})})=\CAlg^{\cov}(\mod{\Gamma(X,\mathcal{O}^{top})}),
 \]
 and these are classified by the \'etale fundamental group $\pi_1^{\mathrm{\acute{e}t}}(X)$. 
\end{theorem}

\begin{proof}
 Note that our assumptions imply that $X$ is Noetherian and separated. Therefore we are under the assumptions of~\cite{MathewLennart}*{Theorem 4.1} so the result applies to give a symmetric monoidal equivalence
 \begin{equation}\label{eq-affiness}
 \mod{\Gamma(X,\mathcal{O}^{top})}\simeq \QCoh(\mathfrak{X}).
 \end{equation}
 By definition, the stable homotopy theory $\QCoh(\mathfrak{X})$ is a limit of module categories over regular, Noetherian, even periodic $\E_\infty$-rings $\mathcal{O}^{top}(\spec R)$ where $2$ acts invertibly. By Theorem~\ref{thm-sep-even-periodic}, separable commutative algebras with underlying perfect module over these $\E_\infty$-rings agree with the (weak) finite covers and these have all finite degree. Our descent result in the form of Theorem~\ref{thm-descent} applies to give 
 \[
 \CAlg^{\sep, \mathrm{f}}(\QCoh(\mathfrak{X})^\dual)= \CAlg^{\wcov}(\QCoh(\mathfrak{X})).
 \]
 Note that in this case there is no difference between weak finite covers and finite covers since the unit object is compact by~(\ref{eq-affiness}). Finally, the (weak) finite covers have already been classified by Mathew in~\cite{Mathew2016}*{Theorem 10.4} by the \'etale fundamental group $\pi_1^{\mathrm{\acute{e}t}}(X)$. It is only left to argue that all separable commutative algebras in $\QCoh(\mathfrak{X})$ with underlying perfect module have finite degree. 
 To this end pick an \'etale cover $\spec(R) \to X$ and observe that the corresponding pullback functor $\QCoh(\mathfrak{X})^\dual\to \Perf{\mathcal{O}^{top}(\spec(R))}$ is conservative. As any separable algebra in the target of the pullback functor has finite degree by Theorem~\ref{thm-sep-even-periodic}, it follows from~\cite{Balmer2014}*{Theorem 3.7(b)} that the same is true for any separable algebra in $\QCoh(\mathfrak{X})^\dual$.
\end{proof}

\begin{example}[Non-periodic topological modular forms] 
 We consider $\overline{M_{ell}}$, the compactification of the moduli stack $M_{ell}$ of elliptic curves. As discussed
 in ~\cite{MathewLennart}*{Section 7}, work of Goerss-Hopkins-Miller-Lurie gives an even periodic refinement $\mathcal{O}^{top}$ on $\overline{M_{ell}}$ whose $\mathbb E_\infty$-ring of global section is the spectrum $\mathrm{Tmf}$ of non-connective, non-periodic topological modular forms. The proof of~\cite{MathewLennart}*{Theorem 7.2} shows that the map $\overline{M_{{ell}}}\to M_{\mathrm{FG}}$ is quasi-affine. We can now base change to $\Z_{(p)}$ for some prime number $p\neq 2$ and obtain $\mathrm{Tmf}_{(p)}$ as the $\mathbb E_\infty$-ring of global sections. Combining Theorem \ref{thm-even-ref-DM} with~\cite{Mathew2016}*{Corollary 10.5(1)} we find that the separable commutative algebras with perfect underlying module over $\mathrm{Tmf}_{(p)}$ are classified by the \'etale fundamental group of $\Z_{(p)}$.
\end{example}

\section{Spectral Deligne--Mumford stacks}\label{sec:spectral-DM}

In this section we classify separable commutative algebras with underlying perfect module in the stable homotopy theory of quasi coherent sheaves on a spectral Deligne--Mumford stack. We start by mildly elaborating on the relevant definition ~\cite{SAG}*{Definition 1.4.4.2}.

\begin{definition}
A \emph{non-connective spectral Deligne--Mumford stack} is a spectrally ringed $\infty$-topos $\mathsf{X}=(\mathcal{X}, \mathcal{O}_\mathcal{X})$ satisfying the following conditions:
\begin{itemize}
    \item[i)] There is a collection of objects $U_\alpha\in {\mathcal X}$ covering ${\mathcal X}$ and such that
    \item[ii)] For each $\alpha$ there is an equivalence $(\mathcal{X}_{/U_\alpha}, \mathcal{O}_\mathcal{X}|U_\alpha)\simeq \mathrm{Sp\acute{e} t}(A_{\alpha})$ for some $\mathbb E_\infty$-ring $A_\alpha$.
\end{itemize}

A {\em spectral Deligne Mumford-stack} is a non-connective spectral Deligne-Mumford stack $\mathsf{X}=(\mathcal{X}, \mathcal{O}_\mathcal{X})$ with connective structure sheaf $\mathcal{O}_\mathcal{X}=\tau_{\ge 0}\mathcal{O}_\mathcal{X}$.
We observe that in this case, all $A_\alpha$ appearing in ii) above are connective: Since 
\[ \tau_{\ge 0}({\mathcal O}_{\mathcal X}|U_\alpha)\simeq(\tau_{\ge 0}{\mathcal O}_{\mathcal X})|_{U_\alpha}\simeq{\mathcal O}_{\mathcal X}|_{U_\alpha},\]
we see that $\mathrm{Sp\acute{e} t}(A_{\alpha})$ is a spectral Deligne-Mumford stack, hence $A_\alpha$ is connective by~\cite{SAG}*{Corollary 1.4.5.3}.
\end{definition}

\begin{example}
Any Deligne--Mumford stack canonically determines a spectral Deligne--Mumford stack by~\cite{SAG}*{Remark 1.4.8.3}. In particular, we can view any qcqs scheme as a spectral Deligne--Mumford stack.
\end{example}

\begin{definition}
Given a spectral Deligne--Mumford stack $\mathsf{X}$, one associates with it a stable homotopy theory of quasi-coherent sheaves by setting 
\begin{equation}\label{eq-qcoh}
\QCoh(\mathsf{X})=\lim_{\spec(R) \to \mathsf{X}}\mod{R}
\end{equation}
where the limits runs over all maps $\spec(R) \to \mathsf{X}$ with $R$ a connective $\E_\infty$-ring. For more discussion see~\cite{SAG}*{Proposition 6.2.4.1}.
\end{definition}

We recall the following result.

\begin{lemma}[\cite{SAG}*{Proposition 6.2.6.2}]
 Let $\mathsf{X}$ be a spectral Deligne-Mumford stack, and consider some $\mathcal{F}\in \QCoh(\mathsf{X})$. Then, the following are equivalent:
 \begin{itemize}
     \item[(a)] $\mathcal{F}$ is dualizable in $\QCoh(\mathsf{X})$;
     \item[(b)] $\mathcal{F}$ is {\em perfect}, that is $\mathcal{F}(\spec(R))$ is a perfect $R$-module for all $\spec(R) \to \mathsf{X}$.
 \end{itemize}
\end{lemma}

\begin{example}
 For any Deligne--Mumford stack $X$,
 the homotopy category of $\QCoh(X)$ identifies with $\mathsf{D}_{qc}(X)$, the derived category of complexes of $\mathcal{O}_X$-modules with quasi-coherent cohomology, see ~\cite{HallRydh}*{Remark 1.7} for details.
 \end{example}

\begin{theorem}\label{thm-sep-SDM-stack}
 Let $\mathsf{X}=(\mathfrak{X}, \mathcal{O}_\mathfrak{X})$ be a spectral Deligne--Mumford stack. Then $$\CAlg^{\sep,\mathrm{f}}(\mathrm{QCoh}(\mathsf{X})^\dual)=\CAlg^{\wcov}(\QCoh(\mathsf{X})),$$
 and this category identifies with the category of finite \'etale covers of the underlying classical stack of $\mathsf{X}$. If $\mathsf{X}$ is quasi-compact, then any separable commutative algebra in $\mathrm{QCoh}(\mathsf{X})^\dual$ has finite degree.
\end{theorem}

\begin{proof}
 By (\ref{eq-qcoh}), the $\infty$-category of quasi-coherent sheaves is a limit of module categories over connective $\E_\infty$-rings $R$. For any such $R$, we have 
 \[ \Cov_{\pi_0(R)} =  \CAlg^{\wcov}(\mod{R}) =\CAlg^{\sep}(\Perf{R})\]
 by Proposition~\ref{prop-connective-sep}. In particular, all the categories are invariant under passage from $R$ to $\pi_0(R)$, i.e. under passage to the underlying classical stack. Therefore the first claim follows from Theorem~\ref{thm-descent}, and the second is clear. 
 For the final claim, suppose that $\mathsf{X}$ is quasi-compact so that we can find a cover $\mathrm{Sp\acute{e} t}(R) \to \mathsf{X}$ where $R$ is a connective $\E_\infty$-ring. 
 Note that the corresponding pullback functor $\QCoh(\mathsf{X})^\dual\to \Perf{R}$ is conservative. As any separable algebra in the target of the pullback functor has finite degree by Proposition~\ref{prop-connective-sep}, it follows from~\cite{Balmer2014}*{Theorem 3.7(b)} that the same is true for any separable algebra in $\QCoh(\mathsf{X})^\dual$. 
\end{proof}

\begin{remark}\label{rem-no-diff-cov}
  If $\mathsf{X}=(\mathfrak{X}, \mathcal{O}_{\mathfrak{X}})$ is a perfect stack in the sense of~\cite{SAG}*{Definition 9.4.4.1}, then the structure sheaf $\mathcal{O}_{\mathfrak{X}}$ is compact in $\QCoh(\mathsf{X})$ so there is no difference between weak finite covers and finite covers. Important examples of perfect stacks are:
  \begin{itemize}
      \item (stack-theoretic) quotients $X/G$ for $G$ an affine algebraic group defined over a field $k$ of characteristic zero and $X$ is quasi-projective $k$-scheme, see~\cite{SAG}*{Example 9.4.4.4}.
      \item quasi-compact and quasi-separated schemes by~\cite{SAG}*{Proposition 9.6.1.1}.
  \end{itemize}
  For more example we refer to~\cite{HallRydh}.
\end{remark}

We note that the next result may alternatively
be deduced form ~\cite{Neeman2018}*{Theorem 7.10}, using Noetherian approximation. 

\begin{corollary}\label{cor-sep-scheme}
 Let $X$ be a qcqs scheme. Then 
 \[
 \CAlg^{\sep}(\QCoh(X)^\dual)= \CAlg^{\cov}(\QCoh(X)),
 \]
 and these are classified by the the \'etale fundamental group $\pi_1^{\mathrm{\acute{e}t}}(X)$. 
\end{corollary}

\begin{proof}
 The first part of the claim follows from Theorem~\ref{thm-sep-SDM-stack} together with the fact that $X$ is perfect so there is no difference between weak finite covers and finite covers. Note that all separable algebras have finite degree by Theorem~\ref{thm-sep-SDM-stack} and our assumption that $X$ is quasi-compact. Finally, the description of the Galois group is given in~\cite{Mathew2016}*{Discussion after Example 7.19}.
\end{proof}

\section{The stable module category for finite groups of \texorpdfstring{$p$}{p}-rank one}\label{sec:stable-mod}

In this section we classifiy all commutative separable algebras of finite degree in the small stable module category of a finite group of $p$-rank one. This in particular answers affirmatively a question of Balmer in this case. We fix a field $k$ of characteristic $p>0$ and a finite group $G$.

\begin{definition}
 We denote by $\mod{G}(k)$ the $\infty$-category $\Fun(BG,\mod{}(k))$. 
 The pointwise tensor product turns this $\infty$-category into a stable homotopy theory with indecomposable unit object. Therefore $\mod{G}(k)$ is a connected stable homotopy theory.
\end{definition}

As discussed in~\cite{Mathew2015}*{Section 2}, the category of $k[G]$-modules admits a combinatorial stable model structure in which the fibrations are the surjections, the cofibrations are the injections and the weak equivalences are the stable equivalences. The $k$-linear tensor product makes the category of $k[G]$-modules into a symmetric monoidal model category. 

\begin{definition}
 The stable module category $\StMod(G;k)$ is the $\infty$-categorical localization of 
 the category of $k[G]$-modules at the class of stable equivalences. This inherits the structure of a symmmetric monoidal $\infty$-category by~\cite{HA}*{Proposition 4.1.7.4.}. We denote by $\stmod(G;k)$ the full subcategory of compact objects; this also agree with the subcategory of dualizable objects.  
\end{definition}

\begin{remark}
  The stable module category $\StMod(G;k)$ is a stable homotopy theory. 
  It follows from Lemma~\ref{lem-dual-is-2ring} that $\stmod(G;k)$ is a $2$-ring. One calculates that
 \[
 \pi_0(k)=\widehat{H}^0(G;k)=\mathrm{cok}(k \xrightarrow{|G|}k),
 \]
  so $\StMod(G;k)$ is a connected stable homotopy theory which is trivial if and only if $p$ does not divide the order of $G$. 
\end{remark}

\begin{definition}\label{def-A_H^G}
For every finite $G$-set $X$, we define a commutive algebra in $k[G]$-modules by $A^G_X:= k[X]$ with multiplication $\mu \colon A^G_X \otimes A^G_X \to A^G_X$ 
and unit $\eta\colon k \to A_X^G$ obtained by extending $k$-linearly the formulas $\mu(x \otimes x)=x$ and $\mu(x \otimes x')=0$ for all $x \not =x' \in X$, and $\eta(1)= \sum _{x \in X}x\otimes x$.
For a subgroup $H\subseteq G$, we ease the notation by setting $A^G_H:=A^G_{G/H}$. Note that $A^G_X$ can also be viewed as a commutative algebra in $\mod{G}(k)$ and $\stmod(G;k)$.
\end{definition}

The next result was first observed by Balmer in~\cite{Balmer2015}*{Proposition 3.16}.
\begin{lemma}\label{lem-all-sep-alg-fin-degree}
 For every finite $G$-set $X$, the commutative algebra $A^G_X$ is separable of finite 
 degree in $\mod{G}(k)$ and $\stmod(G;k)$.
\end{lemma}

\begin{proof}
It is easy to check that $A_H^G$ is separable with bimodule section $\sigma\colon A_H^G \to A_H^G \otimes A_H^G$ given by $\sigma(x)=x\otimes x$ for all $x\in G/H$.  
The discussion in~\cite{Balmer2014}*{Example 4.6} shows that for any $H\subseteq G$, the commutative algebra $A^G_H$ has finite degree in $\mod{G}(k)$ and $\stmod(G;k)$.

For a general finite $G$-set, we can write $A_X^G\simeq \prod_i  A_{H_i}^G$ by decomposing $X$ into its orbits. Then $A_X^G$ is separable by Example~\ref{ex-separable-product-retract} and of finite degree by Lemma~\ref{lem-product-fin-degree}.
\end{proof}

We now recall the determination of finite covers in $\mod{G}(k)$.

\begin{theorem}[Mathew] \label{thm-galois-group-mod_G(k)}
 Let $k$ be a separably closed field of characteristic $p>0$ and let $G$ be a 
 finite group. 
 Then the Galois group $\pi_{1}(\mod{G}(k))$ is isomorphic to the quotient of $G$ by the normal 
 subgroup $N_p(G)\trianglelefteq G$ generated by the elements of order $p$ \footnote{The group $N_p(G)$ is traditionally denoted by $\Omega_1(G)$ in group
theory}. 
 Furthermore, every object of $\CAlg^\cov(\mod{G}(k))$ is of the form $A_X$ for a finite 
 $G$-set $X$ with trivial $N_p(G)$-action.
\end{theorem}

\begin{proof}
 The calculation of the Galois group is due to Mathew~\cite{Mathew2016}*{Theorem 7.16}.  
 In the proof the author shows that the Galois covers are precisely 
 those of the form $A_N^G$ for some normal subgroup $N\trianglelefteq G$ containing 
 $N_p(G)$. Our claim about the shape of a general finite covers follows formally from this: Given $A\in \CAlg^\cov(\mod{G}(k))$, there exist Galois covers $\{A_{N_\alpha}^G\}$ and subgroups $ N_\alpha \subseteq H_\alpha$ such that 
 $A\simeq \prod_\alpha (A_{N_\alpha}^G)^{hH_\alpha}$, see (the proof of) Theorem~\ref{thm-fin-cover-sep}. Here $H_\alpha$ acts on $A_{N_\alpha}^G$ on the right by $[g]_{N_\alpha} .h=[gh]_{N_\alpha}$; one can check that this action is well-defined since $N_\alpha$ is normal in $G$. For each $\alpha$, we find that 
 \[
 (A_{N_\alpha}^G)^{hH_\alpha}\simeq (\prod_{G/N_\alpha} k)^{hH_\alpha}\simeq \prod_{G/N_\alpha H_\alpha}k  \simeq A^G_{N_\alpha H_\alpha}=A^G_{H_\alpha}
 \]
 as $N_\alpha \subseteq H_\alpha$. We conclude that $A\simeq A^G_X$ for $X=\coprod_{\alpha} A^G_{ H_{\alpha}}$.

\end{proof}

The goal of this section is to settle the following question posed by Balmer~\cite{Balmer2015}*{Question 4.7} in the simplest non-trivial case. 

\begin{question}\label{que-Balmer-calrson}
 Let $k$ be a separably closed field of positive characteristic and $G$ a finite group, and let $A$ be a separable commutative 
 algebra in $\stmod(G;k)$. Is there a finite $G$-set $X$ such that $A\simeq A^G_X$ in $\stmod(G;k)$? 
 \end{question}

Equivalently, we can ask if the only indecomposable separable algebras in 
 $\stmod(G;k)$ are those of the form $A^G_H$ for some subgroup $H \subseteq G$. Observe that by Theorem~\ref{thm-galois-group-mod_G(k)}
 this is also equivalent to asking if every separable algebra in $\stmod(G;k)$ lifts to $\mod{G}(k)$.

We remind the reader of the computation of the Balmer spectrum of the stable module category.  Let $k$ be a field of characteristic $p>0$ and let $G$ be a finite group. Let $H^\bullet(G;k)$ denote the commutative ring $H^*(G;k)$ if $p=2$, and $H^{\mathrm{even}}(G;k)$ otherwise, and write $\CV_G$ for the projective support 
 variety $\mathrm{Proj}(H^\bullet(BG;k))$. Note that $\CV_G$ is Noetherian since $H^\bullet(G;k)$ is so~\cite{Benson}*{Corollary 4.2.2}.

\begin{theorem}[{\cite{BCR1997}*{Theorem 3.4}}]
There is a homeomorphism $\Spc(\stmod(G;k))\simeq \CV_G$. 
\end{theorem}

We next need a group-theoretical observation about finite groups of $p$-rank one.
Recall that the $p$-rank of a finite group $G$ is the the largest integer $n$ such that $G$ has an elementary abelian subgroup of order $p^n$. We recall the following classification result, see for instance ~\cite{Gorenstein}*{Sec. 5.4, Thm. 4.10}.

\begin{lemma}\label{lem-class-rank-one}
 Let $G$ be a finite group of $p$-rank one and let $G_p\subseteq G$ be a Sylow $p$-subgroup. 
 Then either $G_p$ is cyclic, or $p=2$ and $G_p$ is a generalized quaternion group. Furthermore, all maximal elementary abelian $p$-subgroups of $G$ are conjugate. 
\end{lemma}

\begin{comment}
\begin{proof}
 The classification of $p$-rank one groups is well-known, see for instance ~\cite{Gorenstein}*{Sec. 5.4, Thm. 4.10}.  We now claim that if $G$ has $p$-rank one, then any $p$-Sylow subgroup contains precisely one copy of $\Z/p$. By uniqueness and the fact that $p$-Sylow subgroups are all conjugate, it then follows that all maximal elementary abelian subgroups are also conjugate. To prove the claim we use the aforementioned classification. If $G_p$ is cyclic, then it is clear that there is only one copy of $\Z/p$ in $G_p$. If $p=2$ and $G_p$ is a generalized quaternion group, then one has to note that the only copy of $\Z/2$ in $G_p$ is the center. 
\end{proof}
\end{comment}

\begin{corollary}\label{cor-spectrum-point-rank-one}
  Let $G$ be a finite group and let $k$ be a field of characteristic $p$, which divides the order of $G$.  
  Then $\CV_G$ is discrete if and only if $\CV_G \simeq *$ if and only if $G$ has $p$-rank one. 
\end{corollary}

\begin{proof}
We have already noted that $\StMod(G;k)$ is a connected stable homotopy theory so its Balmer spectrum must be connected by Proposition~\ref{prop-char-connected-Balmer}. 
This immediatelly gives that $\CV_G$ is discrete if and only if $\CV_G \simeq *$. 
Let us now prove the other implications. Recall from~\cite{Quillen1971}*{Corollary 7.8} that the Krull dimension of $H^\bullet(BG;k)$ coincides with the $p$-rank of $G$, and that $\dim \CV_G=\dim H^\bullet(BG;k)-1$. So if $\CV_G$ is discrete, then $\dim H^\bullet(BG;k)=1$ and so $G$ must have $p$-rank one.
Conversely if $G$ has $p$-rank one, then $\dim H^\bullet(BG;k)=1$ and so $\dim\CV_G=0$.  We conclude by noting that any $0$-dimensional Noetherian scheme is discrete~\cite{EisenbudHarris}*{Exercise I.36}.
\end{proof}

The next result is our classification of finite separable algebras. In the special case that $G$ is cyclic, this is due to Balmer-Carlson~\cite{BC2018}. Observe however,
that they prove more in this special case as they do not have to assume from the outset that their algebras have finite degree. We do not know if this finiteness condition hold in our situation, too. 

\begin{proposition}\label{prop-sep-cov-rank-one-case}
 Let $k$ be a field of characteristic $p>0$ and let $G$ be a finite group of $p$-rank one. 
 Then there is an equality 
 \[
 \CAlg^\cov(\StMod(G;k))=\CAlg^{\sep,\mathrm{f}}(\stmod(G;k))
 \]
 between the finite covers and the separable and compact commutative algebras 
 of finite degree.
\end{proposition}

\begin{proof}
 Since $G$ has $p$-rank one then $\Spc(\stmod(G;k))=*$ by 
 Corollary~\ref{cor-spectrum-point-rank-one}. 
 It follows that any separable algebra of finite degree $A$ has locally constant degree 
 function $\deg(A) \colon \Spc(\stmod(G,k)) \to \Z$. Therefore the claim follows from 
 Corollary~\ref{cor-sep-lcf=cov}.
\end{proof}

The following example demonstrates that if the $p$-rank of $G$ is not one, there are examples of separable algebras which are not finite covers. So our approach of classifying separable algebras via Galois theory will not give a conclusive answer in this case.

\begin{example}\label{ex-wc-not-lc-degree}
Let $G$ be elementary abelian $p$-group of rank $2$ and let $E \subseteq G$ an elementary abelian subgroup of rank $1$. Then we have seen that $A^G_E$ is separable of finite degree in $\stmod(G;k)$. We observe that the degree function of $A^G_E$ is not locally constant on $\CV_G$ as it is nonzero exactly on the image of the inclusion $*=\CV_E \to\CV_G$. It then follows that $A^G_E$ is not a finite cover.    
To prove the claim about the degree function, it suffices to note that $A_H^G$ is equivalent to $\mathrm{Coind}_H^G(k)$ and apply~\cite{Balmersurj}*{Theorem 1.7}. Recall also that $A_E^G$ is a weak finite cover in $\mod{G}(k)$ so the same argument as above shows that not all weak finite covers have locally constant degree function. Note we have just given our first example of some $\C:=\StMod_G(k)\in\CAlg(\Pr)$ for which the inclusion of full subcategories
\[ \CAlg^{\mathrm{w.cov}}(\C)\subsetneq\CAlg^{\mathrm{sep}}(\C^{\mathrm{dual}})\]
is proper. We also observe that \cite{Mathew2016}*{Theorem 9.2} exhibits $\C$ as a finite limit of module categories of localizations $R$ of cochain-algebras.\footnote{This can be thought of as a categorification of the Greenlees-May spectral sequence from local cohomology to Tate cohomology.} Consequently, we will also have 
\[ \CAlg^{\mathrm{w.cov}}(\mod{R})\subsetneq\CAlg^{\mathrm{sep}}(\Perf{R})\]
for some such $R$.

\end{example}

By Proposition \ref{prop-sep-cov-rank-one-case}, we can use Galois theory to classify all separable 
algebras of finite degree in $\stmod(G;k)$ for $G$ a finite group of $p$-rank one. We know as a special case of ~\cite{Mathew2016}*{Corollary 9.19} that the Galois group of the stable module category is the Weyl-group of a copy of $\mathbb Z/p\subseteq G$. We next observe that the corresponding $W$-torsor is induced by the canonical action of the Weyl-group on $A^G_{\mathbb{Z}/p}$. We remark that this action happens already in the module category, but in general, it determines a $W$-Galois extension only after passage to the stable module category.

\begin{lemma}\label{lem-rank-one-galois}
Let $k$ be a field of characteristic $p>0$ and let $G$ be a finite group of $p$-rank one. 
Let $P$ denote a maximal elementary abelian subgroup of $G$ so that $P\simeq \Z/p$. 
Write $N$ for the normalizer of $P$ in $G$, and $W$ for the Weyl group of $P$ in $G$ so that $W=N/P$. Then $A^G_P=k[G/P]$ is an indecomposable and descendable $W$-Galois extension in $\stmod(G;k)$. 
\end{lemma}

\begin{proof}
 Note that $A^G_P$ has a canonical left $G$-action by left multiplication, making it an
 object of $\stmod(G;k)$, but it has also a right $N$-action by right multiplication
 via $[g]_P.n=[gn]_P$ for all $g \in G$ and $n\in N$. 
 Clearly $P \subseteq N$ acts trivially on $A^G_P$ 
 so we get an induced $W=N/P$-action. The $W$ and $G$ actions are compatible to one 
 another and so they make $A_P^G$ into an object with a $W$-action in $\stmod(G;k)$. 
 
% Let us now show that $A^G_P$ has degree $w=|W|$. 
 %By~\cite{Bregje}*{Proposition 10.13(a)} the order of $A^G_P$ is the maximum $n$ 
 %such that there exist distinct elements $[g_1],\ldots, [g_n]$ in $P\backslash G$ with $p$
 %dividing $P^{g_1} \cap \ldots \cap P^{g_n}$. Since $|P^{g_i}|=p$, the condition that $p$ 
 %divides $P^{g_1} \cap \ldots \cap P^{g_n}$ is equivalent to $P^{g_1}=\ldots = P^{g_n}$.
 %By conjugating $P^{g_1}=\ldots = P^{g_n}$ further by $g_1^{-1}$, we can assume that $g_1=1$ 
 %so that $P=P^{g_2}=\ldots = P^{g_n}$. Thus $g_i\in N$.
 %Thus we can find at most $w$ distinct elements $[g_1], \ldots, [g_{w}] \in P \backslash N=
 %W\subseteq P \backslash G$ such that $P^{g_1} \cap \ldots \cap P^{g_{w}}=P$. This tells us that $\deg(A_P^G) =w$. 
 
 Note that $A^G_P$ is descendable by the second part of 
 Lemma~\ref{lem-descend-degree-geq1} and the fact that $\CV_G=\ast$. Moreover, $A_P^G$ is indecomposable since 
 \[
 \pi_0(A_P^G)=\widehat{H}^0(P;k)=k/|P|=k
 \]
 by our assumption on the characteristic of the field.   
 
 Let us now show that the map $h\colon A^G_P \otimes A^G_P \to\prod_W A^G_P$ from
 Definition~\ref{def-galois-extension} is an equivalence. Let $T$ be a complete 
 set of representatives for the double coset $P\backslash G /P$. 
 The double coset formula tells us that the following $G$-equivariant map is an isomorphism:
 \begin{equation}\label{double-coset-formula}
 \coprod_{g\in T} \beta_g \colon \coprod_{g\in T} G/P \cap P^g \xrightarrow{\sim} G/P \times G/P, \quad \beta_g([x]_{P \cap P^g})=([x]_P, [xg^{-1}]_P)
 \end{equation}
 see~\cite{Bregje}*{Remark 10.6}. Since $P$ is normal in $N$, we have 
 $P \backslash N/P=N/P=W$, so we can decompose $T$ as $T= W_0 \cup (T-W_0)$ by picking representatives 
 $W_0$ for $W$. 
 We note that $P\cap P^g=P$ if $g \in W_0$, and $P\cap P^g=1$ otherwise. 
 The double coset decomposition~(\ref{double-coset-formula}) induces a ring isomorphism 
 in the stable module category
 \[
 \tau \colon A_P^G \otimes A_P^G \xrightarrow{\sim} 
 \prod_{g \in W_0} A^G_P  \times \prod_{g \in T-W_0} A^G_1 \simeq \prod_{g \in W_0} A^G_P 
 \]
 since $A_1^G$ is projective. For all $g\in W_0$ and $x,y \in G$, the map $\tau$ is defined as 
 follows (see~\cite{Bregje}*{Remark 10.6})
 \[
 \mathrm{pr}_g \tau([x]_P \otimes [y]_P)= 
 \begin{cases}
 [x]_P & \mathrm{if}\; \exists\,  p\in P\; \mathrm{s.t}\; y^{-1}xpg^{-1}\in P \\
 0     & \mathrm{otherwise}
 \end{cases}
 \]
 The condition $y^{-1}xpg^{-1}\in P$ can be rewritten as $y^{-1}xg^{-1}(gpg^{-1})\in P$ which is equivalent to $y^{-1}xg^{-1}\in P$ since $gpg^{-1}\in P$. 
 Therefore we can simplify the formula of $\tau$ as follows:
 \begin{equation}\label{tau}
 \mathrm{pr}_g \tau([x]_P \otimes [y]_P)= 
 \begin{cases}
 [x_P] & \mathrm{if} \; y^{-1}xg^{-1}\in P \\
 0     & \mathrm{otherwise}
 \end{cases}
 \end{equation}
 We claim that $\tau$ agrees with the map $h\colon A_P^G \otimes A_P^G \to \prod_W A^G_P$ 
 of the Galois condition. For all $g\in W$ and $x,y \in G$, the map $h$ is given by 
 \begin{equation}\label{h}
 \mathrm{pr}_g h ([x]_P\otimes [y_P])=[x]_P \cdot [yg]_P = 
 \begin{cases}
 [x]_P & \mathrm{if}\; [x]_P=[yg]_P \\
 0   & \mathrm{otherwise}
 \end{cases}
 \end{equation}
 where we used the definition of the multiplication in $A_P^G$ which we have recalled 
 in Definition~\ref{def-A_H^G}. The condition $[x]_P=[yg]_P$ means that there exists $p\in P$
 such that $xp=yg$, and we can rewrite this as $x^{-1}yg \in P$. Conjugating this last 
 formula by $g$, we find that $gx^{-1}y=gx^{-1}ygg^{-1} \in P^g=P$. Finally applying $(-)^{-1}$ 
 to this we find that $y^{-1}xg^{-1}\in P$. But this is precisely 
 the condition defining $\tau$ see~(\ref{tau}). Therefore $h=\tau$ and this is an isomorphism 
 (as $\tau$ was).
 
 It is only left to prove that 
 $\eta \colon k \to (A_P^G)^{hW}$ is an equivalence. Since $A_P^G$ is faithful, it suffices to 
 check that $A_P^G \otimes \eta \colon A_P^G \to A_P^G \otimes (A_P^G)^{hW}$ is 
 an equivalence.  
 By Lemma~\ref{lem-dual-nu-map}, we can rewrite 
 \[
 A_P^G \otimes(A_P^G)^{hW} \simeq (A_P^G \otimes A_P^G)^{hW}\simeq (\prod_W A_P^G)^{hW}
 =A^G_P
 \]
 and so the required map is an equivalence.
\end{proof}

\begin{theorem}\label{thm-stable-module-cat}
 Suppose we are in the situation of Lemma~\ref{lem-rank-one-galois} and that the field $k$ is separably closed.
The Galois group of $\StMod(G;k)$ is isomorphic to $W$ and the resulting equivalence
 \[
 \FinSet_W\simeq \CAlg^\cov(\StMod(G;k))^{\op} %\CAlg^{\sep,\mathrm{f}}(\stmod(G;k))^{\op}
 \]
 can be chosen to send $W$ to $A^G_P$. Furthermore, Question~\ref{que-Balmer-calrson} 
 admits a positive answer in this case, if we assume in addition that the separable algebra has finite degree.
\end{theorem}

\begin{proof}
 The fact that the Galois group is isomorphic to $W$ is proved in~\cite{Mathew2016}*{Corollary 9.19} (together with the second part of Lemma~\ref{lem-class-rank-one} to remove one assumption from the cited result) and so by the Galois correspondence (Theorem~\ref{thm-Galois-correspondence}) we get an 
 equivalence 
 \begin{equation}\label{eq-galois}
  \CAlg^\cov(\StMod(G;k))^{\op}\simeq \FinSet_W.
 \end{equation}
 %Combining these two results we get the equivalence in the theorem. However 
 We next argue that the equivalence behaves as claimed. By Lemma~\ref{lem-rank-one-galois} we know that $A^G_P$ is an indecomposable and descendable $W$-Galois extension. Thus by 
 Proposition~\ref{prop-G-galois-are-G-torsor}(b), it corresponds to the indecomposable $W$-torsor $W$ in $\FinSet_W$, see Example~\ref{ex-G-torsor-G}.

 Finally, to see that Question~\ref{que-Balmer-calrson} admits a positive answer as claimed, we assume that $A$ is a
 compact and separable algebra in the stable module category, which has finite degree. By Proposition~\ref{prop-sep-cov-rank-one-case}, $A$ is a finite cover. It thus corresponds to a finite $W$-set $X$, which we can assume transitive, say $X=W/U$.
 We can write $U=V/P$. Since $W/U=(W)_{hU}$, the coset 
 $W/U$ corresponds under the equivalence~(\ref{eq-galois}) to the indecomposable finite cover $
 (A_P^G)^{hU}$. Furthermore,
 \[
 (A_P^G)^{hU}\simeq(\prod_{G/P}k)^{V/P}\simeq \prod_{G/V} k=A^G_V.
 \]
 This provides a positive answer to Question~\ref{que-Balmer-calrson}.
\end{proof}

\bibliography{reference}

% \bib, bibdiv, biblist are defined by the amsrefs package.
\begin{bibdiv}
\begin{biblist}

\bib{Algeltveit}{article}{
      author={Angeltveit, Vigleik},
       title={Topological {H}ochschild homology and cohomology of {$A_\infty$}
  ring spectra},
        date={2008},
        ISSN={1465-3060},
     journal={Geom. Topol.},
      volume={12},
      number={2},
       pages={987\ndash 1032},
         url={https://doi.org/10.2140/gt.2008.12.987},
      review={\MR{2403804}},
}

\bib{Balmer}{article}{
      author={Balmer, Paul},
       title={The spectrum of prime ideals in tensor triangulated categories},
        date={2005},
        ISSN={0075-4102},
     journal={J. Reine Angew. Math.},
      volume={588},
       pages={149\ndash 168},
         url={https://doi.org/10.1515/crll.2005.2005.588.149},
      review={\MR{2196732}},
}

\bib{Balmer2}{article}{
      author={Balmer, Paul},
       title={Supports and filtrations in algebraic geometry and modular
  representation theory},
        date={2007},
        ISSN={0002-9327},
     journal={Amer. J. Math.},
      volume={129},
      number={5},
       pages={1227\ndash 1250},
         url={https://doi.org/10.1353/ajm.2007.0030},
      review={\MR{2354319}},
}

\bib{Balmer2010}{article}{
      author={Balmer, Paul},
       title={Spectra, spectra, spectra---tensor triangular spectra versus
  {Z}ariski spectra of endomorphism rings},
        date={2010},
        ISSN={1472-2747},
     journal={Algebr. Geom. Topol.},
      volume={10},
      number={3},
       pages={1521\ndash 1563},
         url={https://doi.org/10.2140/agt.2010.10.1521},
      review={\MR{2661535}},
}

\bib{Balmer2011}{article}{
      author={Balmer, Paul},
       title={Separability and triangulated categories},
        date={2011},
        ISSN={0001-8708},
     journal={Adv. Math.},
      volume={226},
      number={5},
       pages={4352\ndash 4372},
         url={https://doi.org/10.1016/j.aim.2010.12.003},
      review={\MR{2770453}},
}

\bib{Balmer2014}{article}{
      author={Balmer, Paul},
       title={Splitting tower and degree of tt-rings},
        date={2014},
        ISSN={1937-0652},
     journal={Algebra Number Theory},
      volume={8},
      number={3},
       pages={767\ndash 779},
         url={https://doi.org/10.2140/ant.2014.8.767},
      review={\MR{3218809}},
}

\bib{Balmer2015}{article}{
      author={Balmer, Paul},
       title={Stacks of group representations},
        date={2015},
        ISSN={1435-9855},
     journal={J. Eur. Math. Soc. (JEMS)},
      volume={17},
      number={1},
       pages={189\ndash 228},
         url={https://doi.org/10.4171/JEMS/501},
      review={\MR{3312406}},
}

\bib{Balmer2016-etale}{article}{
      author={Balmer, Paul},
       title={The derived category of an \'{e}tale extension and the separable
  {N}eeman-{T}homason theorem},
        date={2016},
        ISSN={1474-7480},
     journal={J. Inst. Math. Jussieu},
      volume={15},
      number={3},
       pages={613\ndash 623},
         url={https://doi.org/10.1017/S1474748014000449},
      review={\MR{3505660}},
}

\bib{Balmersurj}{article}{
      author={Balmer, Paul},
       title={On the surjectivity of the map of spectra associated to a
  tensor-triangulated functor},
        date={2018},
        ISSN={0024-6093},
     journal={Bull. Lond. Math. Soc.},
      volume={50},
      number={3},
       pages={487\ndash 495},
         url={https://doi.org/10.1112/blms.12158},
      review={\MR{3829735}},
}

\bib{BC2018}{article}{
      author={Balmer, Paul},
      author={Carlson, Jon~F.},
       title={Separable commutative rings in the stable module category of
  cyclic groups},
        date={2018},
        ISSN={1386-923X},
     journal={Algebr. Represent. Theory},
      volume={21},
      number={2},
       pages={399\ndash 417},
         url={https://doi.org/10.1007/s10468-017-9719-7},
      review={\MR{3780773}},
}

\bib{BalmerDellAmbrogioSanders15}{article}{
      author={Balmer, Paul},
      author={Dell'Ambrogio, Ivo},
      author={Sanders, Beren},
       title={Restriction to finite-index subgroups as \'etale extensions in
  topology, {KK}-theory and geometry},
        date={2015},
        ISSN={1472-2747},
     journal={Algebr. Geom. Topol.},
      volume={15},
      number={5},
       pages={3025\ndash 3047},
         url={http://dx.doi.org.ep.fjernadgang.kb.dk/10.2140/agt.2015.15.3025},
      review={\MR{3426702}},
}

\bib{BCHNP}{article}{
      author={{Barthel}, Tobias},
      author={{Castellana}, Natalia},
      author={{Heard}, Drew},
      author={{Naumann}, Niko},
      author={{Pol}, Luca},
       title={{Quillen stratification in equivariant homotopy theory}},
        date={2023-01},
     journal={arXiv e-prints},
       pages={arXiv:2301.02212},
      eprint={2301.02212},
}

\bib{Chromatic}{article}{
      author={Barthel, Tobias},
      author={Schlank, Tomer~M.},
      author={Stapleton, Nathaniel},
       title={Chromatic homotopy theory is asymptotically algebraic},
        date={2020},
        ISSN={0020-9910},
     journal={Invent. Math.},
      volume={220},
      number={3},
       pages={737\ndash 845},
         url={https://doi.org/10.1007/s00222-019-00943-9},
      review={\MR{4094970}},
}

\bib{Benson}{book}{
      author={Benson, David},
       title={Representations and cohomology. {II}},
     edition={Second},
      series={Cambridge Studies in Advanced Mathematics},
   publisher={Cambridge University Press, Cambridge},
        date={1998},
      volume={31},
        ISBN={0-521-63652-3},
        note={Cohomology of groups and modules},
      review={\MR{1634407}},
}

\bib{BCR1997}{article}{
      author={Benson, David},
      author={Carlson, Jon~F.},
      author={Rickard, Jeremy},
       title={Thick subcategories of the stable module category},
        date={1997},
        ISSN={0016-2736},
     journal={Fund. Math.},
      volume={153},
      number={1},
       pages={59\ndash 80},
         url={https://doi.org/10.4064/fm-153-1-59-80},
      review={\MR{1450996}},
}

\bib{BSKS}{article}{
      author={Benson, David},
      author={Iyengar, Srikanth~B.},
      author={Krause, Henning},
       title={Module categories for group algebras over commutative rings},
        date={2013},
        ISSN={1865-2433},
     journal={J. K-Theory},
      volume={11},
      number={2},
       pages={297\ndash 329},
         url={https://doi.org/10.1017/is013001031jkt214},
        note={With an appendix by Greg Stevenson},
      review={\MR{3060999}},
}

\bib{BokstedtNeeman}{article}{
      author={B\"{o}kstedt, Marcel},
      author={Neeman, Amnon},
       title={Homotopy limits in triangulated categories},
        date={1993},
        ISSN={0010-437X},
     journal={Compositio Math.},
      volume={86},
      number={2},
       pages={209\ndash 234},
         url={http://www.numdam.org/item?id=CM_1993__86_2_209_0},
      review={\MR{1214458}},
}

\bib{Nullstellensatz}{article}{
      author={{Burklund}, Robert},
      author={{Schlank}, Tomer~M.},
      author={{Yuan}, Allen},
       title={{The Chromatic Nullstellensatz}},
        date={2022-07},
     journal={arXiv e-prints},
       pages={arXiv:2207.09929},
      eprint={2207.09929},
}

\bib{Monads}{article}{
      author={Böhm, Gabriella},
      author={Brzeziński, Tomasz},
      author={Wisbauer, Robert},
       title={Monads and comonads on module categories},
        date={2009},
        ISSN={0021-8693},
     journal={Journal of Algebra},
      volume={322},
      number={5},
       pages={1719\ndash 1747},
  url={https://www.sciencedirect.com/science/article/pii/S0021869309003597},
}

\bib{DellAmbrogioBeren2018}{article}{
      author={Dell'Ambrogio, Ivo},
      author={Sanders, Beren},
       title={A note on triangulated monads and categories of module spectra},
        date={2018},
        ISSN={1631-073X},
     journal={C. R. Math. Acad. Sci. Paris},
      volume={356},
      number={8},
       pages={839\ndash 842},
         url={https://doi.org/10.1016/j.crma.2018.06.007},
      review={\MR{3851535}},
}

\bib{EisenbudHarris}{book}{
      author={Eisenbud, David},
      author={Harris, Joe},
       title={The geometry of schemes},
      series={Graduate Texts in Mathematics},
   publisher={Springer-Verlag, New York},
        date={2000},
      volume={197},
        ISBN={0-387-98638-3; 0-387-98637-5},
      review={\MR{1730819}},
}

\bib{EKMM}{book}{
      author={Elmendorf, Anthony~D.},
      author={Kriz, Igor},
      author={Mandell, Michael~A.},
      author={May, J.~Peter},
       title={Rings, modules, and algebras in stable homotopy theory},
      series={Mathematical Surveys and Monographs},
   publisher={American Mathematical Society, Providence, RI},
        date={1997},
      volume={47},
        ISBN={0-8218-0638-6},
         url={https://doi.org/10.1090/surv/047},
        note={With an appendix by M. Cole},
      review={\MR{1417719}},
}

\bib{ergus2022hopf}{thesis}{
      author={Ergus, Aras},
       title={{H}opf algebras and {H}opf--{G}alois extensions in
  {$\infty$}-categories},
        type={Ph.D. Thesis},
 institution={EPFL},
        date={2022},
        note={EPFL PhD thesis, available at
  \url{https://infoscience.epfl.ch/record/295824}},
}

\bib{Ford}{book}{
      author={Ford, Timothy~J.},
       title={Separable algebras},
      series={Graduate Studies in Mathematics},
   publisher={American Mathematical Society, Providence, RI},
        date={2017},
      volume={183},
        ISBN={978-1-4704-3770-1},
         url={https://doi.org/10.1090/gsm/183},
      review={\MR{3618889}},
}

\bib{Gorenstein}{book}{
      author={Gorenstein, Daniel},
       title={Finite groups},
     edition={Second},
   publisher={Chelsea Publishing Co., New York},
        date={1980},
        ISBN={0-8284-0301-5},
      review={\MR{569209}},
}

\bib{EGAIV}{article}{
      author={Grothendieck, Alexander},
       title={\'{E}l\'{e}ments de g\'{e}om\'{e}trie alg\'{e}brique. {IV}.
  \'{E}tude locale des sch\'{e}mas et des morphismes de sch\'{e}mas {IV}},
        date={1967},
        ISSN={0073-8301},
     journal={Inst. Hautes \'{E}tudes Sci. Publ. Math.},
      number={32},
       pages={361},
         url={http://www.numdam.org/item?id=PMIHES_1967__32__361_0},
      review={\MR{238860}},
}

\bib{HallRydh}{article}{
      author={Hall, Jack},
      author={Rydh, David},
       title={Perfect complexes on algebraic stacks},
        date={2017},
        ISSN={0010-437X},
     journal={Compos. Math.},
      volume={153},
      number={11},
       pages={2318\ndash 2367},
         url={https://doi.org/10.1112/S0010437X17007394},
      review={\MR{3705292}},
}

\bib{Ambidexterity}{book}{
      author={Hopkins, Michael},
      author={Lurie, Jacob},
       title={{Ambidexterity in {$K(n)$}-Local Stable Homotopy }},
        date={2013},
        note={Avaliable at
  \url{https://www.math.ias.edu/~lurie/papers/Ambidexterity.pdf}},
}

\bib{Hovey2008}{article}{
      author={Hovey, Mark},
       title={Morava {$E$}-theory of filtered colimits},
        date={2008},
        ISSN={0002-9947,1088-6850},
     journal={Trans. Amer. Math. Soc.},
      volume={360},
      number={1},
       pages={369\ndash 382},
         url={https://doi.org/10.1090/S0002-9947-07-04298-5},
      review={\MR{2342007}},
}

\bib{Kadison}{article}{
      author={Kadison, Lars},
       title={Uniquely separable extensions},
        date={2020},
        ISSN={0723-0869},
     journal={Expo. Math.},
      volume={38},
      number={2},
       pages={217\ndash 231},
         url={https://doi.org/10.1016/j.exmath.2020.01.003},
      review={\MR{4132815}},
}

\bib{LandTamme}{article}{
      author={Land, Markus},
      author={Tamme, Georg},
       title={On the {$K$}-theory of pullbacks},
        date={2019},
        ISSN={0003-486X,1939-8980},
     journal={Ann. of Math. (2)},
      volume={190},
      number={3},
       pages={877\ndash 930},
         url={https://doi.org/10.4007/annals.2019.190.3.4},
      review={\MR{4024564}},
}

\bib{LurieEll}{unpublished}{
      author={Lurie, Jacob},
       title={Elliptic cohomology ii: orientations},
        note={Available from the author's webpage at
  \url{https://www.math.ias.edu/~lurie/papers/Elliptic-II.pdf}},
}

\bib{SAG}{unpublished}{
      author={Lurie, Jacob},
       title={Spectral algebraic geometry},
        note={Available from the author's webpage at
  \url{https://www.math.ias.edu/~lurie/papers/SAG-rootfile.pdf}},
}

\bib{HTT}{book}{
      author={Lurie, Jacob},
       title={Higher topos theory},
      series={Annals of Mathematics Studies},
   publisher={Princeton University Press, Princeton, NJ},
        date={2009},
      volume={170},
        ISBN={978-0-691-14049-0; 0-691-14049-9},
         url={https://doi.org/10.1515/9781400830558},
      review={\MR{2522659}},
}

\bib{HA}{book}{
      author={Lurie, Jacob},
       title={{Higher algebra}},
        date={2017},
        note={Avaliable from the author's webpage at
  \url{http://www.math.harvard.edu/~lurie/papers/HA.pdf}},
}

\bib{Mathew2015}{misc}{
      author={Mathew, Akhil},
       title={Torus actions on stable module categories, picard groups, and
  localizing subcategories},
        date={2015},
}

\bib{Mathew2016}{article}{
      author={Mathew, Akhil},
       title={The {G}alois group of a stable homotopy theory},
        date={2016},
        ISSN={0001-8708},
     journal={Adv. Math.},
      volume={291},
       pages={403\ndash 541},
         url={https://doi.org/10.1016/j.aim.2015.12.017},
      review={\MR{3459022}},
}

\bib{Mathew2018}{incollection}{
      author={Mathew, Akhil},
       title={Examples of descent up to nilpotence},
        date={2018},
   booktitle={Geometric and topological aspects of the representation theory of
  finite groups},
      series={Springer Proc. Math. Stat.},
      volume={242},
   publisher={Springer, Cham},
       pages={269\ndash 311},
         url={https://doi.org/10.1007/978-3-319-94033-5_11},
      review={\MR{3901164}},
}

\bib{MathewLennart}{article}{
      author={Mathew, Akhil},
      author={Meier, Lennart},
       title={Affineness and chromatic homotopy theory},
        date={2015},
        ISSN={1753-8416},
     journal={J. Topol.},
      volume={8},
      number={2},
       pages={476\ndash 528},
         url={https://doi.org/10.1112/jtopol/jtv005},
      review={\MR{3356769}},
}

\bib{Neeman2018}{article}{
      author={Neeman, Amnon},
       title={Separable monoids in {${\bf D}_{\bf qc}(X)$}},
        date={2018},
        ISSN={0075-4102},
     journal={J. Reine Angew. Math.},
      volume={738},
       pages={237\ndash 280},
         url={https://doi.org/10.1515/crelle-2015-0039},
      review={\MR{3794893}},
}

\bib{Bregje}{article}{
      author={Pauwels, Bregje},
       title={Quasi-{G}alois theory in symmetric monoidal categories},
        date={2017},
        ISSN={1937-0652},
     journal={Algebra Number Theory},
      volume={11},
      number={8},
       pages={1891\ndash 1920},
         url={https://doi.org/10.2140/ant.2017.11.1891},
      review={\MR{3720934}},
}

\bib{Quillen1971}{article}{
      author={Quillen, Daniel},
       title={The spectrum of an equivariant cohomology ring. {I}, {II}},
        date={1971},
        ISSN={0003-486X},
     journal={Ann. of Math. (2)},
      volume={94},
       pages={549\ndash 572; ibid. (2) 94 (1971), 573\ndash 602},
         url={https://doi.org/10.2307/1970770},
      review={\MR{298694}},
}

\bib{Maxime}{article}{
      author={{Ramzi}, Maxime},
       title={{Separability in homotopical algebra}},
        date={2023-05},
     journal={arXiv e-prints},
       pages={arXiv:2305.17236},
      eprint={2305.17236},
}

\bib{Ravanel}{book}{
      author={Ravenel, Douglas~C.},
       title={Nilpotence and periodicity in stable homotopy theory},
      series={Annals of Mathematics Studies},
   publisher={Princeton University Press, Princeton, NJ},
        date={1992},
      volume={128},
        ISBN={0-691-02572-X},
        note={Appendix C by Jeff Smith},
      review={\MR{1192553}},
}

\bib{Robalo}{article}{
      author={{Robalo}, Marco},
       title={{Noncommutative Motives I: A Universal Characterization of the
  Motivic Stable Homotopy Theory of Schemes}},
        date={2012-06},
     journal={arXiv e-prints},
       pages={arXiv:1206.3645},
      eprint={1206.3645},
}

\bib{Rognes2008}{article}{
      author={Rognes, John},
       title={Galois extensions of structured ring spectra. {S}tably dualizable
  groups},
        date={2008},
        ISSN={0065-9266},
     journal={Mem. Amer. Math. Soc.},
      volume={192},
      number={898},
       pages={viii+137},
         url={https://doi.org/10.1090/memo/0898},
      review={\MR{2387923}},
}

\bib{Sanders21pp}{article}{
      author={Sanders, Beren},
       title={{A characterization of finite \'etale morphisms in tensor
  triangular geometry}},
        date={2022},
     journal={{Épijournal de Géométrie Algébrique}},
      volume={{Volume 6}},
         url={https://epiga.episciences.org/10183},
}

\bib{stacks-project}{misc}{
      author={{Stacks project authors}, The},
       title={The stacks project},
         how={\url{https://stacks.math.columbia.edu}},
        date={2022},
}

\bib{Strickland99}{article}{
      author={Strickland, Neil~P.},
       title={Products on {${\rm MU}$}-modules},
        date={1999},
        ISSN={0002-9947,1088-6850},
     journal={Trans. Amer. Math. Soc.},
      volume={351},
      number={7},
       pages={2569\ndash 2606},
         url={https://doi.org/10.1090/S0002-9947-99-02436-8},
      review={\MR{1641115}},
}

\end{biblist}
\end{bibdiv}

\end{document}